\documentclass[UTF-8,final]{siamltex}
\usepackage{amsmath}
\usepackage{epsfig}
\usepackage{graphicx}
\usepackage{amssymb}
\usepackage{CJK}
\usepackage{color}
\usepackage{float}
\usepackage{hyperref}
\usepackage[noadjust]{cite}
\usepackage{tabularx}
\usepackage{booktabs}
\usepackage{enumitem}
\usepackage{subcaption}
\newcolumntype{C}{>{\centering\arraybackslash}X}

\numberwithin{equation}{section}
\newtheorem{algorithm}{Algorithm}[section]
\newtheorem{remark}{Remark}[section]

\renewcommand\abstractname{Abstract}

\renewcommand\figurename{Figure}
\renewcommand\tablename{Table}

\renewcommand\refname{References}

\allowdisplaybreaks[4]

\normalsize

\def\Div{{\mbox{\rm div\,}}}

\begin{document}
	
	
	\title{A time-nonlocal multiphysics finite element method with the Crank-Nicolson scheme for the poroelasticity model with secondary consolidation\footnote{Last update: \today}}
	
\author{
		Yanan He\thanks{School of Mathematics and Statistics, Henan University, Kaifeng 475004, PR China ({\tt Email:hyn639@163.com}).}
	\and
		Zhihao Ge\thanks{Corresponding author. School of Mathematics and Statistics, Henan University, Kaifeng 475004, PR China ({\tt Email:zhihaoge@henu.edu.cn}).
			The work of this author was supported by the National Natural Science Foundation of China(No. 12371393) and Natural Science Foundation of Henan(No. 242300421047).}
}
	
	\maketitle
	
	
	\setcounter{page}{1}
	
	
	
	\begin{abstract}
	This paper studies a time-nonlocal multiphysics finite element method with the Crank-Nicolson scheme for the poroelasticity model with secondary consolidation. For the case where the physical parameters $\lambda,\lambda^*$ and $c_0$ are all finite positive constants, by introducing two auxiliary variables---the fluid content $\eta$ and the generalized pressure $\xi$---the original strongly coupled poroelasticity model is reformulated into a generalized Stokes equation with time integral terms and a diffusion equation. The reformulated model not only reveals the underlying multiphysics processes in the original model, but also exhibits time-nonlocal characteristics. A time-nonlocal multiphysics finite element method is designed for the reformulated model: the spatial discretization employs the high-order Taylor-Hood mixed finite element method, and the temporal discretization adopts the Crank-Nicolson scheme. The time integral terms are approximated using the composite trapezoidal rule, and the integral terms $J_{\xi}^n$ and $J_{\eta}^n$ are introduced for real-time updates, which not only avoids repeated calculations and improves efficiency, but also maintains second-order temporal accuracy. The existence and uniqueness of weak solutions for the reformulated model are proved via energy estimate methods, the stability of the fully discrete time-nonlocal multiphysics finite element method is established, and optimal-order error estimates are derived using projection operator techniques. Finally, numerical examples verify the theoretical results and compare the long-time convergence of the Crank-Nicolson scheme and the backward Euler scheme.
	\end{abstract}
	
	
	
	\pagestyle{myheadings}
	\thispagestyle{plain}
	\markboth{YANAN HE, ZHIHAO GE}{ A TIME-NONLOCAL MFEM WITH CRANK-NICOLSON SCHEME FOR POROELASTICITY MODEL WITH SECONDARY CONSOLIDATION}
	

\section{Introduction}\label{sec-1}

Poroelasticity theory is an important theory for studying the coupling mechanism between solid skeleton deformation and fluid seepage in porous media. It has wide applications in geotechnical engineering, biomedical engineering, energy development, and environmental science.  


In geotechnical engineering, the secondary consolidation effect is one of the main causes of long-term settlement of soft soil foundations. The poroelasticity model with secondary consolidation can accurately predict the creep settlement of high-speed railway soft soil subgrades, long-span bridge soft foundations, and high-rise building foundations \cite{Sekhdaria2021,Jozefiak2016}. It is also used to evaluate the long-term stability of dams, tunnels, tailings ponds, and other engineering structures \cite{Gaspar2010,GeHe2024}.

In biomedical engineering, biological soft tissues (such as articular cartilage, intervertebral discs, etc.) have significant viscoelastic characteristics, and their mechanical responses show time-dependent and creep behavior. The poroelasticity model with secondary consolidation can better describe the deformation law of articular cartilage tissue under long-term loading, providing a theoretical tool for studying the pathogenesis of arthritis and intervertebral disc degeneration \cite{Mow1986,Fung1993,Yang1991,Bociu2016}.

In energy resources, the secondary consolidation effect has important influences in long-term oil and gas reservoir exploitation and carbon dioxide geological sequestration projects. The creep deformation of reservoir rocks during long-term production affects wellbore integrity and permeability evolution \cite{Pao2001,Coussy2004}, and directly affects the evaluation results of sequestration safety \cite{Rutqvist2012}.

In environmental science, rocks surrounding nuclear waste geological repositories undergo creep deformation under long-term thermal-hydro-mechanical coupling effects \cite{Nelenne2021,Bear2010}. The secondary consolidation effect also affects the long-term performance and remediation efficiency of soil remediation materials \cite{HeGe2022,Murad1996}.

In summary, the poroelasticity model with secondary consolidation has wide application value in many important fields. Therefore, conducting in-depth theoretical research on poroelastic models and constructing efficient and robust numerical methods have become a research hotspot in scientific computing.

In this paper, we consider the following poroelasticity model with secondary consolidation:
\begin{align}
	-\lambda^*\nabla(\Div\mathbf{u})_t-\Div\sigma(\mathbf{u})+\alpha\nabla p=\mathbf{f}, &\qquad (\mathbf{u},p)\in \Omega_T, \label{chapt4-1.1} \\
	(c_0p+\alpha \Div\mathbf{u})_t+\Div\mathbf{v}_f =\phi, 	&\qquad(\mathbf{u},p)\in \Omega_T, 	\label{chapt4-1.2}
\end{align}
where
\begin{align*}
	\sigma(\mathbf{u})=\mu\varepsilon(\mathbf{u})+\lambda \mathrm{tr}(\varepsilon(\mathbf{u}))\mathbf{I},\quad\varepsilon(\mathbf{u})=\frac12(\nabla\mathbf{u}+\nabla^T\mathbf{u}), \quad
	\mathbf{v}_{f}=-\frac{K}{\mu_{f}}\left(\nabla p-\rho_{f}\mathbf{g}\right).
\end{align*}
The corresponding initial and boundary conditions are:
\begin{align}
	\lambda^*(\Div\mathbf{u})_t\mathbf{n}+\sigma(\mathbf{u})\mathbf{n}-\alpha p\mathbf{n}=\mathbf{f}_1\qquad&\text{on~ }\partial\Omega_T,\label{chapt4-2.1}\\
	\mathbf{v}_f\cdot\mathbf{n}=-\frac{K}{\mu_f}(\nabla p-\rho_f\mathbf{g})\cdot\mathbf{n}=\phi_1\qquad&\text{on~}\partial\Omega_T,\label{chapt4-2.2}\\
	\mathbf{u}=\mathbf{u}_{0},\quad p=p_0\qquad&\text{in~ }\Omega\times\{t=0\}.\label{chapt4-2.3}
\end{align}	

We remark that \eqref{chapt4-1.1} is the momentum balance equations for the displacement of the medium and \eqref{chapt4-1.2} is the mass balance equation for the pressure distribution. 

The introduction of the secondary consolidation term transforms the momentum conservation equation \eqref{chapt4-1.1} from elliptic type to parabolic type, forming a parabolic-parabolic coupled system with the mass conservation equation \eqref{chapt4-1.2}, which significantly increases the difficulty of numerical solution. From a mathematical theory perspective, Showalter \cite{Showalter2000} found that the effect of $\lambda^*\geq 0$ on the momentum equation is similar to that of $c_0\ge0$ on the diffusion equation, and provided existence and uniqueness analysis. Bellassoued and Riahi \cite{Riahi2016} established local Carleman estimates related to the secondary consolidation. In terms of numerical methods, Vermeer and Verruijt \cite{Vermeer1981} proposed accuracy conditions for finite element analysis, Zienkiewicz and Shiomi \cite{Zienkiewicz1984} established the generalized poroelasticity theory framework, Murad and Loula \cite{Murad1994} established the theoretical framework for mixed finite element approximation, and Gaspar et al. \cite{Gaspar2003,Gaspar2010} introduced a stabilization method using staggered grid finite difference methods. Recently, Ge and He \cite{HeGe2022,GeHe2024} proposed a new multiphysics finite element method for the poroelasticity model with secondary consolidation.

Research on high-order time discretization methods for poroelastic models has made significant progress in recent years. Zienkiewicz et al. \cite{Zienkiewicz1999} and Lewis et al. \cite{Lewis1998} systematically expounded the time integration methods and time stepping schemes for poroelastic models. On this basis, researchers have developed various high-order time discretization methods: Bause et al. \cite{Bause2017} achieved high-order accuracy in space-time finite elements through variational time discretization methods, Kunwar et al. \cite{Kunwar2020} proposed a second-order time discretization scheme for coupled fluid-poroelastic systems, Puente et al. \cite{DeLaPuente2007} and Yang et al. \cite{Yang2006} developed high-order discontinuous Galerkin methods and approximate analytical discrete methods for wave field simulation, respectively, and Egger and Sabouri \cite{Egger2021} studied structure-preserving high-order approximation methods. 
Ge et al. \cite{GeHeLi2019} proposed a stabilized multiphysics finite element method with the Crank-Nicolson scheme for poroelastic models, providing important references for high-order discretization of time nonlocal problems.

Among the above high-order time discretization schemes, the Crank-Nicolson scheme is widely used due to its good stability, conservation properties, and second-order accuracy. This scheme was proposed by Crank and Nicolson \cite{Crank1947} in 1947. Its core idea is to discretize at the midpoint $t_{n+\frac{1}{2}}$ of the time step, using central difference approximation for time derivatives and arithmetic averaging of adjacent time layers for spatial terms. 

The innovations of this paper are as follows:

	(1) Proposing a novel multiphysics reformulated model with time integral terms for the poroelasticity model with secondary consolidation.
	
	For the case where the physical parameters $\lambda$, $\lambda^*$ and $c_0$ are all finite positive constants, by introducing two auxiliary variables with clear physical meanings---fluid content $\eta$ and generalized pressure $\xi$---the original strongly coupled model is reformulated into a generalized Stokes equation with time integral terms and a diffusion equation. This reformulation reveals the multiphysics processes hidden in the original model: fluid transport behaves independently, and solid deformation is driven by the generalized pressure gradient. Particularly, the time integral terms in the reformulated model reflect the memory effect of materials, i.e., the current state depends not only on the current external forces but also on the historical evolution process. Then, for the reformulated model, we prove the existence and uniqueness of weak solutions.
	
	(2) Designing a time-nonlocal multiphysics finite element method with the Crank-Nicolson scheme for the reformulated model.
	
	For the discrete scheme, the spatial discretization employs the high-order Taylor-Hood mixed finite element method: displacement $\mathbf{u}$ uses $(r+1)$-th order Lagrange elements, $\xi$ and $\eta$ use $r$-th order Lagrange elements, and the temporal discretization uses the Crank-Nicolson scheme. Systematic theoretical analysis is conducted in this paper, including the stability analysis of the fully discrete scheme and the optimal-order error estimates.

The remainder of this paper is organized as follows.
In Section \ref{2026-3-22-1}, we introduce two auxiliary variables and reformulate the original strongly coupled model into a generalized Stokes equation with time integral terms and a diffusion equation. Then we prove the existence and uniqueness of weak solutions for the reformulated model.
In Section \ref{2026-3-22-2}, we develop a time-nonlocal multiphysics finite element method for the reformulated model. We employ the high-order Taylor-Hood mixed element for spatial discretization and the Crank-Nicolson scheme for temporal discretization, and we present stability analysis and optimal-order error estimates for the fully discrete scheme.
In Section \ref{2026-3-22-4}, we present a numerical test to verify the theoretical results. In Section \ref{2026-4-1-1}, we conclude with a summary.
\section{Model Reformulation and the Existence and Uniqueness of Weak Solutions}\label{2026-3-22-1}

\subsection{Model Reformulation}
The standard function space notation is adopted in this paper, the definition of which can be found in \cite{Evans2016}. 
We denote $(\cdot,\cdot)_{\Omega}$ and $(\cdot,\cdot)_{\partial\Omega}$ by the inner products of $L^2(\Omega)$ and $L^2(\partial\Omega)$, respectively, 
and denote $\|\cdot\|_0$ and $\|\cdot\|_1$ by the $L^2$-norm and $H^1$-norm, respectively.

Introduce auxiliary variables:
\begin{align}\label{chapt4-2.4}
	\eta:=c_0p+\alpha \Div \mathbf{u},\qquad
	\xi:=\alpha p-\lambda \Div \mathbf{u}-\lambda^*\Div \mathbf{u}_t.
\end{align}
It is easy to verify that
\begin{align}
	p&=\kappa_1\xi+\kappa_2\eta+\lambda^*\kappa_1\Div\mathbf{u}_t,\\
	\Div\mathbf{u}&=\kappa_1\eta-\kappa_3\xi-\lambda^*\kappa_3\Div\mathbf{u}_t,\label{chapt4-2.6}
\end{align}
where
\begin{align}\label{chapt4-4-25-2}
	\kappa_1=\frac{\alpha}{\alpha^2+\lambda c_0},\quad
	\kappa_2=\frac{\lambda}{\alpha^2+\lambda c_0},\quad
	\kappa_3=\frac{c_0}{\alpha^2+\lambda c_0}.
\end{align}

We only consider the case where the parameters $\lambda^*, c_0$ and $\lambda$ are all finite positive constants. Using \eqref{chapt4-2.4} and the method of variation of constants for ordinary differential equations, the \eqref{chapt4-1.1}-\eqref{chapt4-1.2} becomes the following \textbf{reformulated model}:
\begin{align}
	-\mu\Div\varepsilon(\mathbf{u})+\nabla\xi=\mathbf{f}  &\qquad\text{in~ }\Omega_T, \label{chap4-2.11} \\	
	\Div\mathbf{u}=\mathrm{e}^{-\frac{t}{\lambda^{*}\kappa_3}}\left(\int_{0}^{t}\frac{\kappa_1\eta-\kappa_3\xi}{\lambda^{*}\kappa_3}\cdot\mathrm{e}^{\frac{s}{\lambda^{*}\kappa_3}}\mathrm{~d}s+\Div\mathbf{u}_{0}\right)& \qquad\text{in~ }\Omega_T,\label{chap4-2.12} \\
	\eta_{t}-\frac{1}{\mu_{f}}\Div\left[K\left(\nabla(\kappa_{1}\xi+\kappa_{2}\eta+\lambda^{*}\kappa_{1}\Div\mathbf{u}_{t})-\rho_{f}\mathbf{g}\right)\right]  =\phi  & \qquad\text{in~ }\Omega_T.\label{chap4-2.13} 
\end{align}
The initial and boundary conditions \eqref{chapt4-2.1}-\eqref{chapt4-2.3} become
\begin{align}
	\lambda^*(\Div\mathbf{u})_t\mathbf{n}+\sigma(\mathbf{u})\mathbf{n}-\alpha(\kappa_1\xi+\kappa_2\eta+\lambda^*\kappa_1\Div\mathbf{u}_t)\mathbf{n}=\mathbf{f}_1 & \qquad\text{on~ }\partial\Omega_T,\label{chap4-2.21}  \\
	-\frac{K}{\mu_f}\left(\nabla(\kappa_1\xi+\kappa_2\eta+\lambda^*\kappa_1\Div\mathbf{u}_t)-\rho_f\mathbf{g}\right)\mathbf{n} =\phi_1 &\qquad\text{on~ }\partial\Omega_T, \label{chap4-2.22} \\
	\mathbf{u}(0)=\mathbf{u}_{0},\quad p(0)=p_0&\qquad\text{in~ }\Omega\times\{t=0\}. \label{chap4-2.23}
\end{align}

Here $\Omega\subset \mathbb{R}^d$ ($d=2,3$) denotes a bounded convex domain with boundary $\partial\Omega$. The variable $\mathbf{u}$ represents the solid displacement vector, and $p$ represents the fluid pressure. The term $\mathbf{f}$ is the body force (load per unit volume of the object), and $\phi$ is the source term. The tensor $\sigma(\mathbf{u})$ is the effective stress tensor, $\varepsilon(\mathbf{u})$ is the linear strain tensor, $\mathbf{I}$ is the $d\times d$ identity matrix, and $\hat{\sigma}(\mathbf{u})=\sigma(\mathbf{u})-\alpha p\mathbf{I}$ is the total stress tensor. The term $\mathbf{v}_f$ is Darcy velocity, $\mu_f$ denotes the fluid viscosity, $\rho_{f} \not\equiv 0$ represents the fluid density (mass per unit volume of fluid), $\mathbf{g}$ is the gravitational acceleration, and the secondary consolidation coefficient satisfies $\lambda^*\geq0$.
The permeability tensor $K=K(x)$ is assumed to be symmetric and uniformly positive definite, that is, for any $x\in\Omega$ and $\zeta\in\mathbb{R}^d$, there exist positive constants $K_1$ and $K_2$ such that $K_1|\zeta|^2\leq K(x)\zeta\cdot\zeta\leq K_2|\zeta|^2$ holds almost everywhere.
The Lam\'e constants $\lambda$ and $\mu$ represent the bulk elastic modulus and shear modulus, respectively. The coefficient $\alpha>0$ is the Biot-Willis constant, which characterizes the coupling effect between pressure and deformation, and measures the volume of fluid discharged from the solid skeleton due to expansion. The constrained specific storage coefficient $c_0\geq0$ is a comprehensive reflection of the porosity of the medium and the compressibility of both the fluid and solid. The terms $\mathbf{f}_1$ and $\phi_1$ in the boundary conditions are given known functions.

In addition, in some engineering literature, the Lam\'e constant $\mu$ is also called the shear modulus, denoted by $G$. The constants $\lambda$ and $\mu$ are calculated using Young's modulus $E$ and Poisson's ratio $\nu$:
\begin{align}\label{chap2-4-25-11}
	\lambda=\frac{E\nu}{(1+\nu)(1-2\nu)},\qquad\mu=G=\frac{E}{2(1+\nu)}.
\end{align}

\begin{remark}
	$\mathrm{(1)}$ When $\lambda^*= 0$ or $c_0=0$, the model \eqref{chapt4-1.1}-\eqref{chapt4-1.2} degenerates into a quasi-static poroelastic model, for which there are already abundant research results \cite{gezhihao2018,Phillips2007,Phillips2007_2}. Therefore, this paper assumes that the parameters $\lambda^*, c_0$ and $\lambda$ are all finite positive constants.
	
	$\mathrm{(2)}$ When $\lambda\rightarrow\infty$, we have $\kappa_1\rightarrow 0$ and $\kappa_3\rightarrow 0$. At this point,  \eqref{chap4-2.11}-\eqref{chap4-2.12} become
	\begin{align*}
		-\mu\Div\varepsilon(\mathbf{u})+\nabla\xi&=\mathbf{f},\\  
		\Div\mathbf{u}&=0,
	\end{align*}
	 and the research results can be found in the literature \cite{GeHeHe2020}.
\end{remark}
\subsection{Existence and Uniqueness of Weak Solutions}

The purpose of this section is to prove the existence and uniqueness of weak solutions for the reformulated model \eqref{chap4-2.11}-\eqref{chap4-2.13}. For simplicity and convenience, we assume that $\mathbf{f}$, $\mathbf{f}_1$, $\phi$ and $\phi_1$ are time-independent functions. All results can be extended to the time-dependent case.

Introduce the minimal rigid motion space:
\begin{align}\label{2026-3-13-1}
	\mathbf{RM}:=\{\mathbf{r}\mid\mathbf{r}=\mathbf{a}+\mathbf{b}\times \mathbf{x},~\mathbf{a},\mathbf{b},\mathbf{x}\in\mathbb{R}^d\}.
\end{align}

From the literature \cite{Brenner2008,Temam1984,Girault1986}, it is known that $\mathbf{RM}$ is the kernel of the strain operator $\varepsilon$, that is, $\mathbf{r}\in \mathbf{RM}$ if and only if $\varepsilon(\mathbf{r})=0$. Therefore, we have
\begin{align}
	\varepsilon(\mathbf{r})=0,\quad\Div\mathbf{r}=0\qquad\forall\mathbf{r}\in\mathbf{RM}.
\end{align}

Define the space $\mathbf{H}_\bot^1( \Omega)$, which is a subspace of $\mathbf{H}^1(\Omega)$ and orthogonal to $\mathbf{RM}$:
\begin{align}\label{2026-3-13-2}
	\mathbf{H}_\perp^1(\Omega):=\{\mathbf{v}\in\mathbf{H}^1(\Omega);\:(\mathbf{v},\mathbf{r})=0\:~~\forall\mathbf{r}\in\mathbf{RM}\}.
\end{align} 

\begin{definition}[Definition of Weak Solution]
	Suppose $\mathbf{u}_0\in \mathbf{H}^1(\Omega)$, $\mathbf{f} \in \mathbf{L}^2(\Omega)$, $\mathbf{f}_1 \in \mathbf{L}^2(\partial\Omega)$, $p_0\in L^2(\Omega)$, $\phi\in L^2(\Omega)$, and $\phi_1\in L^2(\partial\Omega)$. Assume that for any $\mathbf{v}\in\mathbf{RM}$, we have $(\mathbf{f},\mathbf{v})+\langle\mathbf{f}_1,\mathbf{v}\rangle=0$. Given $T>0$, for any $t\in [0, T]$, if $(\mathbf{u}, \xi, \eta)$ satisfies:
	\begin{align*}
		&\mathbf{u}\in L^{\infty}\left(0,T;\mathbf{H}_{\perp}^{1}(\Omega)\right), \quad \xi\in L^{2}\left(0,T;H^{1}(\Omega)\right),\\  
		&\eta\in L^{\infty}\left(0,T;H^{1}(\Omega)\right)\cap H^{1}\left(0,T;H^{-1}(\Omega)\right),
	\end{align*}
	and
	\begin{align}
		&\mu\left(\varepsilon(\mathbf{u}),\varepsilon(\mathbf{v})\right)-\left(\xi,\Div\mathbf{v}\right)=\left(\mathbf{f},\mathbf{v}\right)+\langle\mathbf{f}_1,\mathbf{v}\rangle \quad\forall\,\mathbf{v}\in\mathbf{H}^1(\Omega),\label{2.24}\\[3mm]
		&(\Div\mathbf{u},\varphi)+\frac{1}{\lambda^{*}}\cdot
		\mathrm{e}^{-\frac{t}{\lambda^{*}\kappa_3}}\int_{0}^{t}\mathrm{e}^{\frac{s}{\lambda^{*}\kappa_3}}(\xi(s),\varphi)\mathrm{~d}s\nonumber\\
		&=\frac{\alpha}{\lambda^{*}c_0}\cdot
		\mathrm{e}^{-\frac{t}{\lambda^{*}\kappa_3}}\int_{0}^{t}\mathrm{e}^{\frac{s}{\lambda^{*}\kappa_3}}(\eta(s),\varphi)\mathrm{~d}s
		+\mathrm{e}^{-\frac{t}{\lambda^{*}\kappa_3}}(\Div\mathbf{u}_{0},\varphi)\quad \forall\,\varphi\in L^2(\Omega),\label{2.25}\\[3mm]
		&(\eta_{t},\psi)+\frac{1}{\mu_{f}}\left(K\left(\nabla(\kappa_{1}\xi+\kappa_{2}\eta+\lambda^{*}\kappa_{1}\Div\mathbf{u}_{t})-\rho_{f}\mathbf{g}\right),\nabla\psi\right)  \nonumber\\
		&=(\phi,\psi)+\langle\phi_1,\psi\rangle \quad
		\forall\,\psi\in H^1(\Omega), \label{2.26}
	\end{align}
	then $(\mathbf{u}, \xi, \eta)$ is called a weak solution of equations \eqref{chap4-2.11}-\eqref{chap4-2.23}.
\end{definition}

\begin{theorem}[Energy Estimate]\label{lemma-2025-4-9}
	Suppose $(\mathbf{u}, \xi, \eta)$ is a weak solution of equations \eqref{2.24}-\eqref{2.26}. Define the energy functional
	\begin{align*}
		J(t):=\frac{\mu}{2}\|\varepsilon(\mathbf{u}(t))\|_0^{2}+\frac{\kappa_3}{2}\|\xi(t)\|_0^{2} 
		+\frac{\kappa_2}{2}\|\eta(t)\|_0^{2}.
	\end{align*}
	Then there exist positive constants $C_1$ and $C_2$ such that for any $t\in (0,T)$:
	\begin{align}\label{2.29}
		J(t)&+\frac{1}{\mu_f}\int_0^t(K\nabla p,\nabla p)\mathrm{~d}s\nonumber\\
		&\leq C_1J(0)
		+C_2\int_0^t(\|\mathbf{f}\|_0^2+\|\mathbf{f}_1\|_{L^2(\partial\Omega)}^2+\|\phi\|_0^2+\|\phi_1\|_{L^2(\partial\Omega)}^2+\|\rho_f\mathbf{g}\|_0^2)\mathrm{~d}s.
	\end{align}
\end{theorem}

\begin{proof}
	In \eqref{2.24}, taking $\mathbf{v}=\mathbf{u}_t$, we obtain
	\begin{align}\label{2.30}
		\mu\left(\varepsilon(\mathbf{u}),\varepsilon(\mathbf{u}_t)\right)-\left(\xi,\Div\mathbf{u}_t\right)=\left(\mathbf{f},\mathbf{u}_t\right)+\langle\mathbf{f}_1,\mathbf{u}_t\rangle.
	\end{align}
	Differentiating \eqref{2.25} twice with respect to time $t$, and taking $\varphi= \xi$, we have
	\begin{align}\label{2.31}
		\lambda^*\kappa_3(\Div\mathbf{u}_{tt},\xi)
		+(\Div\mathbf{u}_t,\xi)
		+\kappa_3(\xi_t,\xi)
		-\kappa_1(\eta_t,\xi)=0.
	\end{align}
	In \eqref{2.26}, taking $\psi=p=\kappa_{1}\xi+ \kappa_{2}\eta +\lambda^*\kappa_1\Div\mathbf{u}_t$, we obtain
	\begin{align}\label{2.32}
		(\eta_{t},\kappa_{1}\xi+\kappa_{2}\eta+\lambda^{*}\kappa_{1}\Div\mathbf{u}_{t})+\frac{1}{\mu_{f}}\left(K\left(\nabla p-\rho_{f}\mathbf{g}\right),\nabla p \right)  =(\phi,p)+\langle\phi_1,p \rangle. 
	\end{align}
	Adding equations \eqref{2.30}-\eqref{2.32} together, we get
	\begin{align}\label{2026-3-9-1}
		&\mu\left(\varepsilon(\mathbf{u}),\varepsilon(\mathbf{u}_t)\right)
		+\kappa_3(\xi_t,\xi)
		+\kappa_{2}(\eta_{t},\eta)
		+\frac{1}{\mu_{f}}\left(K\left(\nabla p-\rho_{f}\mathbf{g}\right),\nabla p \right)
		\nonumber\\
		&+\lambda^*\kappa_3(\Div\mathbf{u}_{tt},\xi)
		+(\eta_{t},\lambda^{*}\kappa_{1}\Div\mathbf{u}_{t})
		=\left(\mathbf{f},\mathbf{u}_t\right)
		+\langle\mathbf{f}_1,\mathbf{u}_t\rangle
		+(\phi,p)+\langle\phi_1,p \rangle.
	\end{align}
	Integrating \eqref{2026-3-9-1} over the interval $(0, t)$ with respect to time $t$, and using the identities
	\begin{align*}
		\int_{0}^{t}(\eta_t,\eta)\mathrm{~d}s
		=\frac{1}{2}\|\eta(t)\|_0^2-\frac{1}{2}\|\eta(0)\|_0^2,
	\end{align*}
	and
	\begin{align*}
		\int_{0}^{t} \left( \eta_{t},\lambda^*\kappa_1\Div \mathbf{u}_t\right)\mathrm{~d}s
		+\int_{0}^{t} \left( \eta,\lambda^*\kappa_1\Div \mathbf{u}_{tt}\right)\mathrm{~d}s
		=\left( \eta(t),\lambda^*\kappa_1\Div \mathbf{u}_t(t)\right)|_0^t,
	\end{align*}
	we obtain
	\begin{align}\label{2.34}
		&\frac{\mu}{2}\|\varepsilon(\mathbf{u}(t))\|_0^{2}+\frac{\kappa_3}{2}\|\xi(t)\|_0^{2} 
		+\frac{\kappa_2}{2}\|\eta(t)\|_0^{2}
		+\frac{1}{\mu_{f}}\int_{0}^{t}\left(K(\nabla p-\rho_{f}\mathbf{g}),\nabla p\right)\mathrm{~d}s
		\nonumber\\
		&\quad
		+\left( \eta,\lambda^*\kappa_1\Div \mathbf{u}_t\right)
		+\int_{0}^{t} \left( \lambda^*\Div \mathbf{u}_{tt}, \kappa_3\xi-\kappa_1\eta\right)\mathrm{~d}s
		\nonumber\\
		&=
		\int_{0}^{t}\left(\mathbf{f},\mathbf{u}_t\right)	\mathrm{~d}s
		+\int_{0}^{t}\langle\mathbf{f}_{1},\mathbf{u}_t\rangle	\mathrm{~d}s
		+\int_{0}^{t}\left(\phi,p\right)
		\mathrm{~d}s
		+\int_{0}^{t}\langle\phi_{1},p\rangle \mathrm{~d}s\nonumber\\
		&\quad+\frac{\mu}{2}\|\varepsilon(\mathbf{u}_0)\|_0^{2}
		+\frac{\kappa_3}{2}\|\xi_0\|_0^{2}
		+\frac{\kappa_2}{2}\|\eta_0\|_0^{2}
		+\left( \eta(0),\lambda^*\kappa_1\Div \mathbf{u}_t(0)\right).
	\end{align}	
	Estimate the cross terms in \eqref{2.34}. Using \eqref{chapt4-2.6}, we have
	\begin{align*}
		\int_0^t\left(\lambda^*\Div\mathbf{u}_{tt},\kappa_1\eta-\kappa_3\xi\right)\mathrm{~d}s
		&=\int_0^t\left(\lambda^*\Div\mathbf{u}_{tt},\Div\mathbf{u}\right)\mathrm{~d}s
		+\int_0^t\left(\lambda^*\kappa_3\Div\mathbf{u}_{tt},\lambda^*\Div\mathbf{u}_t\right)\mathrm{~d}s.
	\end{align*}	
	For the two terms on the right-hand side of the above equation, using integration by parts, Young's inequality and Korn's inequality, we get
	\begin{align*}
		&\int_0^t\left(\lambda^*\Div\mathbf{u}_{tt},\Div\mathbf{u}\right)\mathrm{~d}s
		=\left[\left(\lambda^*\Div\mathbf{u}_{t},\Div\mathbf{u}\right)\right]_0^t
		-\int_0^t\left(\lambda^*\Div\mathbf{u}_{t},\Div\mathbf{u}_t\right)\mathrm{~d}s\nonumber\\		
		&\leq \frac{\lambda^*}{2}(
		\|\Div\mathbf{u}_t(t)\|_0^2
		+\|\Div\mathbf{u}(t)\|_0^2
		+\|\Div\mathbf{u}_t(0)\|_0^2
		+\|\Div\mathbf{u}_0\|_0^2)
		-\lambda^*\int_0^t\|\Div\mathbf{u}_t\|_0^2\mathrm{~d}s\nonumber\\
		&\leq C\left(\|\varepsilon(\mathbf{u}_t(t))\|_0^2+J(t)+J(0)\right)
		-\lambda^*\int_0^t\|\Div\mathbf{u}_t\|_0^2\mathrm{~d}s,
	\end{align*}
	and
	\begin{align*}
		\int_0^t\left(\lambda^*\kappa_3\Div\mathbf{u}_{tt},\lambda^*\Div\mathbf{u}_t\right)\mathrm{~d}s
		&=\frac{(\lambda^*)^2\kappa_3}{2}\int_0^t\frac{\mathrm{d}}{\mathrm{d}t}\|\Div\mathbf{u}_t\|_0^2\mathrm{~d}s\nonumber\\
		&=\frac{(\lambda^*)^2\kappa_3}{2}\left(\|\Div\mathbf{u}_t(t)\|_0^2-\|\Div\mathbf{u}_t(0)\|_0^2\right)\nonumber\\
		&\leq C\left(\|\varepsilon(\mathbf{u}_t(t))\|_0^2+\|\varepsilon(\mathbf{u}_t(0))\|_0^2\right).
	\end{align*}		
	Estimate the source terms and boundary terms using the Cauchy-Schwarz inequality and Young's inequality:
	\begin{align*}
		&(\mathbf{f},\mathbf{u}_t)\leq \frac{1}{2\epsilon_1}\|\mathbf{f}\|_0^2+\frac{\epsilon_1}{2}\|\mathbf{u}_t\|_0^2,\qquad
		\langle\mathbf{f}_1,\mathbf{u}_t\rangle\leq \frac{1}{2\epsilon_2}\|\mathbf{f}_1\|_{L^2(\partial\Omega)}^2+\frac{\epsilon_2}{2}\|\mathbf{u}_t\|_{L^2(\partial\Omega)}^2,\\
		&(\phi,p)\leq \frac{1}{2\epsilon_3}\|\phi\|_0^2+\frac{\epsilon_3}{2}\|p\|_0^2,\qquad
		\langle\phi_1,p\rangle\leq \frac{1}{2\epsilon_4}\|\phi_1\|_{L^2(\partial\Omega)}^2+\frac{\epsilon_4}{2}\|p\|_{L^2(\partial\Omega)}^2,\\
		&\frac{1}{\mu_f}(K\rho_f\mathbf{g},\nabla p)\leq \frac{1}{4\mu_f}(K\nabla p,\nabla p)+\frac{1}{\mu_f}\|K\rho_f\mathbf{g}\|_0^2.
	\end{align*}	
	Substituting the above estimates into \eqref{2.34}, and choosing appropriate $\epsilon_i$ such that $\frac{1}{2\mu_f}\int_0^t(K\nabla p,\nabla p)\mathrm{~d}s$ remains on the left-hand side
	\begin{align*}
		J(t)&+\frac{1}{2\mu_f}\int_0^t(K\nabla p,\nabla p)\mathrm{~d}s
		\leq CJ(0)+C\int_0^tJ(s)\mathrm{~d}s
		\\
		& +C\int_0^t\left(\|\mathbf{f}\|_0^2+\|\mathbf{f}_1\|_{L^2(\partial\Omega)}^2+\|\phi\|_0^2+\|\phi_1\|_{L^2(\partial\Omega)}^2+\|\rho_f\mathbf{g}\|_0^2\right)\mathrm{~d}s.
	\end{align*}	
	Applying the integral form of Gronwall's inequality, we obtain the energy estimate \eqref{2.29}.
\end{proof}

\begin{theorem}\label{2026-3-28-1}
	Suppose $\mathbf{u}_0 \in\mathbf{H}^1(\Omega)$, $\mathbf{f}\in\mathbf{L}^2(\Omega)$, $\mathbf{f}_1\in\mathbf{L}^2(\partial\Omega)$, $p_0\in L^2(\Omega)$, $\phi\in L^2(\Omega)$, and $\phi_1\in L^2(\partial\Omega)$. Assume that for any $\mathbf{v}\in\mathbf{RM}$, we have $(\mathbf{f},\mathbf{v})+\langle\mathbf{f}_1,\mathbf{v}\rangle=0$. Then the solution of equations \eqref{chap4-2.11}-\eqref{chap4-2.23} exists and is unique.
\end{theorem}

\begin{proof}
	Since equations \eqref{chap4-2.11}-\eqref{chap4-2.23} form a linear system, the existence of weak solutions $(\mathbf{u},\xi,\eta)$ can be proved using the standard Galerkin method and compactness theory. The energy law in Theorem \ref{lemma-2025-4-9} provides the uniform estimates needed for the Galerkin method \cite{Temam1984}. Since the proof process is standard, the details are omitted here.
	
	For the uniqueness of weak solutions $(\mathbf{u},\xi,\eta)$, we only need to verify that when $\mathbf{f}=\mathbf{0}$, $\phi=0$, with zero initial and boundary conditions, the weak solution of equations \eqref{chap4-2.11}-\eqref{chap4-2.23} is:
	\[
	\mathbf{u}=\mathbf{0},~ \xi=0,~ \eta=0.
	\]
	According to Theorem \ref{lemma-2025-4-9}, we have:
	\begin{align*}
		\sqrt{\mu}\|\varepsilon(\mathbf{u})\|_{L^\infty(0,T;L^2(\Omega))}+\sqrt{\kappa_3}\|\xi\|_{L^\infty(0,T;L^2(\Omega))}
		+\sqrt{\kappa_2}\|\eta\|_{L^\infty(0,T;L^2(\Omega))}\leq 0.
	\end{align*}
	Using the integral form of Gronwall's inequality, letting $C_2=0$, we have $\varepsilon(\mathbf{u})=\mathbf{0}$, $\xi=0$, and $\eta=0$. Then, according to the initial condition $\mathbf{u}(0)=\mathbf{0}$ and $\mathbf{u}\in L^{\infty}\left(0,T;\mathbf{H}_{\perp}^{1}(\Omega)\right)$, we obtain $\mathbf{u}=\mathbf{0}$.
\end{proof}
\section{A time-nonlocal multiphysics finite element method with the Crank-Nicolson scheme for the reformulated model}\label{2026-3-22-2}

First, we present the finite element spaces corresponding to the variables. Let $\Omega\subset\mathbb{R}^n$ ($n=1,2,3$) be a bounded domain with Lipschitz boundary, and $\mathcal{T}_h$ be a quasi-uniform triangular ($n=2$) or tetrahedral ($n=3$) mesh defined on $\Omega$. The mesh elements $\mathcal{K}_j$, $j=1,\cdots,M$ are closed sets with non-empty connected interiors and Lipschitz continuous boundaries, satisfying $\bar{\Omega}=\bigcup_{j=1}^M\bar{\mathcal{K}_j}$. The maximum diameter of the mesh is $h=\max_{1\leq j\leq M}\{\operatorname{diam} \mathcal{K}_j\}$.
Divide the time interval $[0,T]$ into $N$ equal parts. The time mesh elements are $[t_{n-1},t_n]$, $n=1,2,\cdots, N$. The time step size is $\tau=\frac{T}{N}$, and $t_n=n\tau$.

For any positive integer $r$, define the piecewise continuous polynomial spaces $\mathbf{S}_h^{r}$ and $S_h^r$:
\begin{align*}
	\mathbf{S}_h^{r}&=\left\{\mathbf{v}_h\in\mathbf{C}^0(\overline{\Omega});~\mathbf{v}_h|_{\mathcal{K}_j}\in\mathbf{P}_{r}(\mathcal{K}_j)~\forall \mathcal{K}_j\in\mathcal{T}_h\right\},\\
	S_h^r&=\left\{v_h\in C^0(\overline{\Omega});~v_h|_{\mathcal{K}_j}\in P_r(\mathcal{K}_j)~\forall \mathcal{K}_j\in\mathcal{T}_h\right\},
\end{align*}
where $P_r(\mathcal{K}_j)$ is the space of polynomials of degree $r$ on the mesh element $\mathcal{K}_j$.

Define:
\begin{align}\label{2026-3-28-2}
	\mathbf{X}_h=\mathbf{S}_h^{r+1}\cap \mathbf{H}^1(\Omega), \quad 
	M_h = S_h^r\cap L_0^2(\Omega),
\end{align}
where $L_0^2(\Omega)=\{q\in L^2(\Omega); (q,1)=0
\}$.

Define the $L^2$ projection operator $\mathcal{P}: \mathbf{L}^2(\Omega) \rightarrow\mathbf{RM}$. For any $\mathbf{v}\in\mathbf{L}^2(\Omega)$, we have $\mathcal{P}\mathbf{v}\in\mathbf{RM}$ satisfying:
\[
(\mathcal{P}\mathbf{v},\mathbf{r})=(\mathbf{v},\mathbf{r})\quad\forall\mathbf{r}\in\mathbf{R}\mathbf{M}.
\]
Further define
\begin{align}\label{2026-3-28-3}
	\mathbf{V}_h:=(I-\mathcal{P})\mathbf{X}_h=\{\mathbf{v}_h\in\mathbf{X}_h;\mathrm{~}(\mathbf{v}_h,\mathbf{r})=0~~\forall\mathbf{r}\in\mathbf{R}\mathbf{M}\}.
\end{align}
Then we have $\mathbf{X}_h=\mathbf{V}_h\bigoplus\mathbf{RM}$, and the elements in $\mathbf{V}_h$ satisfy the zero-mean constraint.

Select $(\mathbf{V}_h,M_h,W_h)$ as the mixed finite element spaces for the variables $(\mathbf{u},\xi,\eta)$. The finite element approximation space $W_h$ for the variable $\eta$ can be chosen independently. Any piecewise polynomial space satisfying $M_h\subseteq W_h\subseteq L^2(\Omega)$ can be selected. For simplicity and convenience, we directly choose $W_h=M_h$.

The purpose of this section is to present a time-nonlocal fully discrete multiphysics finite element method based on the Crank-Nicolson scheme. First, we give the weak form of the reformulated model \eqref{chap4-2.11}-\eqref{chap4-2.13} at time $t_{n+\frac{1}{2}}$:
\begin{align}
	&\mu(\varepsilon(\mathbf{u}(t_{n+\frac{1}{2}})), \varepsilon(\mathbf{v})) - (\xi(t_{n+\frac{1}{2}}), \Div\mathbf{v}) = (\mathbf{f}(t_{n+\frac{1}{2}}), \mathbf{v}) + \langle \mathbf{f}_1(t_{n+\frac{1}{2}}), \mathbf{v} \rangle \quad \forall \mathbf{v} \in \mathbf{H}^1(\Omega),\label{2025-10-19-1}\\[3mm]
	&(\Div\mathbf{u}(t_{n+\frac{1}{2}}),\varphi)+\left(\frac{1}{\lambda^{*}}\cdot
	\mathrm{e}^{-\tfrac{t_{n+\frac{1}{2}}}{\lambda^{*}\kappa_3}}\int_{0}^{t_{n+\frac{1}{2}}}\xi(s)\cdot\mathrm{e}^{\frac{s}{\lambda^{*}\kappa_3}}\ \mathrm{~d}s,~\varphi\right)\nonumber\\[1.5mm]
	&=\left(\frac{\alpha}{\lambda^{*}c_0}\cdot
	\mathrm{e}^{-\tfrac{t_{n+\frac{1}{2}}}{\lambda^{*}\kappa_3}}\int_{0}^{t_{n+\frac{1}{2}}}\eta(s)\cdot\mathrm{e}^{\frac{s}{\lambda^{*}\kappa_3}}\ \mathrm{~d}s,~\varphi\right)
	+\mathrm{e}^{-\tfrac{t_{n+\frac{1}{2}}}{\lambda^{*}\kappa_3}}(\Div\mathbf{u}_{0},\varphi)\quad \forall\,\varphi\in L^2(\Omega),\label{2025-10-19-2}
	\\[3mm]
	&(\eta_{t}(t_{n+\frac{1}{2}}),\psi)+\frac{1}{\mu_{f}}\left(K\left(\nabla(\kappa_{1}\xi(t_{n+\frac{1}{2}})+\kappa_{2}\eta(t_{n+\frac{1}{2}})+\lambda^{*}\kappa_{1}\Div\mathbf{u}_{t}(t_{n+\frac{1}{2}}))-\rho_{f}\mathbf{g}\right),\nabla\psi\right)  \nonumber\\
	&=(\phi(t_{n+\frac{1}{2}}),\psi)+\langle\phi_1(t_{n+\frac{1}{2}}),\psi\rangle \quad
	\forall\,\psi\in H^1(\Omega). \label{2025-10-19-3}
\end{align}

Next, we present the fully discrete scheme for \eqref{2025-10-19-1}-\eqref{2025-10-19-3}. The time derivative is approximated using central differences. If there is no time derivative, we discretize directly at $t_{n+\frac{1}{2}}$, that is, take the average of the $n$ and $n+1$ time levels.

Since \eqref{2025-10-19-2} involves time integral terms $\int_{0}^{t_{n+\frac{1}{2}}}\xi(s)\cdot\mathrm{e}^{\frac{s}{\lambda^{*}\kappa_3}}\ \mathrm{~d}s$, we need to use the composite trapezoidal rule for numerical integration. Decompose the integral into a known part and a current step part:
\begin{align*}
	\int_{0}^{t_{n+\frac{1}{2}}}\xi(s)\cdot\mathrm{e}^{\frac{s}{\lambda^{*}\kappa_3}}\ \mathrm{~d}s
	=\int_{0}^{t_n}\xi(s)\cdot\mathrm{e}^{\frac{s}{\lambda^{*}\kappa_3}}\ \mathrm{~d}s
	+\int_{t_n}^{t_{n+\frac{1}{2}}}\xi(s)\cdot\mathrm{e}^{\frac{s}{\lambda^{*}\kappa_3}}\ \mathrm{~d}s.
\end{align*}
The known part can be expressed as an integral term
\begin{align*}
	J_{\xi}^n=\int_{0}^{t_n}\xi(s)\cdot\mathrm{e}^{\frac{s}{\lambda^{*}\kappa_3}}\ \mathrm{~d}s.
\end{align*}
Applying the trapezoidal rule to approximate the current step integral
\begin{align*}
	\int_{t_n}^{t_{n+\frac{1}{2}}}\xi(s)\cdot\mathrm{e}^{\frac{s}{\lambda^{*}\kappa_3}}\ \mathrm{~d}s
	&\approx \frac{\tau}{4}\left(
	\mathrm{e}^{\tfrac{t_{n+\frac{1}{2}}}{\lambda^{*}\kappa_3}}\cdot\xi_h^{n+\frac{1}{2}}
	+\mathrm{e}^{\tfrac{t_n}{\lambda^{*}\kappa_3}}\cdot\xi_h^{n}
	\right)\nonumber\\
	&= \frac{\tau}{4}\left(
	\mathrm{e}^{\tfrac{t_{n+\frac{1}{2}}}{\lambda^{*}\kappa_3}}\cdot\frac{\xi_h^{n+1}+\xi_h^n}{2}
	+\mathrm{e}^{\tfrac{t_n}{\lambda^{*}\kappa_3}}\cdot\xi_h^{n}
	\right).
\end{align*}
For the next step calculation, update the integral
\begin{align*}
	J_{\xi}^{n+1}=J_{\xi}^n
	+\frac{\tau}{2}\left(
	\mathrm{e}^{\tfrac{t_{n+1}}{\lambda^{*}\kappa_3}}\cdot\xi_h^{n+1}
	+\mathrm{e}^{\tfrac{t_n}{\lambda^{*}\kappa_3}}\cdot\xi_h^{n}
	\right),
\end{align*}
with initial condition
\begin{align*}
	J_\xi^0 = 0.
\end{align*}
For the other time integral term $\int_{0}^{t_{n+\frac{1}{2}}}\eta(s)\cdot\mathrm{e}^{\frac{s}{\lambda^{*}\kappa_3}}\ \mathrm{~d}s$, we adopt the same treatment method. The fully discrete scheme for equation \eqref{2025-10-19-2} is
\begin{align*}
	&\left( \frac{\Div \mathbf{u}_h^{n+1} + \Div \mathbf{u}_h^n}{2}, \varphi_h \right) 
	+ \left(\frac{1}{\lambda^*} \cdot \mathrm{e}^{-\tfrac{t_{n+\frac{1}{2}}}{\lambda^* \kappa_3}} \left( J_\xi^n 
	+ \frac{\tau}{4} \cdot  \mathrm{e}^{\tfrac{t_{n+\frac{1}{2}}}{\lambda^* \kappa_3}}\cdot\frac{\xi_h^{n+1} + \xi_h^n}{2}
	+\frac{\tau}{4}\cdot\mathrm{e}^{\tfrac{t_n}{\lambda^* \kappa_3}}\cdot\xi_h^n 
	\right) ,\varphi_h\right) \nonumber\\
	&= \left(\frac{\alpha}{\lambda^* c_0} \cdot \mathrm{e}^{-\tfrac{t_{n+\frac{1}{2}}}{\lambda^* \kappa_3}} \left( J_\eta^n 
	+ \frac{\tau}{4} \cdot  \mathrm{e}^{\tfrac{t_{n+\frac{1}{2}}}{\lambda^* \kappa_3}}\cdot\frac{\eta_h^{n+1} + \eta_h^n}{2}
	+\frac{\tau}{4}\cdot\mathrm{e}^{\tfrac{t_n}{\lambda^* \kappa_3}}\cdot\eta_h^n 
	\right) ,\varphi_h\right) \nonumber\\
	&\quad + \mathrm{e}^{-\tfrac{t_{n+\frac{1}{2}}}{\lambda^* \kappa_3}} (\Div \mathbf{u}_0, \varphi_h) \quad \forall \varphi_h \in M_h.
\end{align*}

Further simplifying the coefficients and combining with the fully discrete schemes for \eqref{2025-10-19-1} and \eqref{2025-10-19-3}, we obtain the core algorithm of this paper:
\begin{algorithm}\em A Time-Nonlocal Fully Discrete Multiphysics Finite Element Method Based on Crank-Nicolson Scheme: \label{algorithm-2025-10-15}
\begin{itemize}
\item[(i)]
Calculate initial values $(\mathbf{u}_{h}^{0}, \xi_{h}^{0}, \eta_{h}^{0})\in (\mathbf{V}_{h}, M_{h}, W_{h})$:
		\begin{align*}
			&\mathbf{u}_{h}^{0}=\mathcal{R}_{h}\mathbf{u}_{0},  \quad p_{h}^{0}=\mathcal{Q}_{h}p_{0}, \quad
			\eta_{h}^{0}= c_0 p_{h}^{0}+\alpha \mathcal{Q}_{h} \Div\mathbf{u}_{0} \\
			&\xi_h^0 =\alpha p_{h}^{0} -\lambda \mathcal{Q}_{h} \Div \mathbf{u}_{0} - \lambda^*  \mathcal{Q}_{h}(\Div\mathbf{u}_{t}(0) ),
		\end{align*}
		where $\mathcal{Q}_{h}$ and $\mathcal{R}_{h}$ are defined by \eqref{2026-3-13-3} and \eqref{2025-11-7-6}, and $\Div\mathbf{u}_{t}(0)$ denotes $\Div\mathbf{u}_{t}$ at $t=0$.

\item[(ii)] For $n=0,1,2, \cdots$, do the following two steps.
\end{itemize}

 Step 1: Solve for $(\mathbf{u}_{h}^{n+1},\xi_{h}^{n+1},\eta_{h}^{n+1})\in\mathbf{V}_{h}\times M_{h}\times W_{h}$: 
			\begin{align} 
				&\mu\left(\varepsilon(\mathbf{u}_h^{n+1} + \mathbf{u}_h^n), \varepsilon(\mathbf{v}_h)\right) -(\xi_h^{n+1} + \xi_h^n, \Div\mathbf{v}_h)\nonumber\\
				&= (\mathbf{f}^{n+1} + \mathbf{f}^n, \mathbf{v}_h) 
				+ \langle \mathbf{f}_1^{n+1} + \mathbf{f}_1^n, \mathbf{v}_h \rangle 
				\quad\forall \mathbf{v}_h \in \mathbf{V}_h,
				\label{2025-10-20-1}\\[5mm]
				&4\lambda^*\kappa_3( \Div \mathbf{u}_h^{n+1} + \Div \mathbf{u}_h^n, \varphi_h )
				+\kappa_3\tau (\xi_h^{n+1} + \xi_h^n, \varphi_h )\nonumber\\
				&=\kappa_1\tau (\eta_h^{n+1} + \eta_h^n, \varphi_h )
				+8\mathrm{e}^{-\tfrac{t_{n+\frac{1}{2}}}{\lambda^* \kappa_3}}(\kappa_1J_\eta^n-\kappa_3J_\xi^n,\varphi_h
				)\nonumber\\
				&\quad+2\tau \cdot \mathrm{e}^{-\tfrac{\tau}{2\lambda^* \kappa_3}}(\kappa_1\eta_h^n-\kappa_3\xi_h^n, \varphi_h) 
				+ 8\lambda^*\kappa_3\mathrm{e}^{-\tfrac{t_{n+\frac{1}{2}}}{\lambda^* \kappa_3}} (\Div \mathbf{u}_0, \varphi_h) \quad \forall \varphi_h \in M_h,\label{2025-10-20-2}\\[5mm]
				&\frac{1}{\mu_f} \left( K \left( \nabla \left( \kappa_1 \frac{\xi_h^{n+1} + \xi_h^n}{2} + \kappa_2 \frac{\eta_h^{n+1} + \eta_h^n}{2} + \lambda^* \kappa_1 \frac{\Div\mathbf{u}_h^{n+1} - \Div\mathbf{u}_h^n}{\tau} \right) - \rho_f \mathbf{g} \right), \nabla \psi_h \right)\nonumber\\
				&\quad+\left( \frac{\eta_h^{n+1} - \eta_h^n}{\tau}, \psi_h \right) 
				= \left(\frac{\phi^{n+1} + \phi^n}{2}, \psi_h\right) + \left\langle \frac{\phi_1^{n+1} + \phi_1^n}{2}, \psi_h \right\rangle \quad \forall \psi_h \in W_h.\label{2025-10-20-3}
			\end{align}

 Step 2: Update $p_h^{n+1}$, $J_\xi^{n+1}$ and $J_\eta^{n+1}$:
			\begin{align}
				&p_h^{n+1}=\kappa_{1}\xi_h^{n+1}+\kappa_{2}\eta_h^{n+1}+\lambda^{*}\kappa_{1}d_t\Div\mathbf{u}_h^{n+1},\label{2025-10-20-4}\\[2mm]
				&J_{\xi}^{n+1}=J_{\xi}^n
				+\frac{\tau}{2}\left(
				\mathrm{e}^{\tfrac{t_{n+1}}{\lambda^{*}\kappa_3}}\cdot\xi_h^{n+1}
				+\mathrm{e}^{\tfrac{t_n}{\lambda^{*}\kappa_3}}\cdot\xi_h^{n}
				\right),\label{2025-10-20-5}\\
				&J_{\eta}^{n+1}=J_{\eta}^n
				+\frac{\tau}{2}\left(
				\mathrm{e}^{\tfrac{t_{n+1}}{\lambda^{*}\kappa_3}}\cdot\eta_h^{n+1}
				+\mathrm{e}^{\tfrac{t_n}{\lambda^{*}\kappa_3}}\cdot\eta_h^{n}
				\right),\label{2025-10-20-6}\\
				&J_\xi^0 = 0,\quad J_\eta^0 = 0.\nonumber
			\end{align}
\end{algorithm}

\subsection{Stability Analysis of the Scheme}

The main purpose of this subsection is to demonstrate that Algorithm \ref{algorithm-2025-10-15} is stable through energy estimates. Energy estimates transform the stability analysis of numerical schemes into an analysis of energy functions, and use tools such as the discrete Gronwall's inequality to prove the boundedness of energy, thereby ensuring the boundedness of numerical solutions within finite time, that is, stability.

\begin{theorem}[Discrete Energy Estimate]\label{2025-11-3-1}
	Suppose $\left\{(\mathbf{u}_h^{n+1},\xi_h^{n+1},\eta_h^{n+1})\right\}_{n\geq0}$ is the numerical solution of Algorithm $\mathrm{\ref{algorithm-2025-10-15}}$ . Then for any $l\geq 0$, the following inequalities hold
	\begin{align*}
		\mu\lambda^*\kappa_3\|\varepsilon(\mathbf{u}_h^{l+1})\|_0^2	
		+\kappa_1\|\xi_h^{l+1}\|_0^2
		+\kappa_2\|\eta_h^{l+1}\|_0^2
		&\leq C_1N_0^2+C_2M_0^2,\\
		\frac{\tau}{\mu_f} \sum_{n=0}^{l}( K \nabla p_h^{n+\frac{1}{2}}, \nabla p_h^{n+\frac{1}{2}})
		&\leq C_1N_0^2+C_2M_0^2,
	\end{align*}
	where the constant $C$ depends on $\Omega, T$ and the parameters $\mu_f,K,\kappa_1,\kappa_2,\kappa_3$, and
	\begin{align*}
		N_0^2&=
		\|\varepsilon(\mathbf{u}_h^{1})\|_0^2 
		+\|\varepsilon(\mathbf{u}_h^0)\|_0^2
		+\|\Div \mathbf{u}_0\|_0^2
		+\|\xi_h^{1}\|_0^2 +\|\xi_h^0\|_0^2
		+\|\eta_h^1\|_0^2+\|\eta_h^0\|_0^2
		+\| \rho_f \mathbf{g}\|_0^2,\\
		M_0^2&=\|\mathbf{f}\|^2_{L^{2}(0,T,L^2(\Omega))}
		+\|\mathbf{f}_1\|^2_{L^{2}(0,T,L^2(\partial\Omega))}
		+\|\phi\|^2_{L^{2}(0,T,L^2(\Omega))}
		+\|\phi_1\|^2_{L^{2}(0,T,L^2(\partial\Omega))}.
	\end{align*}
\end{theorem}

\begin{proof}
	In \eqref{2025-10-20-1}, taking $\mathbf{v}_h=d_t(\mathbf{u}_h^{n+1} + \mathbf{u}_h^n)$, we obtain
	\begin{align}\label{2025-10-20-7}
		\mu\left(\varepsilon(\mathbf{u}_h^{n+1} + \mathbf{u}_h^n), d_t \varepsilon(\mathbf{u}_h^{n+1} + \mathbf{u}_h^n)\right) 
		-\left(\xi_h^{n+1} + \xi_h^n, d_t\Div(\mathbf{u}_h^{n+1} + \mathbf{u}_h^n)\right)=R_1,
	\end{align}
	where
	\begin{align*}
		R_1= \left(\mathbf{f}^{n+1} + \mathbf{f}^n, d_t(\mathbf{u}_h^{n+1} + \mathbf{u}_h^n)\right) 
		+ \left\langle \mathbf{f}_1^{n+1} + \mathbf{f}_1^n, d_t(\mathbf{u}_h^{n+1} + \mathbf{u}_h^n) \right\rangle.
	\end{align*}
	In \eqref{2025-10-20-2}, first applying the $d_t$ operator to both sides of the equation, then taking $\varphi=\xi_h^{n+1} + \xi_h^n$, we get
	\begin{align}\label{2025-10-20-8}
		&4\lambda^*\kappa_3\left( d_t\Div( \mathbf{u}_h^{n+1} + \mathbf{u}_h^n), \xi_h^{n+1} + \xi_h^n \right)
		+\kappa_3\tau \left(d_t(\xi_h^{n+1} + \xi_h^n), \xi_h^{n+1} + \xi_h^n \right)\nonumber\\
		&=\kappa_1\tau \left(d_t(\eta_h^{n+1} + \eta_h^n), \xi_h^{n+1} + \xi_h^n \right)	
		+2\tau \mathrm{e}^{-\tfrac{\tau}{2\lambda^* \kappa_3}}(\kappa_1d_t\eta_h^n-\kappa_3d_t\xi_h^n, \xi_h^{n+1} + \xi_h^n) 
		+R_2,
	\end{align}
where
\begin{align*}
	R_2&=\frac{8}{\tau} \cdot \mathrm{e}^{-\tfrac{t_{n+\frac{1}{2}}}{\lambda^* \kappa_3}}\left(\kappa_1J_\eta^n-\kappa_3J_\xi^n,\xi_h^{n+1} + \xi_h^n
	\right)
	-\frac{8}{\tau} \cdot \mathrm{e}^{-\tfrac{t_{n-\frac{1}{2}}}{\lambda^* \kappa_3}}\left(\kappa_1J_\eta^{n-1}-\kappa_3J_\xi^{n-1},\xi_h^{n+1} + \xi_h^n
	\right)\nonumber\\
	&\quad+ \frac{8\lambda^*\kappa_3}{\tau}\mathrm{e}^{-\tfrac{t_{n+\frac{1}{2}}}{\lambda^* \kappa_3}} (\Div \mathbf{u}_0, \xi_h^{n+1} + \xi_h^n)
	-\frac{8\lambda^*\kappa_3}{\tau}\mathrm{e}^{-\tfrac{t_{n-\frac{1}{2}}}{\lambda^* \kappa_3}} (\Div \mathbf{u}_0, \xi_h^{n+1} + \xi_h^n).
\end{align*}
In \eqref{2025-10-20-3}, letting $\psi=p_h^{n+\frac{1}{2}}=\kappa_1 \frac{\xi_h^{n+1} + \xi_h^n}{2} + \kappa_2 \frac{\eta_h^{n+1} + \eta_h^n}{2} + \lambda^* \kappa_1 \frac{\Div\mathbf{u}_h^{n+1} - \Div\mathbf{u}_h^n}{\tau}$, we obtain
\begin{align}\label{2025-10-20-9}
	&\kappa_2(d_t\eta_h^{n+1},\eta_h^{n+1} + \eta_h^n)
	+\kappa_1(d_t\eta_h^{n+1},\xi_h^{n+1} + \xi_h^n)
	+\frac{2}{\mu_f} ( K \nabla p_h^{n+\frac{1}{2}}, \nabla p_h^{n+\frac{1}{2}} )\nonumber\\
	&\quad+\frac{2\lambda^* \kappa_1}{\tau}(d_t\eta_h^{n+1},\Div(\mathbf{u}_h^{n+1} - \mathbf{u}_h^n))
	= R_{31},
\end{align}
To eliminate the cross terms, reducing the time level of equation \eqref{2025-10-20-3} to level $n$, and still letting $\psi=p_h^{n+\frac{1}{2}}$, we get
\begin{align}\label{2026-1-8-2}
	&\kappa_2(d_t\eta_h^{n},\eta_h^{n+1} + \eta_h^n)
	+\kappa_1(d_t\eta_h^{n},\xi_h^{n+1} + \xi_h^n)
	+\frac{2}{\mu_f} ( K \nabla p_h^{n-\frac{1}{2}}, \nabla p_h^{n+\frac{1}{2}} )\nonumber\\
	&\quad+\frac{2\lambda^* \kappa_1}{\tau}(d_t\eta_h^{n},\Div(\mathbf{u}_h^{n+1} - \mathbf{u}_h^n))
	= R_{32},
\end{align}
where
\begin{align*}
	R_{31}&=\frac{2}{\mu_f} ( K \rho_f \mathbf{g} , \nabla p_h^{n+\frac{1}{2}} )
	+(\phi^{n+1} + \phi^n, p_h^{n+\frac{1}{2}}) + \langle \phi_1^{n+1} + \phi_1^n, p_h^{n+\frac{1}{2}} \rangle,\\
	R_{32}&=\frac{2}{\mu_f} ( K \rho_f \mathbf{g} , \nabla p_h^{n+\frac{1}{2}} )
	+(\phi^{n} + \phi^{n-1}, p_h^{n+\frac{1}{2}}) + \langle \phi_1^{n} + \phi_1^{n-1}, p_h^{n+\frac{1}{2}} \rangle.
\end{align*}
Multiplying \eqref{2025-10-20-7} by $4\lambda^*\kappa_3$ on both sides, and adding \eqref{2025-10-20-8}-\eqref{2026-1-8-2}, we obtain
\begin{align}\label{2025-10-30-2}
	&4\lambda^*\kappa_3\left(\varepsilon(\mathbf{u}_h^{n+1} + \mathbf{u}_h^n), d_t \varepsilon(\mathbf{u}_h^{n+1} + \mathbf{u}_h^n)\right) 
	+\kappa_3\tau \left( d_t(\xi_h^{n+1} + \xi_h^n), \xi_h^{n+1}+ \xi_h^n \right)
	\nonumber\\
	&\quad+\kappa_2\left( d_t(\eta_h^{n+1}+\eta_h^{n}),   \eta_h^{n+1} + \eta_h^n\right)
	+2 \tau\mathrm{e}^{-\tfrac{\tau}{2\lambda^* \kappa_3}}(\kappa_3d_t\xi_h^n-\kappa_1d_t\eta_h^n, \xi_h^{n+1} + \xi_h^n)
	\nonumber\\
	&\quad
	+\frac{2}{\mu_f} \left( K (\nabla p_h^{n+\frac{1}{2}}+\nabla p_h^{n-\frac{1}{2}}), \nabla p_h^{n+\frac{1}{2}} \right)
	+\frac{2\lambda^* \kappa_1}{\tau}\left( d_t(\eta_h^{n+1}+\eta_h^{n}),   \Div(\mathbf{u}_h^{n+1} - \mathbf{u}_h^n) \right)
	\nonumber\\	
	&\quad+\kappa_1\left( d_t(\eta_h^{n+1}+\eta_h^{n}),   \xi_h^{n+1} + \xi_h^n\right)
	-\kappa_1\tau\left( d_t(\eta_h^{n+1}+\eta_h^{n}),   \xi_h^{n+1} + \xi_h^n\right)\nonumber\\
	&=4\lambda^*\kappa_3R_1+R_2+ R_{31}+ R_{32},
\end{align}
Using the identity
\begin{align}\label{2025-10-30-9}
	(d_t\xi_h^{n+1},\xi_h^{n+1})=\frac{1}{2}d_t\|\xi_h^{n+1}\|_0^2+\frac{\tau}{2}\|d_t\xi_h^{n+1}\|_0^2,
\end{align}
simplify \eqref{2025-10-30-2} to obtain
\begin{align}\label{2025-10-30-3}
	&2\mu\lambda^*\kappa_3d_t\|\varepsilon\left(\mathbf{u}_h^{n+1} + \mathbf{u}_h^n\right)\|_0^2
	+2\mu\lambda^*\kappa_3\tau\|d_t\varepsilon\left(\mathbf{u}_h^{n+1} + \mathbf{u}_h^n\right)\|_0^2
	+\frac{\kappa_3\tau^2}{2}\|d_t(\xi_h^{n+1} + \xi_h^n)\|_0^2\nonumber\\
	&\quad+\frac{\kappa_3\tau}{2}d_t\|\xi_h^{n+1}+ \xi_h^n\|_0^2
	+\frac{\kappa_2}{2}d_t\|\eta_h^{n+1}+ \eta_h^n\|_0^2
	+\frac{\kappa_2\tau}{2}\|d_t(\eta_h^{n+1} + \eta_h^n)\|_0^2\nonumber\\
	&\quad+\frac{2}{\mu_f} \left( K \nabla (p_h^{n+\frac{1}{2}}+p_h^{n-\frac{1}{2}}), \nabla p_h^{n+\frac{1}{2}} \right)
	+\frac{2\lambda^* \kappa_1}{\tau}\left( d_t(\eta_h^{n+1}+\eta_h^{n}),   \Div(\mathbf{u}_h^{n+1} - \mathbf{u}_h^n) \right)
	\nonumber\\
	&\quad+\kappa_1\left( d_t(\eta_h^{n+1}+\eta_h^{n}),   \xi_h^{n+1} + \xi_h^n\right)
	-\kappa_1\tau\left( d_t(\eta_h^{n+1}+\eta_h^{n}),   \xi_h^{n+1} + \xi_h^n\right)\nonumber\\
	&=2\tau\mathrm{e}^{-\tfrac{\tau}{2\lambda^* \kappa_3}}(\kappa_1d_t\eta_h^n-\kappa_3d_t\xi_h^n, \xi_h^{n+1} + \xi_h^n)
	+4\lambda^*\kappa_3R_1+R_2+ R_{31}+ R_{32}.
\end{align}
Applying the operator $\tau\sum_{n=0}^{l}$ to both sides of \eqref{2025-10-30-3}, and letting $\eta_h^{-1}=\xi_h^{-1}=0$, we get
\begin{align}\label{2025-10-30-4}
	&2\mu\lambda^*\kappa_3\|\varepsilon(\mathbf{u}_h^{l+1} + \mathbf{u}_h^l)\|_0^2
	+2\mu\lambda^*\kappa_3\tau^2\sum_{n=0}^{l}\|d_t\varepsilon(\mathbf{u}_h^{n+1} + \mathbf{u}_h^n)\|_0^2
	+\frac{2\tau}{\mu_f} \sum_{n=0}^{l}( K \nabla p_h^{n+\frac{1}{2}}, \nabla p_h^{n+\frac{1}{2}} )\nonumber\\
	&\quad+\frac{\kappa_3\tau}{2}\|\xi_h^{l+1}+ \xi_h^l\|_0^2
	+\frac{\kappa_3\tau^3}{2}\sum_{n=0}^{l}\|d_t(\xi_h^{n+1} + \xi_h^n)\|_0^2
	+\kappa_1\tau\sum_{n=0}^{1}( d_t(\eta_h^{n+1}+\eta_h^{n}),   \xi_h^{n+1} + \xi_h^n)\nonumber\\
	&\quad
	+\frac{\kappa_2}{2}\|\eta_h^{l+1}+ \eta_h^l\|_0^2
	+\frac{\kappa_2\tau^2}{2}\sum_{n=0}^{l}\|d_t(\eta_h^{n+1} + \eta_h^n)\|_0^2
	-\kappa_1\tau^2\sum_{n=0}^{l}( d_t(\eta_h^{n+1}+\eta_h^{n}),   \xi_h^{n+1} + \xi_h^n)\nonumber\\
	&=2\mu\lambda^*\kappa_3\|\varepsilon(\mathbf{u}_h^{1} + \mathbf{u}_h^0)\|_0^2
	+\frac{\kappa_3\tau}{2}\|\xi_h^{1}+ \xi_h^0\|_0^2
	+\frac{\kappa_2}{2}\|\eta_h^{1}+ \eta_h^0\|_0^2
	-\frac{2\tau}{\mu_f} \sum_{n=0}^{l}( K \nabla p_h^{n-\frac{1}{2}}, \nabla p_h^{n+\frac{1}{2}} )\nonumber\\
	&\quad-2\lambda^* \kappa_1\tau\sum_{n=0}^{l}\left( d_t(\eta_h^{n+1}+\eta_h^{n}), d_t\Div\mathbf{u}_h^{n+1} \right)
	+\tau\sum_{n=0}^{l}(4\lambda^*\kappa_3R_1+R_2+ R_{31}+R_{32})\nonumber\\
	&\quad+2\tau^2\mathrm{e}^{-\tfrac{\tau}{2\lambda^* \kappa_3}}\sum_{n=0}^{l}(\kappa_1d_t\eta_h^n-\kappa_3d_t\xi_h^n, \xi_h^{n+1} + \xi_h^n).
\end{align}
Next, we estimate the terms in \eqref{2025-10-30-4}. Since $K$ is symmetric positive definite, using the Cauchy-Schwarz inequality, we have
\begin{align}\label{2026-1-8-3}
	\frac{2\tau}{\mu_f} \sum_{n=0}^{l}( K \nabla p_h^{n-\frac{1}{2}}, \nabla p_h^{n+\frac{1}{2}} )
	&\leq
	\frac{2\tau}{\mu_f} \sum_{n=0}^{l}
	\sqrt{(K \nabla p_h^{n-\frac{1}{2}},\nabla p_h^{n-\frac{1}{2}})}
	\sqrt{(K \nabla p_h^{n+\frac{1}{2}},\nabla p_h^{n+\frac{1}{2}})}\nonumber\\
	&\leq
	\frac{\tau}{\mu_f} \sum_{n=0}^{l}
	(K \nabla p_h^{n-\frac{1}{2}},\nabla p_h^{n-\frac{1}{2}})
	+\frac{\tau}{\mu_f} \sum_{n=0}^{l}
	(K \nabla p_h^{n+\frac{1}{2}},\nabla p_h^{n+\frac{1}{2}})\nonumber\\
	&\leq\frac{2\tau}{\mu_f} \sum_{n=0}^{l}
	(K \nabla p_h^{n+\frac{1}{2}},\nabla p_h^{n+\frac{1}{2}}).
\end{align}
According to \eqref{2025-10-20-1} and \eqref{2025-10-20-4}, we have
\begin{align}\label{2026-3-22-5}
	&\kappa_1\tau\sum_{n=0}^{l}(d_t(\eta_h^{n+1}+\eta_h^{n}),\xi_h^{n+1}+\xi_h^{n})
	=\kappa_3\tau\sum_{n=0}^{l}
	(d_t(\xi_h^{n+1}+\xi_h^{n}),\xi_h^{n+1}+\xi_h^{n})
	\nonumber\\
	&\quad+
	\tau\sum_{n=0}^{l}(
	d_t\Div(\mathbf{u}_h^{n+1}+\mathbf{u}_h^{n}),\xi_h^{n+1}+\xi_h^{n})
	+\lambda^*\kappa_3\tau\sum_{n=0}^{l}
	(d_t^2\Div(\mathbf{u}_h^{n+1}+\mathbf{u}_h^{n}),\xi_h^{n+1}+\xi_h^{n})
	\nonumber\\
	&=\kappa_3\tau\sum_{n=0}^{l}
	(d_t(\xi_h^{n+1}+\xi_h^{n}),\xi_h^{n+1}+\xi_h^{n})
	+\frac{\mu}{2}\|\varepsilon(\mathbf{u}_{h}^{n+1}+\mathbf{u}_{h}^{n})\|_0^2
	+\frac{\mu\tau^2}{2}\sum_{n=0}^{l}\|d_t\varepsilon(\mathbf{u}_{h}^{n+1}+\mathbf{u}_{h}^{n})\|_0^2
	\nonumber\\	
	&\quad+\frac{\mu\lambda^*\kappa_3\tau^3}{2}\sum_{n=0}^{l}\|d_t^2\varepsilon(\mathbf{u}_{h}^{n+1}+\mathbf{u}_{h}^{n})\|_0^2
	-\frac{\mu\lambda^*\kappa_3\tau^2}{2}\sum_{n=0}^{l}\|d_t\varepsilon(\mathbf{u}_{h}^{n+1}+\mathbf{u}_{h}^{n})\|_0^2\nonumber\\
	&\quad
	-\tau\sum_{n=0}^{l}(\mathbf{f},d_t(\mathbf{u}_h^{n+1}+\mathbf{u}_h^{n}))
	-\tau\sum_{n=0}^{l}\langle \mathbf{f}_1,d_t(\mathbf{u}_h^{n+1}+\mathbf{u}_h^{n})\rangle.
\end{align}
Similarly, according to \eqref{2025-10-20-1} and \eqref{2025-10-20-4}, we have
\begin{align}\label{2025-10-30-7}
	&\lambda^*\kappa_1\tau\sum_{n=0}^{l}(d_t
	(\eta_h^{n+1}+\eta_h^{n}),d_t\Div\mathbf{u}_h^{n+1})
	=(\lambda^*)^2\kappa_3\|\Div\mathbf{u}_h^{n+1}\|_0^2
	+2\lambda^*\tau\sum_{n=0}^{l}\|d_t\Div\mathbf{u}_h^{n+1}\|_0^2\nonumber\\
	&
	+(\lambda^*)^2\kappa_3\tau^2\sum_{n=0}^{l}\|d_t^2\Div\mathbf{u}_h^{n+1}\|_0^2
	+2\mu\lambda^*\kappa_1\tau\sum_{n=0}^{l}	\|d_t\varepsilon(\mathbf{u}_{h}^{n+1})\|_0^2
	-(\lambda^*)^2\kappa_3\|\Div\mathbf{u}_h^{0}\|_0^2.
\end{align}
Using Young's inequality and $\mathrm{e}^{-\tfrac{\tau}{2\lambda^* \kappa_3}}\leq1$, we obtain
\begin{align}\label{2025-10-30-10}
	&\tau^2 \mathrm{e}^{-\tfrac{\tau}{2\lambda^* \kappa_3}}\sum_{n=0}^{l}(\kappa_1d_t\eta_h^n-\kappa_3d_t\xi_h^n, \xi_h^{n+1} + \xi_h^n)\nonumber\\
	&\leq\kappa_1\tau^2\sum_{n=0}^{l}(d_t(\eta_h^n+\eta_h^{n-1}), \xi_h^{n+1} + \xi_h^n)
	+\kappa_3\tau^2\sum_{n=0}^{l}(d_t(\xi_h^{n}+\xi_h^{n-1}), \xi_h^{n+1} + \xi_h^n)\nonumber\\
	&\leq\frac{\kappa_2\tau^2}{4}\sum_{n=0}^{l}\|d_t(\eta_h^n+\eta_h^{n-1})\|_0^2
	+\frac{\kappa_3\tau^2}{4}\sum_{n=0}^{l}\|d_t(\xi_h^{n}+\xi_h^{n-1})\|_0^2
	+C\tau\sum_{n=0}^{l}\|\xi_h^{n+1} + \xi_h^n\|_0^2.
\end{align}
Using Young's inequality and the trace theorem to handle the source terms and boundary terms, we have
\begin{align}\label{2025-10-30-11}
	4\lambda^*\kappa_3\tau\sum_{n=0}^{l}R_1
	&\leq
	\frac{\lambda^*\kappa_3}{8}\|\varepsilon(\mathbf{u}_h^{n+1} + \mathbf{u}_h^n)\|_0^2
	+C(\|\mathbf{f}^{n+1} + \mathbf{f}^n\|_0^2+
	\|\varepsilon(\mathbf{u}_h^{1} + \mathbf{u}_h^0)\|_0^2)\nonumber\\
	&\quad+C(\|\mathbf{f}_1^{n+1} + \mathbf{f}_1^n\|_{L^2(\partial\Omega)}^2+
	\|\varepsilon(\mathbf{u}_h^{1} + \mathbf{u}_h^0)\|_{L^2(\partial\Omega)}^2),\\
	\tau\sum_{n=0}^{l}(R_{31}+R_{32})
	&\leq 
	\frac{K_1\tau}{4\mu_f}\sum_{n=0}^{l}
	(\|\nabla p_h^{n+\frac{1}{2}}\|_0^2
	+\|p_h^{n+\frac{1}{2}}\|_0^2)
	+C\tau\sum_{n=0}^{l}
	(\|\phi^{n+1}+\phi^{n}\|_0^2)\nonumber\\
	&\quad+C\tau\sum_{n=0}^{l}
	\|\phi_1^{n+1}+\phi_1^{n}\|_{L^2(\partial\Omega)}^2
	+C\tau\sum_{n=0}^{l}
	\| \rho_f \mathbf{g}\|_0^2,
\end{align}
According to Poincar\'e's inequality, it is known that $\|p_h^{n+\frac{1}{2}}\|_0^2\leq \|\nabla p_h^{n+\frac{1}{2}}\|_0^2$.
According to the definition of $R_2$, we have
\begin{align}\label{2025-10-30-14}
	\tau\sum_{n=0}^{l}R_2
	&=8\kappa_1\sum_{n=0}^{l} \mathrm{e}^{-\tfrac{t_{n+\frac{1}{2}}}{\lambda^* \kappa_3}}\left(J_\eta^n,\xi_h^{n+1} + \xi_h^n
	\right)
	-8\kappa_3\sum_{n=0}^{l} \mathrm{e}^{-\tfrac{t_{n+\frac{1}{2}}}{\lambda^* \kappa_3}}\left(J_\xi^n,\xi_h^{n+1} + \xi_h^n
	\right)
	\nonumber\\
	&\quad-8\kappa_1\sum_{n=0}^{l} \mathrm{e}^{-\tfrac{t_{n-\frac{1}{2}}}{\lambda^* \kappa_3}}\left(J_\eta^{n-1},\xi_h^{n+1} + \xi_h^n
	\right)
	+8\kappa_3\sum_{n=0}^{l} \mathrm{e}^{-\tfrac{t_{n-\frac{1}{2}}}{\lambda^* \kappa_3}}\left(J_\xi^{n-1},\xi_h^{n+1} + \xi_h^n
	\right)\nonumber\\
	&\quad+8\lambda^*\kappa_3\sum_{n=0}^{l}\left(\mathrm{e}^{-\tfrac{t_{n+\frac{1}{2}}}{\lambda^*\kappa_3}}-\mathrm{e}^{-\tfrac{t_{n-\frac{1}{2}}}{\lambda^* \kappa_3}}\right) (\Div \mathbf{u}_0, \xi_h^{n+1} + \xi_h^n),
\end{align}
According to the recurrence formula of $J_{\eta}^{n}$, we have
\begin{align*}
	J_{\eta}^{n}-J_{\eta}^{n-1}=
	\frac{\tau}{2}\left(
	\mathrm{e}^{\tfrac{t_{n}}{\lambda^{*}\kappa_3}}\cdot\eta_h^{n}
	+\mathrm{e}^{\tfrac{t_{n-1}}{\lambda^{*}\kappa_3}}\cdot\eta_h^{n-1}
	\right),
\end{align*}
Given the initial value $J_{\eta}^{0}=0$, we have
\begin{align*}
	J_{\eta}^{n}&=\frac{\tau}{2}\left(
	\mathrm{e}^{\tfrac{t_{n}}{\lambda^{*}\kappa_3}}\cdot\eta_h^{n}
	+2\sum_{i=1}^{n-1}\mathrm{e}^{\tfrac{t_{i}}{\lambda^{*}\kappa_3}}\cdot\eta_h^{i}
	+\eta_h^{0}\right),\\
	J_{\eta}^{n-1}&=\frac{\tau}{2}\left(
	\mathrm{e}^{\tfrac{t_{n-1}}{\lambda^{*}\kappa_3}}\cdot\eta_h^{n-1}
	+2\sum_{i=1}^{n-1}\mathrm{e}^{\tfrac{t_{i-1}}{\lambda^{*}\kappa_3}}\cdot\eta_h^{i-1}
	+\eta_h^{0}\right),
\end{align*}
and
\begin{align}\label{2025-10-30-15}
	&8\kappa_1\sum_{n=0}^{l} \mathrm{e}^{-\tfrac{t_{n+\frac{1}{2}}}{\lambda^* \kappa_3}}\left(J_\eta^n,\xi_h^{n+1} + \xi_h^n
	\right)
	-8\kappa_1\sum_{n=0}^{l} \mathrm{e}^{-\tfrac{t_{n-\frac{1}{2}}}{\lambda^* \kappa_3}}\left(J_\eta^{n-1},\xi_h^{n+1} + \xi_h^n
	\right)\nonumber\\
	&=4\kappa_1\tau\sum_{n=0}^{l}
	\left(\mathrm{e}^{-\tfrac{\tau}{2\lambda^* \kappa_3}}\eta_h^{n}
	+2\sum_{i=1}^{n-1}\mathrm{e}^{-\tfrac{(n+\frac{1}{2}-i)\tau}{\lambda^{*}\kappa_3}}\cdot\eta_h^{i}
	+\mathrm{e}^{-\tfrac{t_{n+\frac{1}{2}}}{\lambda^* \kappa_3}}\eta_h^0	
	,\xi_h^{n+1} + \xi_h^n\right)\nonumber\\
	&\quad-4\kappa_1\tau\sum_{n=0}^{l}
	\left(\mathrm{e}^{-\tfrac{\tau}{2\lambda^* \kappa_3}}\eta_h^{n-1}
	+2\sum_{i=1}^{n-1}\mathrm{e}^{-\tfrac{(n+\frac{1}{2}-i)\tau}{\lambda^{*}\kappa_3}}\cdot\eta_h^{i-1}
	+\mathrm{e}^{-\tfrac{t_{n-\frac{1}{2}}}{\lambda^* \kappa_3}}\eta_h^0	
	,\xi_h^{n+1} + \xi_h^n\right)\nonumber\\
	&=4\kappa_1\tau^2\sum_{n=0}^{l}
	\left(\mathrm{e}^{-\tfrac{\tau}{2\lambda^* \kappa_3}}d_t\eta_h^{n}
	+2\sum_{i=1}^{n-1}\mathrm{e}^{-\tfrac{(n+\frac{1}{2}-i)\tau}{\lambda^{*}\kappa_3}}d_t\eta_h^{i}
	,\xi_h^{n+1} + \xi_h^n\right)\nonumber\\
	&\quad +4\kappa_1\tau\sum_{n=0}^{l}\left(
	(\mathrm{e}^{-\tfrac{t_{n+\frac{1}{2}}}{\lambda^* \kappa_3}}-\mathrm{e}^{-\tfrac{t_{n-\frac{1}{2}}}{\lambda^* \kappa_3}})\eta_h^0	
	,\xi_h^{n+1} + \xi_h^n\right),
\end{align}
Similarly treating the recurrence formula of $J_{\xi}^{n}$, and using \eqref{2025-10-30-14} and \eqref{2025-10-30-15}, we obtain
\begin{align}\label{2025-10-30-16}
	\tau\sum_{n=0}^{l}R_2
	&\leq12\tau^2\sum_{n=0}^{l}
	\left(\kappa_1d_t\eta_h^{n}-\kappa_3d_t\xi_h^{n}
	,\xi_h^{n+1} + \xi_h^n\right)
	+4\kappa_3\tau\sum_{n=0}^{l}\left(
	\xi_h^0,\xi_h^{n+1} + \xi_h^n\right)
	\nonumber\\
	&\quad
	+8\lambda^*\kappa_3\tau\sum_{n=0}^{l}(\Div \mathbf{u}_0, \xi_h^{n+1} + \xi_h^n)
	+4\kappa_1\tau\sum_{n=0}^{l}\left(
	\eta_h^0,\xi_h^{n+1} + \xi_h^n\right),
\end{align}
The treatment of the first term on the right-hand side of \eqref{2025-10-30-16} is similar to that of \eqref{2025-10-30-10}. For the other terms on the right-hand side of \eqref{2025-10-30-16},
we apply the Cauchy-Schwarz inequality to obtain
\begin{align}\label{2025-10-30-18}
&8\lambda^*\kappa_3\tau\sum_{n=0}^{l}(\Div \mathbf{u}_0, \xi_h^{n+1} + \xi_h^n)
+4\kappa_3\tau\sum_{n=0}^{l}(
\xi_h^0,\xi_h^{n+1} + \xi_h^n)
+4\kappa_1\tau\sum_{n=0}^{l}(
\eta_h^0,\xi_h^{n+1} + \xi_h^n)\nonumber\\
&\leq8\lambda^*\kappa_3\tau\sum_{n=0}^{l}\|\Div \mathbf{u}_0\|_0^2
+4\kappa_3\tau\sum_{n=0}^{l}\|\xi_h^0\|_0^2
+4\kappa_1\tau\sum_{n=0}^{l}\|\eta_h^0\|_0^2
+4\kappa_3\tau\sum_{n=0}^{l}
\|\xi_h^{n+1} + \xi_h^n\|_0^2\nonumber\\
&\leq8\lambda^*\kappa_3\|\Div \mathbf{u}_0\|_0^2
+4\kappa_3\|\xi_h^0\|_0^2
+4\kappa_1\|\eta_h^0\|_0^2
+4\kappa_3\tau\sum_{n=0}^{l}
\|\xi_h^{n+1} + \xi_h^n\|_0^2
\end{align}
Substituting \eqref{2026-1-8-3}-\eqref{2025-10-30-11} and \eqref{2025-10-30-16}-\eqref{2025-10-30-18} into \eqref{2025-10-30-4}, we obtain
\begin{align}\label{2025-10-30-19}
	&2\mu\lambda^*\kappa_3\|\varepsilon(\mathbf{u}_h^{l+1} + \mathbf{u}_h^l)\|_0^2	
	+\frac{\kappa_3}{2}\|\xi_h^{l+1}+ \xi_h^l\|_0^2
	+\frac{\kappa_2}{2}\|\eta_h^{l+1}+\eta_h^l\|_0^2
	+\frac{2\tau}{\mu_f} \sum_{n=0}^{l}( K \nabla p_h^{n+\frac{1}{2}}, \nabla p_h^{n+\frac{1}{2}} )\nonumber\\
	&\leq
	C\lambda^*\kappa_3\tau\sum_{n=0}^{l}
	\|\varepsilon(\mathbf{u}_h^{n+1} + \mathbf{u}_h^n)\|_0^2
	+C\kappa_3\tau\sum_{n=1}^{l}\|\xi_h^{n+1} + \xi_h^n\|_0^2
	+C\kappa_2\tau\sum_{n=0}^{l}\|\eta_h^{n+1}+\eta_h^n\|_0^2
	\nonumber\\
	&\quad+C(
	\|\varepsilon(\mathbf{u}_h^{1})\|_0^2 
	+\|\varepsilon(\mathbf{u}_h^0)\|_0^2
	+\|\Div \mathbf{u}_0\|_0^2
	+\|\xi_h^{1}\|_0^2 +\|\xi_h^0\|_0^2
	+\|\eta_h^0\|_0^2
	+\| \rho_f \mathbf{g}\|_0^2
	)\nonumber\\
	&\quad+C\tau\sum_{n=0}^{l}(\|\mathbf{f}^{n+1}\|_0^2
	+\|\mathbf{f}_1^{n+1} \|_{L^2(\partial\Omega)}^2
	+\|\phi^{n+1}\|_0^2
	+\|\phi^{n+1}_1\|_{L^2(\partial\Omega)}^2).
\end{align}
Using the discrete Gronwall's inequality and \eqref{2025-10-30-19}, we obtain
 \begin{align}\label{2025-10-30-20}
 	\mu\lambda^*\kappa_3\|\varepsilon(\mathbf{u}_h^{l+1} + \mathbf{u}_h^l)\|_0^2	
 	&+\kappa_3\|\xi_h^{l+1}+ \xi_h^l\|_0^2
 	+\kappa_2\|\eta_h^{l+1}+\eta_h^l\|_0^2\nonumber\\
 	&+\frac{\tau}{\mu_f} \sum_{n=0}^{l}( K \nabla p_h^{n+\frac{1}{2}}, \nabla p_h^{n+\frac{1}{2}} )
 	\leq C_1N_0^2+C_2M_0^2,
\end{align}
where
\begin{align*}
 	N_0^2&=
 	\|\varepsilon(\mathbf{u}_h^{1})\|_0^2 
 	+\|\varepsilon(\mathbf{u}_h^0)\|_0^2
 	+\|\Div \mathbf{u}_0\|_0^2
 	+\|\xi_h^{1}\|_0^2 +\|\xi_h^0\|_0^2
 	+\|\eta_h^1\|_0^2+\|\eta_h^0\|_0^2
 	+\| \rho_f \mathbf{g}\|_0^2,\\
 	M_0^2&=\|\mathbf{f}\|^2_{L^{2}(0,T,L^2(\Omega))}
 	+\|\mathbf{f}_1\|^2_{L^{2}(0,T,L^2(\partial\Omega))}
 	+\|\phi\|^2_{L^{2}(0,T,L^2(\Omega))}
 	+\|\phi_1\|^2_{L^{2}(0,T,L^2(\partial\Omega))}.
 \end{align*} 
Using the triangle inequality, mathematical induction, and the boundedness of initial values, we obtain Theorem $\mathrm{\ref{2025-11-3-1}}$.
\end{proof}
\subsection{Convergence Analysis of the Scheme}

The purpose of this subsection is to use projection operators to obtain optimal error estimates for Algorithm \ref{algorithm-2025-10-15}. First, we give the definitions of three projection operators.

For any $\varphi\in L^2(\Omega)$, define the $L^2$ projection operator $\mathcal{Q}_h: L^2(\Omega)\to S^r_h$:
\begin{align}\label{2026-3-13-3}
	(\varphi-\mathcal{Q}_h\varphi,\psi_h)=0,\quad\forall\ \psi_h\in S^r_h.
\end{align}
According to finite element theory \cite{Brenner2008}, for any $\varphi\in H^{k}(\Omega)$,
\begin{align}
	&\|\varphi-\mathcal{Q}_h\varphi\|_{0}
	+h\|\nabla(\varphi-\mathcal{Q}_h\varphi)\|_{0}
	\leq Ch^{\min\{r+1,~k\}}\|\varphi\|_{H^{k}(\Omega)},\label{2025-11-7-2}
\end{align}

For any $\varphi\in H^1(\Omega)$, define the elliptic projection operator $\mathcal{S}_h: H^1(\Omega)\to S^r_h$:
\begin{align}
	(K\nabla(\varphi-\mathcal{S}_h\varphi),\nabla\psi_h) &=0,\quad \forall \psi_h\in S^r_h,\label{2026-3-13-4}\\
	\int_{\Omega}(\varphi-\mathcal{S}_h\varphi)&=0.\nonumber
\end{align} 
According to \cite{Brenner2008}, for any $\varphi\in H^{k}(\Omega)$
\begin{align}
	&\|\varphi-\mathcal{S}_h\varphi\|_{0}+h\|\nabla(\varphi-\mathcal{S}_h\varphi)\|_{0}\leq Ch^{\min\{r+1,~k\}}\|\varphi\|_{H^{k}(\Omega)},\label{2025-11-7-4}
\end{align}

For any $\mathbf{v}\in\mathbf{H}_\perp^1(\Omega)$, define the elastic projection operator $\mathcal{R}_h:\mathbf{H}_\perp^1(\Omega)\to \mathbf{V}_h$:
\begin{align}\label{2025-11-7-6}
	\left(\varepsilon(\mathbf{v}-\mathcal{R}_h\mathbf{v}),\varepsilon(\mathbf{w}_h)\right)=0\quad \forall\ \mathbf{w}_h\in\mathbf{V}_h.
\end{align}
According to \cite{Brezzi1991}, for any $\mathbf{v}\in \mathbf{H}_\perp^1(\Omega)\cap H^{k}(\Omega)$,
\begin{align}\label{2025-11-7-7}
	\|\mathbf{v}-\mathcal{R}_h\mathbf{v}\|_{0}+h\|\nabla(\mathbf{v}-\mathcal{R}_h\mathbf{v})\|_{0}\leq Ch^{\min\{r+2,~k\}}\|\mathbf{v}\|_{H^{k}(\Omega)}.
\end{align}

To obtain error estimates, let $q=\Div\mathbf{u}$ and introduce the following notation:
\begin{align*}	
	&E_{\mathbf{u}}^{n+1}=\mathbf{u}(t_{n+1})-\mathbf{u}_{h}^{n+1},\quad E_{\xi}^{n+1}=\xi(t_{n+1})-\xi_{h}^{n+1},\\
	&E_{\eta}^{n+1}=\eta(t_{n+1})-\eta_{h}^{n+1},
	\quad
	E_{p}^{n+1}=p(t_{n+1})-p_{h}^{n+1},\\
	&E_{q}^{n+1}=q(t_{n+1})-q_{h}^{n+1}.
\end{align*}

Using the projection operators $\mathcal{Q}_h, \mathcal{S}_h,\mathcal{R}_h$, we further introduce:
\begin{align*}
	&E_\mathbf{u}^{n+1} =\mathbf{u}(t_{n+1})-\mathcal{R}_h(\mathbf{u}(t_{n+1}))+\mathcal{R}_h(\mathbf{u}(t_{n+1}))-\mathbf{u}_h^{n+1}:=\Lambda_\mathbf{u}^{n+1}+\Theta_\mathbf{u}^{n+1}, \\
	&E_q^{n+1} =q(t_{n+1})-\mathcal{S}_h(q(t_{n+1}))+\mathcal{S}_h(q(t_{n+1}))-q_h^{n+1}:=\hat{\Lambda}_q^{n+1}+\hat{\Theta}_q^{n+1}, \\
	&E_{\xi}^{n+1} =\xi(t_{n+1})-\mathcal{Q}_h(\xi(t_{n+1}))+\mathcal{Q}_h(\xi(t_{n+1}))-\xi_h^{n+1}:=\Lambda_\xi^{n+1}+\Theta_\xi^{n+1}, \\
	&E_{\xi}^{n+1} =\xi(t_{n+1})-\mathcal{S}_h(\xi(t_{n+1}))+\mathcal{S}_h(\xi(t_{n+1}))-\xi_h^{n+1}:=\hat{\Lambda}_\xi^{n+1}+\hat{\Theta}_\xi^{n+1}, \\
	&E_{\eta}^{n+1} =\eta(t_{n+1})-\mathcal{Q}_h(\eta(t_{n+1}))+\mathcal{Q}_h(\eta(t_{n+1}))-\eta_h^{n+1}:=\Lambda_\eta^{n+1}+\Theta_\eta^{n+1}, \\
	&E_{\eta}^{n+1} =\eta(t_{n+1})-\mathcal{S}_h(\eta(t_{n+1}))+\mathcal{S}_h(\eta(t_{n+1}))-\eta_h^{n+1}:=\hat{\Lambda}_\eta^{n+1}+\hat{\Theta}_\eta^{n+1},\\
	&E_{p}^{n+1} =p(t_{n+1})-\mathcal{Q}_h(p(t_{n+1}))+\mathcal{Q}_h(p(t_{n+1}))-p_h^{n+1}:=\Lambda_p^{n+1}+\Theta_p^{n+1},\\
	&E_{p}^{n+1} =p(t_{n+1})-\mathcal{S}_h(p(t_{n+1}))+\mathcal{S}_h(p(t_{n+1}))-p_h^{n+1}:=\hat{\Lambda}_p^{n+1}+\hat{\Theta}_p^{n+1}.
\end{align*}

\begin{theorem}\label{th-2025-11-7-1}
	Suppose $\left\{(\mathbf{u}_h^n,\xi_h^n,\eta_h^n)\right\}_{n\geq0}$ is the numerical solution of Algorithm \ref{algorithm-2025-10-15}. Then the following inequality holds:
	\begin{align}\label{2025-11-7-1}
		&\mathcal{J}_h^{l+1}
		+\tau\sum_{n=0}^{l}(
		d_t\Div(\Theta_\mathbf{u}^{n+1}+\Theta_\mathbf{u}^{n}),(\Theta_\xi^{n+1}+\Theta_\xi^{n}))
		+2\tau\sum_{n=0}^{l}
		(d_t(\Theta_\eta^{n+1}+\Theta_\eta^{n}),\lambda^*\kappa_1d_t\Div\Theta_{\mathbf{u}}^{n+1})
		\nonumber\\
		&\qquad+\lambda^*\kappa_3\tau\sum_{n=0}^{l}
		(d_t(d_t\Div(\Theta_\mathbf{u}^{n+1}+\Theta_\mathbf{u}^{n})),(\Theta_\xi^{n+1}+\Theta_\xi^{n}))
		+\frac{2\tau}{\mu_f}\sum_{n=0}^{l}
		(K\nabla\hat{Z}_p^{n+\frac{1}{2}},\nabla\hat{Z}_p^{n+\frac{1}{2}})
		\nonumber\\		
		&\qquad+\tau\sum_{n=0}^{l}\left[
		\mu\lambda^*\kappa_3\tau\|d_t\varepsilon(\Theta_\mathbf{u}^{n+1}+\Theta_\mathbf{u}^{n})\|_0^2
		+\frac{\kappa_3\tau}{2}\|d_t(\Theta_\xi^{n+1}+\Theta_\xi^n)\|_0^2	
		+\frac{\kappa_2\tau}{2}\|d_t(\Theta_\eta^{n+1}+\Theta_\eta^n)\|_0^2	
		\right]\nonumber\\
		&\leq\mathcal{J}_h^{0}+\sum_{i=1}^{12}\Phi_i,
	\end{align}	 
	where
	\begin{align*}
		&\mathcal{J}_h^{l+1}=2\mu\lambda^*\kappa_3\|\varepsilon(\Theta_\mathbf{u}^{l+1}+\Theta_\mathbf{u}^{l})\|_0^2
		+\frac{\kappa_3}{2}\|\Theta_\xi^{l+1}+\Theta_\xi^l\|_0^2
		+\frac{\kappa_{2}}{2}
		\|\Theta_\eta^{l+1}+\Theta_\eta^{l}\|_0^2,
	\end{align*}
	and
	\begin{align*}
		\Phi_1&=\tau\sum_{n=0}^{l}\left[
		-(d_tE_{\xi\_re}^n,\Theta_\xi^{n+1}+\Theta_\xi^{n})+(d_tE_{\eta\_re}^n,\Theta_\xi^{n+1}+\Theta_\xi^{n})\right]\\
		\Phi_2&=-\kappa_1\tau\sum_{n=0}^{l}(d_t(\Theta_\eta^{n+1}+\Theta_\eta^{n}),\Lambda_\xi^{n+1}-\hat{\Lambda}_\xi^{n+1}+\Lambda_\xi^{n}-\hat{\Lambda}_\xi^{n})\\
		\Phi_3&=-\kappa_2\tau\sum_{n=0}^{l}(d_t(\Theta_\eta^{n+1}+\Theta_\eta^{n}),\Lambda_\eta^{n+1}-\hat{\Lambda}_\eta^{n+1}+\Lambda_\eta^{n}-\hat{\Lambda}_\eta^{n})\\
		\Phi_4&=-2\lambda^*\kappa_1\tau\sum_{n=0}^{l}
		(d_t(\Theta_\eta^{n+1}+\Theta_\eta^{n}),d_t(\Div\Lambda_{\mathbf{u}}^{n+1}-\hat{\Lambda}_q^{n+1}))\\
		\Phi_5&=-\frac{2\lambda^*\kappa_1}{\mu_{f}}\tau\sum_{n=0}^{l}
		(K\nabla R_\mathbf{u}^{n+\frac{1}{2}}+K\nabla R_\mathbf{u}^{n-\frac{1}{2}},\nabla\hat{Z}_p^{n+\frac{1}{2}})\\	
		\Phi_6&=-4\lambda^*\kappa_3\tau\sum_{n=0}^{l}
		(d_t\Div (\Lambda_\mathbf{u}^{n+1}+ \Lambda_\mathbf{u}^{n}),\Theta_\xi^{n+1}+\Theta_\xi^{n})\\
		\Phi_7&=4\lambda^*\kappa_3\tau\sum_{n=0}^{l}
		\left(\Lambda_\xi^{n+1}+\Lambda_\xi^{n},d_t\Div(\Theta_\mathbf{u}^{n+1}+\Theta_\mathbf{u}^{n})\right)\\
		\Phi_8&=2\kappa_3\tau^2\cdot
		\mathrm{e}^{-\tfrac{\tau}{2\lambda^{*}\kappa_3}}
		\sum_{n=0}^{l}
		(d_t\Theta_\xi^{n+1}, \Theta_\xi^{n+1}+\Theta_\xi^{n})\\	
		\Phi_{9}&=2\kappa_1\tau^2\cdot
		\mathrm{e}^{-\tfrac{\tau}{2\lambda^{*}\kappa_3}}\sum_{n=0}^{l}
		(d_t\Theta_\eta^n, \Theta_\xi^{n+1}+\Theta_\xi^{n})\\	
		\Phi_{10}&=-2\tau\sum_{n=0}^{l}(R_\eta^{n+\frac{1}{2}}+R_\eta^{n-\frac{1}{2}},\hat{Z}_p^{n+\frac{1}{2}})\\
		\Phi_{11}&=-\frac{2\tau}{\mu_f}\sum_{n=0}^{l}
		(K\nabla\hat{Z}_p^{n-\frac{1}{2}},\nabla\hat{Z}_p^{n+\frac{1}{2}})\\
		\Phi_{12}&=\tau\sum_{n=0}^{l} (-R_4+ R_5).
	\end{align*}
\end{theorem}
\begin{proof}
Subtract the discrete scheme from the continuous variational equations to obtain the error equations. The initial value errors are
\begin{align*}
	E_{\mathbf{u}}^0=0,\quad E_{\xi}^0=0,\quad E_{\eta}^{0}=0.
\end{align*}
Subtracting \eqref{2025-10-20-1} from \eqref{2025-10-19-1}, we obtain
\begin{align}\label{2025-11-4-1}
	\mu\left(\varepsilon(E_\mathbf{u}^{n+1})+\varepsilon(E_\mathbf{u}^{n}),\varepsilon(\mathbf{v}_h)\right)-(E_\xi^{n+1}+E_\xi^{n},\Div\mathbf{v}_h)=0\qquad\forall\mathbf{v}_h\in \mathbf{V}_h.
\end{align}
Subtracting \eqref{2025-10-20-2} from \eqref{2025-10-19-2}, we obtain
\begin{align}\label{2025-11-4-2}
	4\lambda^*\kappa_3(\Div E_\mathbf{u}^{n+1}+\Div E_\mathbf{u}^{n},\varphi_h)
	+\left(E_{\xi r},\varphi_h\right)
	=\left(E_{\eta r},\varphi_h\right)
	\qquad	\forall\,\varphi_h\in M_h,
\end{align}
where
\begin{align*}
	E_{\xi r}&=8\kappa_3\cdot
	\mathrm{e}^{-\tfrac{t_{n+\frac{1}{2}}}{\lambda^{*}\kappa_3}}\int_{0}^{t_{n+\frac{1}{2}}}\xi(s)\cdot\mathrm{e}^{\frac{s}{\lambda^{*}\kappa_3}}\ \mathrm{~d}s\nonumber\\
	&\quad-8\kappa_3 \cdot \mathrm{e}^{-\tfrac{t_{n+\frac{1}{2}}}{\lambda^* \kappa_3}} \left( J_\xi^n 
	+ \frac{\tau}{4} \cdot  \mathrm{e}^{\tfrac{t_{n+\frac{1}{2}}}{\lambda^* \kappa_3}}\cdot\frac{\xi_h^{n+1} + \xi_h^n}{2}
	+\frac{\tau}{4}\cdot\mathrm{e}^{\tfrac{t_n}{\lambda^* \kappa_3}}\cdot\xi_h^n 
	\right),\\
	E_{\eta r}&=8\kappa_1\cdot
	\mathrm{e}^{-\tfrac{t_{n+\frac{1}{2}}}{\lambda^{*}\kappa_3}}\int_{0}^{t_{n+\frac{1}{2}}}\eta(s)\cdot\mathrm{e}^{\frac{s}{\lambda^{*}\kappa_3}}\ \mathrm{~d}s\nonumber\\
	&\quad-8\kappa_1 \cdot \mathrm{e}^{-\tfrac{t_{n+\frac{1}{2}}}{\lambda^* \kappa_3}} \left( J_\eta^n 
	+ \frac{\tau}{4} \cdot  \mathrm{e}^{\tfrac{t_{n+\frac{1}{2}}}{\lambda^* \kappa_3}}\cdot\frac{\eta_h^{n+1} + \eta_h^n}{2}
	+\frac{\tau}{4}\cdot\mathrm{e}^{\tfrac{t_n}{\lambda^* \kappa_3}}\cdot\eta_h^n 
	\right).
\end{align*}
%
%
We introduce some intermediate terms, and obtain
\begin{align}\label{2025-11-4-3}
	E_{\xi r}
	&=8\kappa_3\cdot
	\mathrm{e}^{-\tfrac{t_{n+\frac{1}{2}}}{\lambda^{*}\kappa_3}}\left[
	\int_{0}^{t_{n}}\xi(s)\cdot\mathrm{e}^{\frac{s}{\lambda^{*}\kappa_3}}\ \mathrm{~d}s
	-\frac{\tau}{2}\left(
	\mathrm{e}^{\tfrac{t_{n}}{\lambda^{*}\kappa_3}}\cdot\xi(t_n)
	+2\sum_{i=1}^{n-1}\mathrm{e}^{\tfrac{t_{i}}{\lambda^{*}\kappa_3}}\cdot\xi(t_i)
	+\xi(0)\right)
	\right]\nonumber\\
	&\quad+8\kappa_3\cdot
	\mathrm{e}^{-\tfrac{t_{n+\frac{1}{2}}}{\lambda^{*}\kappa_3}}\left[
	\int_{t_n}^{t_{n+\frac{1}{2}}}\xi(s)\cdot\mathrm{e}^{\frac{s}{\lambda^{*}\kappa_3}}\ \mathrm{~d}s
	- \frac{\tau}{4} \cdot  \mathrm{e}^{\tfrac{t_{n+\frac{1}{2}}}{\lambda^* \kappa_3}}\cdot\frac{\xi(t_{n+1}) + \xi(t_n)}{2}
	-\frac{\tau}{4}\cdot\mathrm{e}^{\tfrac{t_n}{\lambda^* \kappa_3}}\cdot\xi(t_n)
	\right]\nonumber\\
	&\quad+8\kappa_3 \cdot
	\mathrm{e}^{-\tfrac{t_{n+\frac{1}{2}}}{\lambda^{*}\kappa_3}}\left[
	\frac{\tau}{2}\left(
	\mathrm{e}^{\tfrac{t_{n}}{\lambda^{*}\kappa_3}}\cdot\xi(t_n)
	+2\sum_{i=1}^{n-1}\mathrm{e}^{\tfrac{t_{i}}{\lambda^{*}\kappa_3}}\cdot\xi(t_i)
	+\xi(0)\right)-J_{\xi}^n
	\right]\nonumber\\
	&\quad+8\kappa_3\cdot
	\mathrm{e}^{-\tfrac{t_{n+\frac{1}{2}}}{\lambda^{*}\kappa_3}}\left[
	\frac{\tau}{4} \cdot  \mathrm{e}^{\tfrac{t_{n+\frac{1}{2}}}{\lambda^* \kappa_3}}\cdot\frac{\xi(t_{n+1}) + \xi(t_n)}{2}
	+\frac{\tau}{4}\cdot\mathrm{e}^{\tfrac{t_n}{\lambda^* \kappa_3}}\cdot\xi(t_n)
	\right]\nonumber\\
	&\quad-8\kappa_3\cdot
	\mathrm{e}^{-\tfrac{t_{n+\frac{1}{2}}}{\lambda^{*}\kappa_3}}\left[
	\frac{\tau}{4} \cdot  \mathrm{e}^{\tfrac{t_{n+\frac{1}{2}}}{\lambda^* \kappa_3}}\cdot\frac{\xi_h^{n+1} + \xi_h^n}{2}
	+\frac{\tau}{4}\cdot\mathrm{e}^{\tfrac{t_n}{\lambda^* \kappa_3}}\cdot\xi_h^n 
	\right]\nonumber\\
	&=E_{\xi r_1}+E_{\xi r_2}+E_{\xi r_3}+E_{\xi r_4},
\end{align}
where $E_{\xi r_1}$ is the error of the standard composite trapezoidal rule on the interval $[0,t_n]$, given by
\begin{align}\label{2025-11-4-4}
	E_{\xi r_1}&=8\kappa_3\cdot
	\mathrm{e}^{-\tfrac{t_{n+\frac{1}{2}}}{\lambda^{*}\kappa_3}}\left[
	\int_{0}^{t_{n}}\xi(s)\cdot\mathrm{e}^{\frac{s}{\lambda^{*}\kappa_3}}\ \mathrm{~d}s
	-\frac{\tau}{2}\left(
	\mathrm{e}^{\tfrac{t_{n}}{\lambda^{*}\kappa_3}}\cdot\xi(t_n)
	+2\sum_{i=1}^{n-1}\mathrm{e}^{\tfrac{t_{i}}{\lambda^{*}\kappa_3}}\cdot\xi(t_i)
	+\xi(0)\right)
	\right]\nonumber\\
	&=8\kappa_3\cdot
	\mathrm{e}^{-\tfrac{t_{n+\frac{1}{2}}}{\lambda^{*}\kappa_3}}\left[-\frac{t_n}{12}\tau^2\cdot(\xi(s)\cdot\mathrm{e}^{\frac{s}{\lambda^{*}\kappa_3}})''\right]_{s=\zeta_1},
\end{align}
where $\zeta_1\in (0,t_n)$.

$E_{\xi r_2}$ is the error on the interval $[t_n,t_{n+\frac{1}{2}}]$, which consists of two parts: one part is the error of the trapezoidal rule itself, and the other part is the error caused by linear interpolation, which can be estimated using Taylor's formula.
\begin{align}\label{2025-11-4-5}
	E_{\xi r_2}&=8\kappa_3\cdot
	\mathrm{e}^{-\tfrac{t_{n+\frac{1}{2}}}{\lambda^{*}\kappa_3}}\left[
	\int_{t_n}^{t_{n+\frac{1}{2}}}\xi(s)\cdot\mathrm{e}^{\frac{s}{\lambda^{*}\kappa_3}}\ \mathrm{~d}s
	- \frac{\tau}{4} \cdot  \mathrm{e}^{\tfrac{t_{n+\frac{1}{2}}}{\lambda^* \kappa_3}}\cdot\frac{\xi(t_{n+1}) + \xi(t_n)}{2}
	-\frac{\tau}{4}\cdot\mathrm{e}^{\tfrac{t_n}{\lambda^* \kappa_3}}\cdot\xi(t_n)
	\right]\nonumber\\
	&=8\kappa_3\cdot
	\mathrm{e}^{-\tfrac{t_{n+\frac{1}{2}}}{\lambda^{*}\kappa_3}}\left[
	\int_{t_n}^{t_{n+\frac{1}{2}}}\xi(s)\cdot\mathrm{e}^{\frac{s}{\lambda^{*}\kappa_3}}\ \mathrm{~d}s
	- \frac{\tau}{4} \cdot  \mathrm{e}^{\tfrac{t_{n+\frac{1}{2}}}{\lambda^* \kappa_3}}\cdot\xi(t_{n+\frac{1}{2}})
	-\frac{\tau}{4}\cdot\mathrm{e}^{\tfrac{t_n}{\lambda^* \kappa_3}}\cdot\xi(t_n)
	\right]\nonumber\\
	&\quad+8\kappa_3\cdot
	\mathrm{e}^{-\tfrac{t_{n+\frac{1}{2}}}{\lambda^{*}\kappa_3}}\left[
	\frac{\tau}{4} \cdot  \mathrm{e}^{\tfrac{t_{n+\frac{1}{2}}}{\lambda^* \kappa_3}}\cdot\xi(t_{n+\frac{1}{2}})
	- \frac{\tau}{4} \cdot  \mathrm{e}^{\tfrac{t_{n+\frac{1}{2}}}{\lambda^* \kappa_3}}\cdot\frac{\xi(t_{n+1}) + \xi(t_n)}{2}
	\right]\nonumber\\
	&=8\kappa_3\cdot
	\mathrm{e}^{-\tfrac{t_{n+\frac{1}{2}}}{\lambda^{*}\kappa_3}}\left[
	\int_{t_n}^{t_{n+\frac{1}{2}}}\xi(s)\cdot\mathrm{e}^{\frac{s}{\lambda^{*}\kappa_3}}\ \mathrm{~d}s
	- \frac{\tau}{4} \cdot  \mathrm{e}^{\tfrac{t_{n+\frac{1}{2}}}{\lambda^* \kappa_3}}\cdot\xi(t_{n+\frac{1}{2}})
	-\frac{\tau}{4}\cdot\mathrm{e}^{\tfrac{t_n}{\lambda^* \kappa_3}}\cdot\xi(t_n)
	\right]\nonumber\\
	&\quad+2\kappa_3\tau\left[
	\xi(t_{n+\frac{1}{2}})
	- \frac{\xi(t_{n+1}) + \xi(t_n)}{2}
	\right]\nonumber\\
	&=-\frac{\kappa_3\tau^3}{12}\cdot
	\mathrm{e}^{-\tfrac{t_{n+\frac{1}{2}}}{\lambda^{*}\kappa_3}}
	(\xi(s)\cdot\mathrm{e}^{\frac{s}{\lambda^{*}\kappa_3}})''|_{s=\zeta_2}
	-\frac{\kappa_3\tau^3}{4}[(\xi(s))''|_{s=\zeta_3}
	+(\xi(s))''|_{s=\zeta_4}],
\end{align}
where $\zeta_2\in (t_n,t_{n+\frac{1}{2}})$, $\zeta_3\in(t_n,t_{n+\frac{1}{2}})$, $\zeta_4\in(t_{n+\frac{1}{2}},t_{n+1})$.

Next, estimate $E_{\xi r_3}$, we obtain
\begin{align}\label{2025-11-4-6}
	E_{\xi r_3}&=4\kappa_3\tau\cdot
	\mathrm{e}^{-\tfrac{t_{n+\frac{1}{2}}}{\lambda^{*}\kappa_3}}\left[
	\mathrm{e}^{\tfrac{t_{n}}{\lambda^{*}\kappa_3}}\cdot(\xi(t_n)-\xi_h^n)
	+2\sum_{i=1}^{n-1}\mathrm{e}^{\tfrac{t_{i}}{\lambda^{*}\kappa_3}}\cdot(\xi(t_i)-\xi_h^i)
	+\xi(0)-\xi_h^0
	\right]\nonumber\\
	&=4\kappa_3\tau\cdot
	\mathrm{e}^{-\tfrac{t_{n+\frac{1}{2}}}{\lambda^{*}\kappa_3}}\left[
	\mathrm{e}^{\tfrac{t_{n}}{\lambda^{*}\kappa_3}}\cdot E_\xi^n
	+2\sum_{i=1}^{n-1}\mathrm{e}^{\tfrac{t_{i}}{\lambda^{*}\kappa_3}}\cdot E_\xi^i
	+E_\xi^0\right].
\end{align}
Finally, estimate $E_{\xi r_4}$, we obtain
\begin{align}\label{2025-11-4-7}
	E_{\xi r_4}&=2\kappa_3\tau\cdot
	\mathrm{e}^{-\tfrac{t_{n+\frac{1}{2}}}{\lambda^{*}\kappa_3}}\left[
	 \mathrm{e}^{\tfrac{t_{n+\frac{1}{2}}}{\lambda^* \kappa_3}}\cdot\left(\frac{\xi(t_{n+1}) + \xi(t_n)}{2}
	-\frac{\xi_h^{n+1} + \xi_h^n}{2}\right)
	+\mathrm{e}^{\tfrac{t_n}{\lambda^* \kappa_3}}\cdot\left(\xi(t_n)-\xi_h^n\right)
	\right]\nonumber\\
	&=\kappa_3\tau
	\left(E_\xi^{n+1}+E_\xi^n
	\right)
	+2\kappa_3\tau\cdot
	\mathrm{e}^{-\tfrac{\tau}{2\lambda^{*}\kappa_3}}
	E_\xi^n.
\end{align}
Substituting \eqref{2025-11-4-4}-\eqref{2025-11-4-7} into \eqref{2025-11-4-3}, we get
\begin{align}\label{2025-11-4-8}
	E_{\xi r}&=\kappa_3\tau
	\left(E_\xi^{n+1}+E_\xi^n
	\right)
	+2\kappa_3\tau\cdot
	\mathrm{e}^{-\tfrac{\tau}{2\lambda^{*}\kappa_3}}\cdot
	E_\xi^n\nonumber\\
	&\quad
	+4\kappa_3\tau\cdot
	\mathrm{e}^{-\tfrac{t_{n+\frac{1}{2}}}{\lambda^{*}\kappa_3}}\left[
	\mathrm{e}^{\tfrac{t_{n}}{\lambda^{*}\kappa_3}}\cdot E_\xi^n
	+2\sum_{i=1}^{n-1}\mathrm{e}^{\tfrac{t_{i}}{\lambda^{*}\kappa_3}}\cdot E_\xi^i
	+E_\xi^0\right]+E_{\xi\_{re}}^n,
\end{align}
where
\begin{align}\label{2025-11-14-3}
	E_{\xi\_{re}}^n&=
	-\frac{2n\kappa_3\tau^3}{3}\cdot
	\mathrm{e}^{-\tfrac{t_{n+\frac{1}{2}}}{\lambda^{*}\kappa_3}}\left[(\xi(s)\cdot\mathrm{e}^{\frac{s}{\lambda^{*}\kappa_3}})''\right]_{s=\zeta_1}
	-\frac{\kappa_3\tau^3}{12}\cdot
	\mathrm{e}^{-\tfrac{t_{n+\frac{1}{2}}}{\lambda^{*}\kappa_3}}
	\left[(\xi(s)\cdot\mathrm{e}^{\frac{s}{\lambda^{*}\kappa_3}})''\right]_{s=\zeta_2}\nonumber\\
	&\quad
	-\frac{\kappa_3\tau^3}{4}[(\xi(s))''|_{s=\zeta_3}
	+(\xi(s))''|_{s=\zeta_4}],
\end{align}
and
\begin{align*}
	\zeta_1\in (0,t_{n}), 
	\quad \zeta_2\in (t_n,t_{n+\frac{1}{2}}),
	\quad \zeta_3\in	(t_n,t_{n+\frac{1}{2}}),
	\quad\zeta_4\in  
	(t_{n+\frac{1}{2}},t_{n+1}).
\end{align*}
Similarly, we have
\begin{align}\label{2025-11-4-9}
	E_{\eta r}&=\kappa_1\tau
	\left(E_\eta^{n+1}+E_\eta^n
	\right)
	+2\kappa_1\tau\cdot
	\mathrm{e}^{-\tfrac{\tau}{2\lambda^{*}\kappa_3}}\cdot
	E_\eta^n\nonumber\\
	&\quad
	+4\kappa_1\tau\cdot
	\mathrm{e}^{-\tfrac{t_{n+\frac{1}{2}}}{\lambda^{*}\kappa_3}}\left[
	\mathrm{e}^{\tfrac{t_{n}}{\lambda^{*}\kappa_3}}\cdot E_\eta^n
	+2\sum_{i=1}^{n-1}\mathrm{e}^{\tfrac{t_{i}}{\lambda^{*}\kappa_3}}\cdot E_\eta^i
	+E_\eta^0\right]+E_{\eta\_{re}}^n,
\end{align}
where
\begin{align}\label{2025-11-14-4}
	E_{\eta\_{re}}^n&=
	-\frac{2 n\kappa_1\tau^3}{3}\cdot\mathrm{e}^{-\tfrac{t_{n+\frac{1}{2}}}{\lambda^{*}\kappa_3}}
	\left[ (\eta(s)\cdot\mathrm{e}^{\frac{s}{\lambda^{*}\kappa_3}})''\right]_{s=\zeta_5}
	-\frac{\kappa_1\tau^3}{12}\cdot
	\mathrm{e}^{-\tfrac{t_{n+\frac{1}{2}}}{\lambda^{*}\kappa_3}}
	\left[(\eta(s)\cdot\mathrm{e}^{\frac{s}{\lambda^{*}\kappa_3}})''\right]_{s=\zeta_6}\nonumber\\
	&\quad-\frac{\kappa_1\tau^3}{4}[(\eta(s))''|_{s=\zeta_7}
	+(\eta(s))''|_{s=\zeta_8}],
\end{align}
and
\begin{align*}
	\zeta_5\in (0,t_{n}), 
	\quad \zeta_6\in (t_n,t_{n+\frac{1}{2}}),
	\quad \zeta_7\in	(t_n,t_{n+\frac{1}{2}}),
	\quad \zeta_8\in  
	(t_{n+\frac{1}{2}},t_{n+1}).
\end{align*}
Substituting \eqref{2025-11-4-8}-\eqref{2025-11-14-4} into \eqref{2025-11-4-2}, we obtain
\begin{align}\label{2025-11-4-10}
	&4\lambda^*\kappa_3(\Div E_\mathbf{u}^{n+1}+\Div E_\mathbf{u}^{n},\varphi_h)
	+\kappa_3\tau
	\left(E_\xi^{n+1}+E_\xi^n, \varphi_h
	\right)
	+2\kappa_3\tau\cdot
	\mathrm{e}^{-\tfrac{\tau}{2\lambda^{*}\kappa_3}}
	(E_\xi^n, \varphi_h)\nonumber\\
	&\quad
	+4\kappa_3\tau\cdot
	\mathrm{e}^{-\tfrac{t_{n+\frac{1}{2}}}{\lambda^{*}\kappa_3}}\left(
	\mathrm{e}^{\tfrac{t_{n}}{\lambda^{*}\kappa_3}}\cdot E_\xi^n
	+2\sum_{i=1}^{n-1}\mathrm{e}^{\tfrac{t_{i}}{\lambda^{*}\kappa_3}}\cdot E_\xi^i
	+E_\xi^0, \varphi_h\right)
	+(E_{\xi\_{re}}^n,\varphi_h)\nonumber\\
	&=\kappa_1\tau
	\left(E_\eta^{n+1}+E_\eta^n, \varphi_h
	\right)
	+2\kappa_1\tau\cdot
	\mathrm{e}^{-\tfrac{\tau}{2\lambda^{*}\kappa_3}}
	(E_\eta^n, \varphi_h)
	+(E_{\eta\_{re}}^n,\varphi_h)\nonumber\\
	&\quad
	+4\kappa_1\tau\cdot
	\mathrm{e}^{-\tfrac{t_{n+\frac{1}{2}}}{\lambda^{*}\kappa_3}}\left(
	\mathrm{e}^{\tfrac{t_{n}}{\lambda^{*}\kappa_3}}\cdot E_\eta^n
	+2\sum_{i=1}^{n-1}\mathrm{e}^{\tfrac{t_{i}}{\lambda^{*}\kappa_3}}\cdot E_\eta^i
	+E_\eta^0, \varphi_h\right)
	\quad	\forall\,\varphi_h\in M_h.
\end{align}
Subtracting \eqref{2025-10-20-3} from \eqref{2025-10-19-3}, we obtain
\begin{align}\label{2025-11-4-11}
	&(d_tE_\eta^{n+1},\psi_h) 
	+\frac{1}{\mu_{f}}\left(K(\nabla p(t_{n+\frac{1}{2}}) -\nabla p_h^{n+\frac{1}{2}}), \nabla \psi_h\right)
	\nonumber\\
	&=-\left(\eta_{t}(t_{n+\frac{1}{2}})-\frac{\eta(t_{n+1}) - \eta(t_n)}{\tau},\psi_h\right) 
	\quad
	\forall\,\psi_h\in H^1(\Omega).
\end{align}
Using the properties of projection operators $\mathcal{Q}_h$, $\mathcal{S}_h$, $\mathcal{R}_h$ to simplify \eqref{2025-11-4-1}, \eqref{2025-11-4-10} and \eqref{2025-11-4-11}, for any $(\mathbf{v}_h,\varphi_h,\psi_h)\in \mathbf{V}_h\times M_h \times W_h$, we have
\begin{align}
	&\mu\left(\varepsilon(\Theta_\mathbf{u}^{n+1})+\varepsilon(\Theta_\mathbf{u}^{n}),\varepsilon(\mathbf{v}_h)\right)-\left(\Theta_\xi^{n+1}+\Theta_\xi^{n},\Div\mathbf{v}_h\right)=\left(\Lambda_\xi^{n+1}+\Lambda_\xi^{n},\Div\mathbf{v}_h\right),
	\label{2025-11-5-1}
\end{align}
\begin{align}
&4\lambda^*\kappa_3(\Div \Theta_\mathbf{u}^{n+1}+\Div \Theta_\mathbf{u}^{n},\varphi_h)
+4\lambda^*\kappa_3(\Div \Lambda_\mathbf{u}^{n+1}+\Div \Lambda_\mathbf{u}^{n},\varphi_h)
\nonumber\\[1mm]
&\quad+\kappa_3\tau
\left(\Theta_\xi^{n+1}+\Theta_\xi^n, \varphi_h
\right)
+2\kappa_3\tau\cdot
\mathrm{e}^{-\tfrac{\tau}{2\lambda^{*}\kappa_3}}
(\Theta_\xi^n, \varphi_h)
+(E_{\xi\_{re}}^n,\varphi_h)
\nonumber\\[1mm]
&\quad
+4\kappa_3\tau\cdot
\mathrm{e}^{-\tfrac{t_{n+\frac{1}{2}}}{\lambda^{*}\kappa_3}}\left(
\mathrm{e}^{\tfrac{t_{n}}{\lambda^{*}\kappa_3}}\cdot \Theta_\xi^n
+2\sum_{i=1}^{n-1}\mathrm{e}^{\tfrac{t_{i}}{\lambda^{*}\kappa_3}}\cdot \Theta_\xi^i
+\Theta_\xi^0, \varphi_h\right)
\nonumber\\[1mm]
&=\kappa_1\tau
\left(\Theta_\eta^{n+1}+\Theta_\eta^n, \varphi_h
\right)
+2\kappa_1\tau\cdot
\mathrm{e}^{-\tfrac{\tau}{2\lambda^{*}\kappa_3}}
(\Theta_\eta^n, \varphi_h)
+(E_{\eta\_{re}}^n,\varphi_h)\nonumber\\[3mm]
&\quad
+4\kappa_1\tau\cdot
\mathrm{e}^{-\tfrac{t_{n+\frac{1}{2}}}{\lambda^{*}\kappa_3}}\left(
\mathrm{e}^{\tfrac{t_{n}}{\lambda^{*}\kappa_3}}\cdot \Theta_\eta^n
+2\sum_{i=1}^{n-1}\mathrm{e}^{\tfrac{t_{i}}{\lambda^{*}\kappa_3}}\cdot \Theta_\eta^i
+\Theta_\eta^0, \varphi_h\right),\label{2025-11-5-2}\\[3mm]
	&(d_t\Theta_\eta^{n+1},\psi_h)+\frac{1}{\mu_{f}}(K\nabla\hat{Z}_p^{n+\frac{1}{2}},\nabla\psi_h)	
	=-\frac{\lambda^* \kappa_1}{\mu_{f}}
	(K\nabla R_\mathbf{u}^{n+\frac{1}{2}},\nabla\psi_h)
	-(R_\eta^{n+\frac{1}{2}},\psi_h),
	\label{2025-11-5-9}
\end{align}
where
\begin{align*}
	R^{n+\frac{1}{2}}_\eta 
	&=\eta_{t}(t_{n+\frac{1}{2}})-\frac{\eta(t_{n+1}) - \eta(t_n)}{\tau},\\
	R_\mathbf{u}^{n+\frac{1}{2}}&=\Div\mathbf{u}_{t}(t_{n+\frac{1}{2}})-\frac{\Div\mathbf{u}(t_{n+1}) - \Div\mathbf{u}(t_n)}{\tau},\\
	\hat{Z}_p^{n+\frac{1}{2}}&=\kappa_{1}
	\frac{\hat{\Theta}_\xi^{n+1}+\hat{\Theta}_\xi^{n}}{2}+ \kappa_{2}\frac{\hat{\Theta}_\eta^{n+1}+\hat{\Theta}_\eta^{n}}{2}
	+\lambda^*\kappa_1d_t\hat{\Theta}_q^{n+1}.
\end{align*}

\begin{remark}
	In the error estimation, we assume $\mathbf{u}\in H^2(\Omega)$, therefore $q=\Div\mathbf{u}\in H^1(\Omega)$. 
	From \eqref{2025-11-5-9}, we have
	\begin{align*}
		p(t_{n+\frac{1}{2}}) - p_h^{n+\frac{1}{2}}
		&=
		\kappa_{1}\frac{\xi(t_{n+1}) + \xi(t_{n})}{2}
		-\kappa_1 \frac{\xi_h^{n+1} + \xi_h^n}{2} +\kappa_{2}\frac{\eta(t_{n+1}) + \eta(t_{n})}{2}
		\\
		&\quad
		-\kappa_2 \frac{\eta_h^{n+1} + \eta_h^n}{2}		
		+\lambda^{*}\kappa_{1}\Div\mathbf{u}_{t}(t_{n+\frac{1}{2}})
		-\lambda^* \kappa_1 \frac{\Div\mathbf{u}_h^{n+1} - \Div\mathbf{u}_h^n}{\tau}	\\
		&=\kappa_{1}\frac{\hat{\Theta}_\xi^{n+1} + \hat{\Theta}_\xi^{n}+\hat{\Lambda}_\xi^{n+1} + \hat{\Lambda}_\xi^{n}}{2}
		+\kappa_{2}\frac{\hat{\Theta}_\eta^{n+1} + \hat{\Theta}_\eta^{n}+\hat{\Lambda}_\eta^{n+1} + \hat{\Lambda}_\eta^{n}}{2}
		\\
		&\quad
		+\lambda^* \kappa_1 \frac{\Div\mathbf{u}(t_{n+1}) - \Div\mathbf{u}(t_{n})}{\tau}
		-\lambda^* \kappa_1 \frac{\Div\mathbf{u}_h^{n+1} - \Div\mathbf{u}_h^n}{\tau}	\\
		&\quad
		+\lambda^{*}\kappa_{1}\Div\mathbf{u}_{t}(t_{n+\frac{1}{2}})
		-\lambda^* \kappa_1 \frac{\Div\mathbf{u}(t_{n+1}) - \Div\mathbf{u}(t_{n})}{\tau}	\\
		&=\kappa_{1}\frac{\hat{\Theta}_\xi^{n+1} + \hat{\Theta}_\xi^{n}}{2}
		+\kappa_{2}\frac{\hat{\Theta}_\eta^{n+1} + \hat{\Theta}_\eta^{n}}{2}
		+\lambda^* \kappa_1
		d_t\hat{\Theta}_q^{n+1}
		+\lambda^* \kappa_1 R_\mathbf{u}^{n+\frac{1}{2}}
		\\
		&\quad
		+\kappa_{1}\frac{\hat{\Lambda}_\xi^{n+1} + \hat{\Lambda}_\xi^{n}}{2}
		+\kappa_{2}\frac{\hat{\Lambda}_\eta^{n+1} + \hat{\Lambda}_\eta^{n}}{2}
		+\lambda^* \kappa_1
		d_t\hat{\Lambda}_q^{n+1}\\
		&=\hat{Z}_p^{n+\frac{1}{2}}
		+\lambda^* \kappa_1 R_\mathbf{u}^{n+\frac{1}{2}}
		+\kappa_{1}\frac{\hat{\Lambda}_\xi^{n+1} + \hat{\Lambda}_\xi^{n}}{2}
		+\kappa_{2}\frac{\hat{\Lambda}_\eta^{n+1} + \hat{\Lambda}_\eta^{n}}{2}
		+\lambda^* \kappa_1
		d_t\hat{\Lambda}_q^{n+1},
	\end{align*}
	Using the properties of projection operators $\mathcal{S}_h$, we have
	\begin{align*}
		(K\nabla\hat{\Lambda}_\xi^{n+1}, \nabla \psi_h)=
		(K\nabla\hat{\Lambda}_\eta^{n+1}, \nabla \psi_h)=
		(K\nabla(\lambda^* \kappa_1
		d_t\hat{\Lambda}_q^{n+1}), \nabla \psi_h)=0.
	\end{align*}
\end{remark}
For the convenience of subsequent estimates, we give an equivalent form of $\hat{Z}_p^{n+\frac{1}{2}}$:
\begin{align*}
	\hat{Z}_p^{n+\frac{1}{2}}
	&=\frac{\kappa_{1}}{2}\left(\mathcal{S}_h\xi(t_{n+1})-\xi_h^{n+1}+\mathcal{S}_h\xi(t_{n})-\xi_h^{n}\right)
	+\lambda^*\kappa_1d_t\hat{\Theta}_q^{n+1}\\
	&\quad+\frac{\kappa_2}{2}\left(\mathcal{S}_h\eta(t_{n+1})-\eta_h^{n+1}+\mathcal{S}_h\eta(t_{n})-\eta_h^{n}\right)\\
	&=\frac{\kappa_{1}}{2}(\Theta_\xi^{n+1}+\Theta_\xi^{n}) +\frac{\kappa_{2}}{2}(\Theta_\eta^{n+1}+\Theta_\eta^{n})
	+\lambda^*\kappa_1d_t\hat{\Theta}_q^{n+1}\\
	&\quad+\frac{\kappa_{1}}{2}(\Lambda_\xi^{n+1}-\hat{\Lambda}_\xi^{n+1}+\Lambda_\xi^{n}-\hat{\Lambda}_\xi^{n})
	+\frac{\kappa_{2}}{2}(\Lambda_\eta^{n+1}-\hat{\Lambda}_\eta^{n+1}+\Lambda_\eta^{n}-\hat{\Lambda}_\eta^{n})\nonumber\\
	&=\frac{\kappa_{1}}{2}(\Theta_\xi^{n+1}+\Theta_\xi^{n}) +\frac{\kappa_{2}}{2}(\Theta_\eta^{n+1}+\Theta_\eta^{n})
	+\lambda^*\kappa_1d_t\Div\Theta_\mathbf{u}^{n+1}\\
	&\quad+\frac{\kappa_{1}}{2}(\Lambda_\xi^{n+1}-\hat{\Lambda}_\xi^{n+1}+\Lambda_\xi^{n}-\hat{\Lambda}_\xi^{n})
	+\frac{\kappa_{2}}{2}(\Lambda_\eta^{n+1}-\hat{\Lambda}_\eta^{n+1}+\Lambda_\eta^{n}-\hat{\Lambda}_\eta^{n})\nonumber\\
	&\quad+\lambda^*\kappa_1d_t(\Div\Lambda_{\mathbf{u}}^{n+1}-\hat{\Lambda}_q^{n+1}),
\end{align*}
where
\begin{align*}
	\hat{\Theta}_q^{n+1}
	&=\mathcal{S}_hq(t_{n+1})
	-q_h^{n+1}
	-\Div\mathcal{R}_h\mathbf{u}(t_{n+1})
	+\Div\mathcal{R}_h\mathbf{u}(t_{n+1})\\
	&=\Div \mathcal{R}_h\mathbf{u}(t_{n+1})-\Div\mathbf{u}_h^{n+1}+\Div\Lambda_{\mathbf{u}}^{n+1}-\hat{\Lambda}_q^{n+1}\\
	&=\Div\Theta_{\mathbf{u}}^{n+1}+\Div\Lambda_{\mathbf{u}}^{n+1}-\hat{\Lambda}_q^{n+1}.
\end{align*}
Applying the operator $d_t$ to \eqref{2025-11-5-2}, we obtain
\begin{align}\label{2025-11-5-8}
	&4\lambda^*\kappa_3(d_t\Div(\Theta_\mathbf{u}^{n+1}+ \Theta_\mathbf{u}^{n}),\varphi_h)
	+4\lambda^*\kappa_3(d_t\Div (\Lambda_\mathbf{u}^{n+1}+ \Lambda_\mathbf{u}^{n}),\varphi_h)
	+(d_tE_{\xi\_re}^n,\varphi_h)
	\nonumber\\[1mm]
	&\quad+4\kappa_3\tau\left(
	\mathrm{e}^{-\tfrac{\tau}{2\lambda^{*}\kappa_3}}
	d_t\Theta_\xi^n
	+2\sum_{i=1}^{n-1}
	\mathrm{e}^{-\tfrac{(n+\frac{1}{2}-i)\tau}{\lambda^{*}\kappa_3}}
	d_t\Theta_\xi^i
	+\frac{1}{\tau}(\mathrm{e}^{-\tfrac{t_{n+\frac{1}{2}}}{\lambda^{*}\kappa_3}}
	-
	\mathrm{e}^{-\tfrac{t_{n-\frac{1}{2}}}{\lambda^{*}\kappa_3}}
	)\Theta_\xi^0, \varphi_h\right)
	\nonumber\\[1mm]
	&\quad+\kappa_3\tau
	\left(d_t(\Theta_\xi^{n+1}+\Theta_\xi^n), \varphi_h
	\right)
	+2\kappa_3\tau\cdot
	\mathrm{e}^{-\tfrac{\tau}{2\lambda^{*}\kappa_3}}
	(d_t\Theta_\xi^n, \varphi_h)
	\nonumber\\[1mm]
	&=4\kappa_1\tau\left(
	\mathrm{e}^{-\tfrac{\tau}{2\lambda^{*}\kappa_3}}
	d_t\Theta_\eta^n
	+2\sum_{i=1}^{n-1}
	\mathrm{e}^{-\tfrac{(n+\frac{1}{2}-i)\tau}{\lambda^{*}\kappa_3}}
	d_t\Theta_\eta^i
	+\frac{1}{\tau}(\mathrm{e}^{-\tfrac{t_{n+\frac{1}{2}}}{\lambda^{*}\kappa_3}}
	-
	\mathrm{e}^{-\tfrac{t_{n-\frac{1}{2}}}{\lambda^{*}\kappa_3}}
	)\Theta_\eta^0, \varphi_h\right)\nonumber\\[1mm]
	&\quad+\kappa_1\tau
	\left(d_t(\Theta_\eta^{n+1}+\Theta_\eta^n), \varphi_h
	\right)
	+2\kappa_1\tau\cdot
	\mathrm{e}^{-\tfrac{\tau}{2\lambda^{*}\kappa_3}}
	(d_t\Theta_\eta^n, \varphi_h)
	+(d_tE_{\eta\_re}^n,\varphi_h).
\end{align}
Choosing the test functions in equations \eqref{2025-11-5-1}, \eqref{2025-11-5-8}, and \eqref{2025-11-5-9} as $\mathbf{v}_h=d_t(\Theta_\mathbf{u}^{n+1}+\Theta_\mathbf{u}^{n})$, $\varphi_h= \Theta_\xi^{n+1}+\Theta_\xi^{n}$, and $\psi_h= 2\hat{Z}_p^{n+\frac{1}{2}}$, respectively, we get
\begin{align}
	&\mu\left(\varepsilon(\Theta_\mathbf{u}^{n+1}+\Theta_\mathbf{u}^{n}),d_t\varepsilon(\Theta_\mathbf{u}^{n+1}+\Theta_\mathbf{u}^{n})\right)
	-\left(\Theta_\xi^{n+1}+\Theta_\xi^{n},d_t\Div(\Theta_\mathbf{u}^{n+1}+\Theta_\mathbf{u}^{n})\right)\nonumber\\[1mm]
	&
	=\left(\Lambda_\xi^{n+1}+\Lambda_\xi^{n},d_t\Div(\Theta_\mathbf{u}^{n+1}+\Theta_\mathbf{u}^{n})\right),
	\label{2025-11-6-1}\\
&4\lambda^*\kappa_3(d_t\Div(\Theta_\mathbf{u}^{n+1}+ \Theta_\mathbf{u}^{n}),\Theta_\xi^{n+1}+\Theta_\xi^{n})
+4\lambda^*\kappa_3(d_t\Div (\Lambda_\mathbf{u}^{n+1}+ \Lambda_\mathbf{u}^{n}),\Theta_\xi^{n+1}+\Theta_\xi^{n})
\nonumber\\[1mm]
&\quad
+\kappa_3\tau
\left(d_t(\Theta_\xi^{n+1}+\Theta_\xi^n), \Theta_\xi^{n+1}+\Theta_\xi^{n}
\right)
+(d_tE_{\xi\_re}^n,\Theta_\xi^{n+1}+\Theta_\xi^{n})+R_4
\nonumber\\
&\quad+2\kappa_3\tau\cdot
\mathrm{e}^{-\tfrac{\tau}{2\lambda^{*}\kappa_3}}
(d_t(\Theta_\xi^{n+1}+\Theta_\xi^{n}), \Theta_\xi^{n+1}+\Theta_\xi^{n})
\nonumber\\[1mm]
&=\kappa_1\tau
\left(d_t(\Theta_\eta^{n+1}+\Theta_\eta^n), \Theta_\xi^{n+1}+\Theta_\xi^{n}
\right)
+2\kappa_1\tau\cdot
\mathrm{e}^{-\tfrac{\tau}{2\lambda^{*}\kappa_3}}
(d_t\Theta_\eta^n, \Theta_\xi^{n+1}+\Theta_\xi^{n})\nonumber\\[1mm]
&\quad
+2\kappa_3\tau\cdot
\mathrm{e}^{-\tfrac{\tau}{2\lambda^{*}\kappa_3}}
(d_t\Theta_\xi^{n+1}, \Theta_\xi^{n+1}+\Theta_\xi^{n})
+(d_tE_{\eta\_re}^n,\Theta_\xi^{n+1}+\Theta_\xi^{n})+R_5,
\label{2025-11-6-2}\\[2mm]
	&(d_t\Theta_\eta^{n+1},\kappa_{2}(\Theta_\eta^{n+1}+\Theta_\eta^{n}))
	+(d_t\Theta_\eta^{n+1},\kappa_{1}(\Theta_\xi^{n+1}+\Theta_\xi^{n}))
	+2(d_t\Theta_\eta^{n+1},\lambda^*\kappa_1d_t\Div\Theta_\mathbf{u}^{n+1})
	\nonumber\\[1mm]
	&\quad+(d_t\Theta_\eta^{n+1},\kappa_1(\Lambda_\xi^{n+1}-\hat{\Lambda}_\xi^{n+1}+\Lambda_\xi^{n}-\hat{\Lambda}_\xi^{n})
	+\kappa_{2}(\Lambda_\eta^{n+1}-\hat{\Lambda}_\eta^{n+1}+\Lambda_\eta^{n}-\hat{\Lambda}_\eta^{n}))
	\nonumber\\[1mm]
	&\quad+2(d_t\Theta_\eta^{n+1},\lambda^*\kappa_1d_t(\Div\Lambda_{\mathbf{u}}^{n+1}-\hat{\Lambda}_q^{n+1}))
	+\frac{2}{\mu_{f}}(K\nabla\hat{Z}_p^{n+\frac{1}{2}},\nabla\hat{Z}_p^{n+\frac{1}{2}})
		\nonumber\\[1mm]
	&=
	-\frac{2\lambda^* \kappa_1}{\mu_{f}}
	(K\nabla R_\mathbf{u}^{n+\frac{1}{2}},\nabla\hat{Z}_p^{n+\frac{1}{2}})
	-2(R_\eta^{n+\frac{1}{2}},\hat{Z}_p^{n+\frac{1}{2}})
	,\label{2025-11-6-3}
\end{align}
where
\begin{align*}
	R_4&=4\kappa_3\tau\left(
	\mathrm{e}^{-\tfrac{\tau}{2\lambda^{*}\kappa_3}}
	d_t\Theta_\xi^n
	+2\sum_{i=1}^{n-1}
	\mathrm{e}^{-\tfrac{(n+\frac{1}{2}-i)\tau}{\lambda^{*}\kappa_3}}
	d_t\Theta_\xi^i
	+\frac{1}{\tau}(\mathrm{e}^{-\tfrac{t_{n+\frac{1}{2}}}{\lambda^{*}\kappa_3}}
	-
	\mathrm{e}^{-\tfrac{t_{n-\frac{1}{2}}}{\lambda^{*}\kappa_3}}
	)\Theta_\xi^0, \Theta_\xi^{n+1}+\Theta_\xi^{n}\right),\\
	R_5&=4\kappa_1\tau\left(
	\mathrm{e}^{-\tfrac{\tau}{2\lambda^{*}\kappa_3}}
	d_t\Theta_\eta^n
	+2\sum_{i=1}^{n-1}
	\mathrm{e}^{-\tfrac{(n+\frac{1}{2}-i)\tau}{\lambda^{*}\kappa_3}}
	d_t\Theta_\eta^i
	+\frac{1}{\tau}(\mathrm{e}^{-\tfrac{t_{n+\frac{1}{2}}}{\lambda^{*}\kappa_3}}
	-
	\mathrm{e}^{-\tfrac{t_{n-\frac{1}{2}}}{\lambda^{*}\kappa_3}}
	)\Theta_\eta^0, \Theta_\xi^{n+1}+\Theta_\xi^{n}\right).
\end{align*}
To eliminate the cross terms, reducing the time level of \eqref{2025-11-5-9} to level $n$, and still choosing the test function $\psi_h= 2\hat{Z}_p^{n+\frac{1}{2}}$, we obtain
\begin{align}
	&(d_t\Theta_\eta^{n},\kappa_{2}(\Theta_\eta^{n+1}+\Theta_\eta^{n}))
	+(d_t\Theta_\eta^{n},\kappa_{1}(\Theta_\xi^{n+1}+\Theta_\xi^{n}))
	+2(d_t\Theta_\eta^{n},\lambda^*\kappa_1d_t\Div\Theta_\mathbf{u}^{n+1})
	\nonumber\\[1mm]
	&\quad+(d_t\Theta_\eta^{n},\kappa_1(\Lambda_\xi^{n+1}-\hat{\Lambda}_\xi^{n+1}+\Lambda_\xi^{n}-\hat{\Lambda}_\xi^{n})
	+\kappa_{2}(\Lambda_\eta^{n+1}-\hat{\Lambda}_\eta^{n+1}+\Lambda_\eta^{n}-\hat{\Lambda}_\eta^{n}))
	\nonumber\\[1mm]
	&\quad+2(d_t\Theta_\eta^{n},\lambda^*\kappa_1d_t(\Div\Lambda_{\mathbf{u}}^{n+1}-\hat{\Lambda}_q^{n+1}))
	+\frac{2}{\mu_{f}}\left(K\nabla\hat{Z}_p^{n-\frac{1}{2}},\nabla\hat{Z}_p^{n+\frac{1}{2}}\right)
	\nonumber\\[1mm]
	&=
	-\frac{2\lambda^* \kappa_1}{\mu_{f}}
	(K\nabla R_\mathbf{u}^{n-\frac{1}{2}},\nabla\hat{Z}_p^{n+\frac{1}{2}})
	-2(R_\eta^{n-\frac{1}{2}},\hat{Z}_p^{n+\frac{1}{2}})
	.\label{2026-1-24-1}
\end{align}
First, multiplying \eqref{2025-11-6-1} by $4\lambda^*\kappa_3$, then adding it with \eqref{2025-11-6-2}-\eqref{2026-1-24-1}, and finally applying the summation operator $\tau\sum_{n=0}^{l}$, we get
\begin{align}\label{2025-11-6-4}
	&4\mu\lambda^*\kappa_3\tau\sum_{n=0}^{l}\left(d_t\varepsilon(\Theta_\mathbf{u}^{n+1}+\Theta_\mathbf{u}^{n}),\varepsilon(\Theta_\mathbf{u}^{n+1}+\Theta_\mathbf{u}^{n})\right)
	+\kappa_3\tau^2\sum_{n=0}^{l}
	\left(d_t(\Theta_\xi^{n+1}+\Theta_\xi^n), \Theta_\xi^{n+1}+\Theta_\xi^{n}
	\right)\nonumber\\
	&\quad+\kappa_{2}\tau\sum_{n=0}^{l}
	(d_t(\Theta_\eta^{n+1}+\Theta_\eta^{n}),\Theta_\eta^{n+1}+\Theta_\eta^{n})
	+\kappa_1\tau\sum_{n=0}^{l}
	(d_t(\Theta_\eta^{n+1}+\Theta_\eta^{n}),\Theta_\xi^{n+1}+\Theta_\xi^{n})
	\nonumber\\
	&\quad
	+2\tau\sum_{n=0}^{l}
	(d_t(\Theta_\eta^{n+1}+\Theta_\eta^{n}),\lambda^*\kappa_1d_t\Div\Theta_{\mathbf{u}}^{n+1})
	+\frac{2\tau}{\mu_f}\sum_{n=0}^{l}
	(K\nabla\hat{Z}_p^{n+\frac{1}{2}},\nabla\hat{Z}_p^{n+\frac{1}{2}})
	\nonumber\\
	&\quad+2\kappa_3\tau^2
	\mathrm{e}^{-\tfrac{\tau}{2\lambda^{*}\kappa_3}}
	\sum_{n=0}^{l}
	(d_t(\Theta_\xi^{n+1}+\Theta_\xi^{n}), \Theta_\xi^{n+1}+\Theta_\xi^{n})
	=\sum_{j=1}^{12}\Phi_j,
\end{align}
Applying the identity \eqref{2025-10-30-9} and the identity
\begin{align*}
	\kappa_1\tau\sum_{n=0}^{l}(d_t(\Theta_\eta^{n+1}+\Theta_\eta^{n}),(\Theta_\xi^{n+1}+\Theta_\xi^{n}))
	&=\tau\sum_{n=0}^{l}(
	d_t\Div(\Theta_\mathbf{u}^{n+1}+\Theta_\mathbf{u}^{n}),(\Theta_\xi^{n+1}+\Theta_\xi^{n}))
	\nonumber\\
	&\quad
	+\lambda^*\kappa_3\tau\sum_{n=0}^{l}
	(d_t(d_t\Div(\Theta_\mathbf{u}^{n+1}+\Theta_\mathbf{u}^{n})),(\Theta_\xi^{n+1}+\Theta_\xi^{n}))\\
	&\quad+\kappa_3\tau\sum_{n=0}^{l}
	(d_t(\Theta_\xi^{n+1}+\Theta_\xi^{n}),(\Theta_\xi^{n+1}+\Theta_\xi^{n})),
\end{align*}
to both sides of \eqref{2025-11-6-4}, we obtain \eqref{2025-11-7-1}.
\end{proof}

Before estimating the errors, we first give a lemma about $\|R_\eta^{n+1}\|_{H^{-1}(\Omega)}$ and $\|R_\mathbf{u}^{n+1}\|_0$.
\begin{lemma}\label{2025-11-6-5}
	The following inequalities hold:
	\begin{align}
		\|R^{n+\frac{1}{2}}_\eta\|_0^2
		&\leq C\tau^3\int_{t_{n}}^{t_{n+1}} \|\eta_{ttt}(s)\|_0^2  \, \mathrm{~d}s,\label{2026-1-24-2}\\
		\|R_\mathbf{u}^{n+\frac{1}{2}}\|_0^2&\leq C\tau^3\int_{t_{n}}^{t_{n+1}}\|\Div \mathbf{u}_{ttt}(s)\|_0^2\mathrm{~d}s,\label{2026-1-24-3}\\
		\|\nabla R_\mathbf{u}^{n+\frac{1}{2}}\|_0^2&\leq C\tau^3\int_{t_{n}}^{t_{n+1}}\|\nabla\Div \mathbf{u}_{ttt}(s)\|_0^2\mathrm{~d}s.\label{2026-1-24-4}		
	\end{align}
\end{lemma}

\begin{proof}
	Using Taylor's formula with integral remainder to expand $\eta(t_{n+1})$ and $\eta(t_n)$ at $t_{n+\frac{1}{2}}$, we obtain
	\begin{align*}
		\eta(t_{n+1}) &= \eta(t_{n+\frac{1}{2}}) + \frac{\tau}{2}\eta_t(t_{n+\frac{1}{2}}) + \frac{\tau^2}{8}\eta_{tt}(t_{n+\frac{1}{2}})
		+\frac{1}{2}\int_{t_{n+\frac{1}{2}}}^{t_{n+1}} (t_{n+1} - s)^2 \eta_{ttt}(s) \, \mathrm{~d}s,\\	
		\eta(t_n) &= \eta(t_{n+\frac{1}{2}}) - \frac{\tau}{2}\eta_t(t_{n+\frac{1}{2}}) + \frac{\tau^2}{8}\eta_{tt}(t_{n+\frac{1}{2}})
		-\frac{1}{2}\int_{t_n}^{t_{n+\frac{1}{2}}} (t_n - s)^2 \eta_{ttt}(s) \, \mathrm{~d}s,
	\end{align*}
	then
	\begin{align*}
		R^{n+\frac{1}{2}}_\eta 
		&=\eta_{t}(t_{n+\frac{1}{2}})-\frac{\eta(t_{n+1}) - \eta(t_n)}{\tau}\nonumber\\
		&= -\frac{1}{2\tau} \int_{t_{n+\frac{1}{2}}}^{t_{n+1}} (t_{n+1} - s)^2 \eta_{ttt}(s) \, \mathrm{~d}s  -\frac{1}{2\tau}\int_{t_n}^{t_{n+\frac{1}{2}}} (s - t_n)^2 \eta_{ttt}(s) \, \mathrm{~d}s\nonumber\\
		&\leq\frac{1}{2\tau}\left( \int_{t_{n+\frac{1}{2}}}^{t_{n+1}} (t_{n+1} - s)^4  \, \mathrm{~d}s \right)^{\frac{1}{2}}\cdot
		\left( \int_{t_{n+\frac{1}{2}}}^{t_{n+1}} \|\eta_{ttt}(s)\|_0^2  \, \mathrm{~d}s \right)^{\frac{1}{2}}\nonumber\\
		&\quad+
		\frac{1}{2\tau}\left( \int_{t_n}^{t_{n+\frac{1}{2}}} (s-t_{n})^4  \, \mathrm{~d}s \right)^{\frac{1}{2}}\cdot
		\left( \int_{t_n}^{t_{n+\frac{1}{2}}} \|\eta_{ttt}(s)\|_0^2  \, \mathrm{~d}s \right)^{\frac{1}{2}},\nonumber
	\end{align*}
	squaring both sides, we get
	\begin{align*}
		\|R^{n+\frac{1}{2}}_\eta\|_0^2
		&\leq
		\frac{\tau^3}{640}\int_{t_{n+\frac{1}{2}}}^{t_{n+1}} \|\eta_{ttt}(s)\|_0^2  \, \mathrm{~d}s
		+\frac{\tau^3}{640}\int_{t_n}^{t_{n+\frac{1}{2}}} \|\eta_{ttt}(s)\|_0^2  \, \mathrm{~d}s\nonumber\\
		&\leq\frac{\tau^3}{640}\int_{t_{n}}^{t_{n+1}} \|\eta_{ttt}(s)\|_0^2  \, \mathrm{~d}s.
	\end{align*}	
	Inequality \eqref{2026-1-24-2} is proved. Similarly, we can prove \eqref{2026-1-24-3} and \eqref{2026-1-24-4}.
\end{proof}

Next, we estimate \eqref{2025-11-7-1} in Theorem \ref{th-2025-11-7-1}. According to \eqref{2025-11-14-3} and \eqref{2025-11-14-4}, we have
\begin{align}\label{2025-11-14-5}
	\|d_tE_{\xi\_re}^n\|_0\leq C\kappa_3\tau^2\|\xi_{ttt}\|_{L^{\infty}(0,T;L^2(\Omega))},
	\quad 
	\|d_tE_{\eta\_re}^n\|_0\leq C\kappa_1\tau^2\|\eta_{ttt}\|_{L^{\infty}(0,T;L^2(\Omega))},
\end{align}
Using \eqref{2025-11-14-5} and the Cauchy-Schwarz inequality, we obtain
\begin{align}
	\Phi_1&=\tau\sum_{n=0}^{l}\left[
	-(d_tE_{\xi\_re}^n,\Theta_\xi^{n+1}+\Theta_\xi^{n})+(d_tE_{\eta\_re}^n,\Theta_\xi^{n+1}+\Theta_\xi^{n})\right]
	\nonumber\\
	&\leq \tau\sum_{n=0}^{l}(\|d_tE_{\xi\_re}^n\|_0+\|d_tE_{\eta\_re}^n\|_0)\cdot\|\Theta_\xi^{n+1}+\Theta_\xi^{n}\|_0\nonumber\\
	&\leq C\kappa_3^2\tau^4\|\xi_{ttt}\|^2_{L^{\infty}(0,T;L^2(\Omega))}
	+C\kappa_1^2\tau^4\|\eta_{ttt}\|^2_{L^{\infty}(0,T;L^2(\Omega))}+\tau\sum_{n=0}^{l}\|\Theta_\xi^{n+1}+\Theta_\xi^{n}\|_0^2.
\end{align}
Using the summation by parts formula, the Cauchy-Schwarz inequality, and the properties of projection operators $\mathcal{S}_h$, we have
\begin{align}
	\Phi_2
	&=-\kappa_1\tau\sum_{n=0}^{l}(d_t(\Theta_\eta^{n+1}+\Theta_\eta^{n}),\Lambda_\xi^{n+1}-\hat{\Lambda}_\xi^{n+1}+\Lambda_\xi^{n}-\hat{\Lambda}_\xi^{n})\nonumber\\
	&=\kappa_1(\Theta_\eta^{l+1}+\Theta_\eta^{l},\hat{\Lambda}_\xi^{l+1}+\hat{\Lambda}_\xi^{l})
	-\kappa_1(\Theta_\eta^{0}+\Theta_\eta^{-1},\hat{\Lambda}_\xi^{1}+\hat{\Lambda}_\xi^{0})\nonumber\\
	&\quad
	-\kappa_1\tau\sum_{n=1}^{l}(\Theta_\eta^{n}+\Theta_\eta^{n-1},d_t(\hat{\Lambda}_\xi^{n+1}+\hat{\Lambda}_\xi^{n}))\nonumber\\
	&\leq\frac{\kappa_1}{2}\|\Theta_\eta^{l+1}+\Theta_\eta^{l}\|_0^2
	+\frac{\kappa_1}{2}\|\hat{\Lambda}_\xi^{l+1}+\hat{\Lambda}_\xi^{l}\|_0^2
	+\frac{\kappa_1}{2}\|\Theta_\eta^{0}+\Theta_\eta^{-1}\|_0^2
	+\frac{\kappa_1}{2}\|\hat{\Lambda}_\xi^{1}+\hat{\Lambda}_\xi^{0}\|_0^2
	\nonumber\\
	&\quad+\frac{\kappa_1\tau}{2}\sum_{n=1}^{l}\|\Theta_\eta^{n}+\Theta_\eta^{n-1}\|_0^2
	+\frac{\kappa_1\tau}{2}\sum_{n=1}^{l}(\|d_t\hat{\Lambda}_\xi^{n+1}\|_0^2+\|d_t\hat{\Lambda}_\xi^{n}\|_0^2)\nonumber\\
	&\leq C\kappa_1\tau\sum_{n=0}^{l+1}\|\Theta_\eta^{n}+\Theta_\eta^{n-1}\|_0^2
	+C\kappa_1h^{2(r+1)}\|\xi_t\|^2_{L^2(0,T;H^{r+1}(\Omega))}.
\end{align}
Similarly, estimating $\Phi_3$, we get
\begin{align}
	\Phi_3&=-\kappa_2\tau\sum_{n=0}^{l}(d_t(\Theta_\eta^{n+1}+\Theta_\eta^{n}),\Lambda_\eta^{n+1}-\hat{\Lambda}_\eta^{n+1}+\Lambda_\eta^{n}-\hat{\Lambda}_\eta^{n})\nonumber\\	
	&\leq C\kappa_2\tau\sum_{n=0}^{l+1}\|\Theta_\eta^{n}+\Theta_\eta^{n-1}\|_0^2
	+	C\kappa_2h^{2(r+1)}\|\eta_t\|^2_{L^2(0,T;H^{r+1}(\Omega))}.
\end{align}
Using the Cauchy-Schwarz inequality, Korn's inequality, Young's inequality, and the properties of projection operators $\mathcal{R}_h$, $\mathcal{S}_h$, we obtain
\begin{align}
	\Phi_4&=-2\lambda^*\kappa_1\tau\sum_{n=0}^{l}
	(d_t(\Theta_\eta^{n+1}+\Theta_\eta^{n}),d_t(\Div\Lambda_{\mathbf{u}}^{n+1}-\hat{\Lambda}_q^{n+1}))\nonumber\\
	&\leq 4\lambda^*\kappa_1\tau\sum_{n=0}^{l}\|d_t(\Theta_\eta^{n+1}+\Theta_\eta^{n})\|_0\cdot(\|d_t\Div\Lambda_{\mathbf{u}}^{n+1}\|_0+\|d_t\hat{\Lambda}_q^{n+1}\|_0)\nonumber\\
	&\leq
	\kappa_1\tau\sum_{n=0}^{l}\|\Theta_\eta^{n+1}+\Theta_\eta^{n}\|_0^2
	+ C\lambda^*\kappa_1h^{2(r+1)}\|\Div\mathbf{u}_{tt}\|^2_{L^2(0,T;H^{r+1}(\Omega))}.
\end{align}
Using the Cauchy-Schwarz inequality, Korn's inequality, Young's inequality and Lemma \ref{2025-11-6-5}, we get
\begin{align}
	\Phi_5&=-\frac{2\lambda^*\kappa_1}{\mu_{f}}\cdot\tau\sum_{n=0}^{l}
	(K\nabla R_\mathbf{u}^{n+\frac{1}{2}}+K\nabla R_\mathbf{u}^{n-\frac{1}{2}},\nabla\hat{Z}_p^{n+\frac{1}{2}})
	\nonumber\\
	&\leq \frac{2(\lambda^*\kappa_1)^2K_2^2\tau}{\mu_{f}K_1}
	\sum_{n=0}^{l}\|\nabla R_\mathbf{u}^{n+\frac{1}{2}}\|_0^2
	+\frac{K_1\tau}{\mu_f}\sum_{n=0}^{l}\|\nabla\hat{Z}_p^{n+\frac{1}{2}}\|_0^2\nonumber\\
	&\leq\frac{(\lambda^*\kappa_1)^2K_2^2\tau^4}{12\mu_{f}K_1}\|\nabla\Div \mathbf{u}_{ttt}\|^2_{L^2(0,T;L^2(\Omega))}
	+\frac{K_1\tau}{\mu_f}\sum_{n=0}^{l}\|\nabla\hat{Z}_p^{n+\frac{1}{2}}\|_0^2.
\end{align}
Using the Cauchy-Schwarz inequality, we have
\begin{align}
	\Phi_6&=-4\lambda^*\kappa_3\tau\sum_{n=0}^{l}
	(d_t\Div (\Lambda_\mathbf{u}^{n+1}+ \Lambda_\mathbf{u}^{n}),\Theta_\xi^{n+1}+\Theta_\xi^{n})\nonumber\\
	&\leq 2\lambda^*\kappa_3\tau\sum_{n=0}^{l}\|d_t\Div (\Lambda_\mathbf{u}^{n+1}+ \Lambda_\mathbf{u}^{n})\|_0^2
	+2\lambda^*\kappa_3\tau\sum_{n=0}^{l}\|\Theta_\xi^{n+1}+\Theta_\xi^{n}\|_0^2\nonumber\\
	&\leq Ch^{2(r+1)}\|\Div\mathbf{u}_t\|^2_{L^2(0,T;H^{r+1}(\Omega))}
	+2\lambda^*\kappa_3\tau\sum_{n=0}^{l}\|\Theta_\xi^{n+1}+\Theta_\xi^{n}\|_0^2.
\end{align}
Using the Cauchy-Schwarz inequality, Korn's inequality, Young's inequality and the properties of projection operators, we get
\begin{align}
	\Phi_7&=4\lambda^*\kappa_3\tau\sum_{n=0}^{l}
	\left(\Lambda_\xi^{n+1}+\Lambda_\xi^{n},d_t\Div(\Theta_\mathbf{u}^{n+1}+\Theta_\mathbf{u}^{n})\right)
	\nonumber\\
	&\leq Ch^{2(r+1)}\tau\sum_{n=0}^{l}
	\|\xi_t\|^2_{H^{r+1}(\Omega)}
	+\frac{\lambda^*\kappa_3\tau^2}{2}\sum_{n=0}^{l}
	\|d_t\varepsilon(\Theta_\mathbf{u}^{n+1}+\Theta_\mathbf{u}^{n})\|_0^2\nonumber\\
	&\leq Ch^{2(r+1)}
	\|\xi_t\|^2_{L^2(0,T;H^{r+1}(\Omega))}
	+\frac{\lambda^*\kappa_3\tau}{2}\sum_{n=0}^{l}
	\|d_t\varepsilon(\Theta_\mathbf{u}^{n+1}+\Theta_\mathbf{u}^{n})\|_0^2.
\end{align}
Using Young's inequality, we obtain
\begin{align}
	\Phi_8&=2\kappa_3\tau^2\cdot
	\mathrm{e}^{-\tfrac{\tau}{2\lambda^{*}\kappa_3}}
	\sum_{n=0}^{l}
	(d_t\Theta_\xi^{n+1}, \Theta_\xi^{n+1}+\Theta_\xi^{n})\nonumber\\	
	&\leq
	\frac{\kappa_3\tau^2}{4}\sum_{n=1}^{l}\|d_t (\Theta_\xi^{n+1}+\Theta_\xi^{n})\|_0^2
	+\kappa_3\tau\sum_{n=0}^{l}\|\Theta_\xi^{n+1}+\Theta_\xi^{n}\|_0^2.
\end{align}
Similarly, using Young's inequality, we obtain
\begin{align}
	\Phi_{9}
	&=2\kappa_1\tau^2\cdot
	\mathrm{e}^{-\tfrac{\tau}{2\lambda^{*}\kappa_3}}\sum_{n=1}^{l}
	(d_t\Theta_\eta^n, \Theta_\xi^{n+1}+\Theta_\xi^{n})\nonumber\\
	&\leq
	\frac{\kappa_1\tau^2}{4}
	 \sum_{n=0}^{l}\|\Theta_\eta^{n}+\Theta_\eta^{n-1}\|_0^2
	+\kappa_1\tau\sum_{n=0}^{l}\| \Theta_\xi^{n+1}+\Theta_\xi^{n}\|_0^2.
\end{align}
Using Young's inequality, Poincar\'e's inequality and Lemma \ref{2025-11-6-5}, we get
\begin{align}
	\Phi_{10}&=-2\tau\sum_{n=0}^{l}(R_\eta^{n+\frac{1}{2}}+R_\eta^{n-\frac{1}{2}},\hat{Z}_p^{n+\frac{1}{2}})\nonumber\\
	&\leq C\tau^4\|\eta_{ttt}(s)\|^2_{L^2(0,T;L^2(\Omega))}
	+\frac{K_1\tau}{\mu_f}\sum_{n=0}^{l}\|\hat{Z}_p^{n+\frac{1}{2}}\|_0^2\nonumber\\
	&\leq C\tau^4\|\eta_{ttt}(s)\|^2_{L^2(0,T;L^2(\Omega))}
	+C\tau\sum_{n=0}^{l}\|\nabla\hat{Z}_p^{n+\frac{1}{2}}\|_0^2.
\end{align}
Since $K$ is symmetric positive definite, using the Cauchy-Schwarz inequality, we have
\begin{align}
	\Phi_{11}&=-\frac{2\tau}{\mu_f}\sum_{n=0}^{l}
	(K\nabla\hat{Z}_p^{n-\frac{1}{2}},\nabla\hat{Z}_p^{n+\frac{1}{2}})\nonumber\\
	&\leq
	\frac{2\tau}{\mu_f} \sum_{n=0}^{l}
	\sqrt{\left(K \nabla \hat{Z}_p^{n-\frac{1}{2}},\nabla \hat{Z}_p^{n-\frac{1}{2}}\right)}
	\sqrt{\left(K \nabla \hat{Z}_p^{n+\frac{1}{2}},\nabla \hat{Z}_p^{n+\frac{1}{2}}\right)}\nonumber\\
	&\leq
	\frac{\tau}{\mu_f} \sum_{n=0}^{l}
	\left(K \nabla \hat{Z}_p^{n-\frac{1}{2}},\nabla \hat{Z}_p^{n-\frac{1}{2}}\right)
	+\frac{\tau}{\mu_f} \sum_{n=0}^{l}
	\left(K \nabla \hat{Z}_p^{n+\frac{1}{2}},\nabla \hat{Z}_p^{n+\frac{1}{2}}\right)\nonumber\\
	&\leq\frac{2\tau}{\mu_f} \sum_{n=0}^{l}
	\left(K \nabla \hat{Z}_p^{n+\frac{1}{2}},\nabla \hat{Z}_p^{n+\frac{1}{2}}\right).
\end{align}
Using the Cauchy-Schwarz inequality and Young's inequality, we obtain
\begin{align}
	\Phi_{12}
	&\leq 4\kappa_3\tau^2\sum_{n=0}^{l}
	\left(
	d_t\Theta_\xi^n
	+2\sum_{i=1}^{n-1}d_t\Theta_\xi^i
	+d_t\Theta_\xi^0,\Theta_\xi^{n+1}+\Theta_\xi^{n}
	\right)\nonumber\\
	&\quad+4\kappa_1\tau^2\sum_{n=0}^{l}
	\left(
	d_t\Theta_\eta^n
	+2\sum_{i=1}^{n-1}d_t\Theta_\eta^i
	+d_t\Theta_\eta^0,\Theta_\xi^{n+1}+\Theta_\xi^{n}
	\right)\nonumber\\
	&\leq
	C\tau^2\sum_{n=0}^{l}(d_t(\|\Theta_\xi^{n}+\Theta_\xi^{n-1}\|_0^2)
	+d_t(\|\Theta_\eta^{n}+\Theta_\eta^{n-1}\|_0^2))\nonumber\\
	&\quad+C\tau\sum_{n=0}^{l}(\|\Theta_\xi^{n+1}+\Theta_\xi^{n}\|_0^2
	+\|\Theta_\eta^{n+1}+\Theta_\eta^{n}\|_0^2).
\end{align}
The treatment of the left-hand side terms in \eqref{2025-11-7-1} is as follows:
Using the summation by parts formula, the properties of projection operators and \eqref{2025-10-20-1}, we obtain
\begin{align}\label{2026-3-23-1}
	&\lambda^*\kappa_3\tau\sum_{n=0}^{l}
	(d_t(d_t\Div(\Theta_\mathbf{u}^{n+1}+\Theta_\mathbf{u}^{n})),(\Theta_\xi^{n+1}+\Theta_\xi^{n}))
	\nonumber\\
	&= \lambda^*\kappa_3
	(d_t\Div(\Theta_\mathbf{u}^{l+1}+\Theta_\mathbf{u}^{l}),(\Theta_\xi^{l+1}+\Theta_\xi^{l}))
	-\lambda^*\kappa_3\tau\sum_{n=0}^{l}
	(d_t\Div(\Theta_\mathbf{u}^{n}+\Theta_\mathbf{u}^{n-1}),d_t(\Theta_\xi^{n+1}+\Theta_\xi^{n}))\nonumber\\
	&=\lambda^*\kappa_3
	\left[(\varepsilon(\Theta_{\mathbf{u}}^{l+1}+\Theta_{\mathbf{u}}^{l}),d_t\varepsilon(\Theta_{\mathbf{u}}^{l+1}+\Theta_{\mathbf{u}}^{l}))
	-(\Lambda_\xi^{l+1}+\Lambda_\xi^{l},d_t\Div(\Theta_{\mathbf{u}}^{l+1}+\Theta_{\mathbf{u}}^{l}))
	\right]\nonumber\\
	&\quad +\lambda^*\kappa_3\tau\sum_{n=0}^{l}
	\left[	
	(d_t(\Lambda_\xi^{n+1}+\Lambda_\xi^{n}),d_t\Div(\Theta_{\mathbf{u}}^{n}+\Theta_{\mathbf{u}}^{n-1}))
	-(d_t\varepsilon(\Theta_{\mathbf{u}}^{n+1}+\Theta_{\mathbf{u}}^{n}),d_t\varepsilon(\Theta_{\mathbf{u}}^{n+1}+\Theta_{\mathbf{u}}^{n}))
	\right]\nonumber\\
	&\quad+\lambda^*\kappa_3\tau^2\sum_{n=0}^{l}
	(d_t\varepsilon(\Theta_{\mathbf{u}}^{n+1}+\Theta_{\mathbf{u}}^{n}),d_t^2\varepsilon(\Theta_{\mathbf{u}}^{n+1}+\Theta_{\mathbf{u}}^{n})).
\end{align}
Using \eqref{2025-10-20-1} and the properties of projection operators, we have
\begin{align}\label{2026-3-23-2}
	&\tau\sum_{n=0}^{l}(
	d_t\Div(\Theta_\mathbf{u}^{n+1}+\Theta_\mathbf{u}^{n}),(\Theta_\xi^{n+1}+\Theta_\xi^{n}))\nonumber\\
	&=\tau\sum_{n=0}^{l}
	\left[
	(\varepsilon(\Theta_{\mathbf{u}}^{n+1}+\Theta_{\mathbf{u}}^{n}),d_t\varepsilon(\Theta_{\mathbf{u}}^{n+1}+\Theta_{\mathbf{u}}^{n}))
	-(\Lambda_\xi^{n+1}+\Lambda_\xi^{n}, d_t\Div(\Theta_\mathbf{u}^{n+1}+\Theta_\mathbf{u}^{n}))
	\right]
\end{align}
From \eqref{2026-3-23-1}-\eqref{2026-3-23-2}, we can see that the left-hand side terms in equation \eqref{2025-11-7-1} can be transformed into a sum of positive norms, while the negative norms can be controlled by moving to the right-hand side.

Substituting the estimates of $\Phi_1,~\Phi_2,\cdots,\Phi_{12}$ into \eqref{2025-11-7-1} in Theorem \ref{th-2025-11-7-1}, combining like terms and using the discrete Gronwall's inequality, we obtain
\begin{align}\label{2025-11-15-1}
	&2\mu\lambda^*\kappa_3\|\varepsilon(\Theta_\mathbf{u}^{l+1}+\Theta_\mathbf{u}^{l})\|_0^2
	+\frac{\kappa_3}{2}\|\Theta_\xi^{l+1}+\Theta_\xi^l\|_0^2
	+\frac{\kappa_{2}}{2}
	\|\Theta_\eta^{l+1}+\Theta_\eta^{l}\|_0^2
	+\frac{\tau}{\mu_f}\sum_{n=0}^{l}\left(K\nabla\hat{Z}_p^{n+\frac{1}{2}},\nabla\hat{Z}_p^{n+\frac{1}{2}}\right)\nonumber\\
	&\leq 
	C\tau^4(\|\nabla\Div \mathbf{u}_{tt}\|^2_{L^2(0,T;L^2(\Omega))}
	+\|\xi_{ttt}\|^2_{L^\infty(0,T;L^2(\Omega))}
	+\|\eta_{ttt}\|^2_{L^\infty(0,T;L^2(\Omega))}
	+\|\eta_{ttt}\|^2_{L^2(0,T;L^2(\Omega))})
	\nonumber\\
	&\quad+Ch^{2(r+1)}(\|\Div\mathbf{u}_{tt}\|^2_{L^2(0,T;H^{r+1}(\Omega))}
	+\|\Div\mathbf{u}_{t}\|^2_{L^2(0,T;H^{r+1}(\Omega))})\nonumber\\
	&\quad+Ch^{2(r+1)}(\|\xi_t\|^2_{L^2(0,T;H^{r+1}(\Omega))}
	+\|\eta_t\|^2_{L^2(0,T;H^{r+1}(\Omega))}).
\end{align}

\begin{theorem}
	Suppose $\left\{(\mathbf{u}_h^n,\xi_h^n,\eta_h^n)\right\}_{n\geq0}$ is the numerical solution of Algorithm \ref{algorithm-2025-10-15}. Then the following error estimates hold:
	\begin{align}
		\max_{0\leq n\leq N}[
		\sqrt{\mu\lambda^*\kappa_3}\|\varepsilon(\mathbf{u}(t_n)-\mathbf{u}_h^{n})\|_0
		&+\sqrt{\kappa_3}
		\|\xi(t_n)-\xi_h^{n}\|_0\nonumber\\
		 +\sqrt{\kappa_2}\|\eta(t_n)-\eta_h^{n}\|_0]
		&\leq
		\hat{C}_1(T)\tau^2+\hat{C}_2(T)h^{r+1},
		\label{2025-11-15-2}\\
		\left[\frac{\tau}{\mu_f}\sum_{n=0}^{N}
		\|p(t_n)-p_h^n\|_0^2
		\right]^\frac{1}{2}
		&\leq \hat{C}_1(T)\tau^2+\hat{C}_2(T)h^{r+1},
		\label{2025-11-15-6}\\
		\left[\frac{\tau}{\mu_f}\sum_{n=0}^{N}
		\|\nabla (p(t_n)-p_h^n)\|_0^2
		\right]^\frac{1}{2}
		&\leq \hat{C}_1(T)\tau^2+\hat{C}_2(T)h^{r},\label{2025-11-15-3}		
	\end{align}	
	where
	\begin{align*}
		\hat{C}_1(T)&=C\left(
		\|\nabla\Div\mathbf{u}_{ttt}\|_{L^2(0,T;L^2(\Omega))}
		+\|\xi_{ttt}\|_{L^\infty(0,T;L^2(\Omega))}
		\right.\\
		&\qquad~\left.
		+\|\eta_{ttt}\|_{L^{\infty}(0,T;L^2(\Omega))}
		+\|\eta_{ttt}\|_{L^2(0,T;L^2(\Omega))}\right),\\
		\hat{C}_2(T)&=C\left(
		\|\Div\mathbf{u}_{tt}\|_{L^2(0,T;H^{r+1}(\Omega))}
		+\|\Div\mathbf{u}_t\|_{L^2(0,T;H^{r+1}(\Omega))}
		\right.	\\
		&\qquad~
		+\|\Div\mathbf{u}\|_{L^2(0,T;H^{r+1}(\Omega))}
		+\|\xi_t\|_{L^{2}(0,T;H^{r+1}(\Omega))}
		\\
		&\qquad~\left.
		+\|\xi\|_{L^{2}(0,T;H^{r+1}(\Omega))}
		+\|\eta\|_{L^{2}(0,T;H^{r+1}(\Omega))}
		\right.\\
		&\qquad~\left.	+\|\eta_t\|_{L^{2}(0,T;H^{r+1}(\Omega))}	
		\right)
	\end{align*}
\end{theorem}

\begin{proof}
	According to \eqref{2025-11-15-1} and the triangle inequality
	\begin{align*}
		\mathbf u(t_{n})-\mathbf u_{h}^{n}=\Lambda_{\mathbf u}^{n}+\Theta_{\mathbf u}^{n},\quad\xi(t_{n})-\xi_{h}^{n}=\Lambda_{\xi}^{n}+\Theta_{\xi}^{n},\quad
		\eta(t_n)-\eta_h^n=\Lambda_\eta^n+\Theta_\eta^n,
	\end{align*}
	using mathematical induction and the boundedness of initial values, we obtain \eqref{2025-11-15-2}.
	
	The relationship between $\hat{Z}_p^{n+\frac{1}{2}}$ and $\Theta_p^{n+\frac{1}{2}}$ is
	\begin{align*}
		\hat{Z}_p^{n+\frac{1}{2}}
		&=\Theta_p^{n+\frac{1}{2}}+\lambda^*\kappa_1d_t(\Div\Lambda_{\mathbf{u}}^{n+1}-\hat{\Lambda}_q^{n+1})\\
		&\quad+\frac{\kappa_{1}}{2}(\Lambda_\xi^{n+1}-\hat{\Lambda}_\xi^{n+1}+\Lambda_\xi^{n}-\hat{\Lambda}_\xi^{n})
		+\frac{\kappa_{2}}{2}(\Lambda_\eta^{n+1}-\hat{\Lambda}_\eta^{n+1}+\Lambda_\eta^{n}-\hat{\Lambda}_\eta^{n}).
	\end{align*}
	Using the definition of projection operators, we have
	\begin{align}\label{2025-11-15-4}
		\frac{\tau}{\mu_f}\sum_{n=0}^{l}
		\|\nabla \Theta_p^{n+\frac{1}{2}}\|_0^2
		&\leq\frac{\tau}{\mu_f}\sum_{n=0}^{l}
		\|\nabla \hat{Z}_p^{n+\frac{1}{2}}\|_0^2
		+\frac{\tau}{\mu_f}\sum_{n=0}^{l}
		\|\lambda^*\kappa_1d_t\nabla(\Div\Lambda_{\mathbf{u}}^{n+1}-\hat{\Lambda}_q^{n+1})\|_0^2
		\nonumber\\
		&\quad
		+\frac{\tau}{\mu_f}\sum_{n=0}^{l}
		\|\frac{\kappa_{1}}{2}\nabla(\Lambda_\xi^{n+1}-\hat{\Lambda}_\xi^{n+1}+\Lambda_\xi^{n}-\hat{\Lambda}_\xi^{n})\|_0^2	\nonumber\\
		&\quad+\frac{\tau}{\mu_f}\sum_{n=0}^{l}
		\|\frac{\kappa_{2}}{2}\nabla(\Lambda_\eta^{n+1}-\hat{\Lambda}_\eta^{n+1}+\Lambda_\eta^{n}-\hat{\Lambda}_\eta^{n})\|_0^2
		\nonumber\\
		&\leq C_1(T)\tau^4+C_2(T)h^{2r+2}
		+Ch^{2r}\|\Div\mathbf{u}\|_{L^{2}(0,T;H^{r+1}(\Omega))}\nonumber\\
		&\quad+Ch^{2r}\|\xi\|_{L^{2}(0,T;H^{r+1}(\Omega))}
		+Ch^{2r}\|\eta\|_{L^{2}(0,T;H^{r+1}(\Omega))}.
	\end{align}
	Similarly, we have
	\begin{align}\label{2025-11-15-7}
		\frac{\tau}{\mu_f}\sum_{n=0}^{l}
		\|\Theta_p^{n+\frac{1}{2}}\|_0^2
		&\leq C_1(T)\tau^4+C_2(T)h^{2r+2}
		+Ch^{2r+2}\|\Div\mathbf{u}\|_{L^{2}(0,T;H^{r+1}(\Omega))}\nonumber\\
		&\quad+Ch^{2r+2}\|\xi\|_{L^{2}(0,T;H^{r+1}(\Omega))}
		+Ch^{2r+2}\|\eta\|_{L^{2}(0,T;H^{r+1}(\Omega))}.
	\end{align}
	According to the identity $p(t_n)-p_h^n=\Lambda_p^n+\Theta_p^n$, equations \eqref{2025-11-15-4}-\eqref{2025-11-15-7}, and the properties of projection operators $\mathcal{Q}_h$, $\mathcal{R}_h$, we obtain \eqref{2025-11-15-6} and \eqref{2025-11-15-3}. The proof is complete.
\end{proof}

\section{Numerical Examples}\label{2026-3-22-4}

In this section, we present several numerical examples to verify the theoretical results. The gravity term $\mathbf{g}$ is neglected. We define the maximum norm error over all time steps as
\begin{align*}
	R(h, \tau) = \max_{0 \leq n \leq N} \|\mathbf{u}^n - \mathbf{u}_h^n\|,
\end{align*}
where $\mathbf{u}^n$ and $\mathbf{u}_h^n$ denote the exact solution and numerical solution at time node $t_n$, respectively. 
Moreover, the observed convergence orders with respect to the spatial variable and temporal variable are measured by the following formulas:
\begin{align*}
	{\rm{order_{H}}}=\frac{\log(R(h,\tau)/R(\frac12h,\tau))}{\log2},
	\qquad	{\rm{order_{T}}}=\frac{\log(R(h,\tau)/R(h,\frac12\tau))}{\log2},
\end{align*}
where $h$ and $\tau$ represent the spatial mesh size and time step size, respectively. 
In the following examples, Taylor-Hood finite elements are employed for the numerical simulations.

{\bf Example 4.1~} 
Let $\Omega= (0,1)\times(0,1)$, $T=1$, $\Gamma_{1}= \{ (0,y);~0\leq y\leq1 \}$, $\Gamma_{2}= \{(1,y);~0\leq y\leq1 \}$, $\Gamma_{3}= \{(x,1);~0\leq x\leq1 \}$, $\Gamma_{4} = \{(x,0);~0\leq x\leq1\}$.\\
The body force $\mathbf{f}=[f_1,f_2]^{T}$ and the source term $\phi$ are given by
\begin{align*}
	 f_1(x,y,t) &= t^3\left[ \pi^2\left(\frac{3\mu}{2}+\lambda\right)\sin(\pi x)\sin(\pi y)
	  - 2\alpha\pi\sin(2\pi x)\cos(2\pi y) \right] \nonumber \\ 
	  &\quad + t^2\left[ 3\pi^2\lambda^*\sin(\pi x)\sin(\pi y) - 3\lambda^*(2x-1)(2y-1) \right]\\
	  &\quad-t^3\left(\frac{\mu}{2}+\lambda\right)(2x-1)(2y-1), 
\end{align*}
\begin{align*}
	  f_2(x,y,t) &= t^3\left[ -\pi^2\left(\frac{\mu}{2}+\lambda\right)\cos(\pi x)\cos(\pi y)  - 2\alpha\pi\cos(2\pi x)\sin(2\pi y) \right] \nonumber \\ 
	  &\quad + t^2\left[ -3\pi^2\lambda^*\cos(\pi x)\cos(\pi y) - 6\lambda^* x(x-1) \right]\\
	  &\quad -2t^3(\mu+\lambda)x(x-1) - \mu y(y-1),
\end{align*}
and
\begin{align*}
	\phi(x,y,t) &= 3t^2\left[ c_0 \cos(2\pi x)\cos(2\pi y) + \alpha\pi \cos(\pi x)\sin(\pi y) + \alpha x(x-1)(2y-1) \right] \\
	&- \frac{8\pi^2 K}{\mu_f} t^3 \cos(2\pi x)\cos(2\pi y).
\end{align*}
The initial and boundary conditions are prescribed as
\begin{align*}
	p(x,y,t) = t^3 \cos(2\pi x) \cos(2\pi y) &\qquad\mbox{on }\partial\Omega_T,\\
	u_1(x,y,t)= t^{3}\sin(\pi x)\sin(\pi y) &\qquad\mbox{on }\Gamma_j\times (0,T),\, j=2,4,\\
	u_2(x,y,t) = t^3 \, x(x-1) \, y(y-1) &\qquad\mbox{on }\Gamma_j\times (0,T),\, j=2,4,\\
	\sigma(\mathbf{u})\mathbf{n}-\alpha p\mathbf{n} = \mathbf{f}_1 &\qquad \mbox{on } \partial\Omega_T\backslash\Gamma_j,\\
	\mathbf{u}(x,y,0) = \mathbf{0}, \quad p(x,y,0) =0 &\qquad\mbox{in } \Omega.
\end{align*}
It is easy to verify that the exact solution is
\begin{align*}
	&\mathbf{u}(x,y,t)=
	\begin{pmatrix}
		u_1(x,y,t)\\
		u_2(x,y,t)
	\end{pmatrix}
	=
	\begin{pmatrix}
		t^3 \sin(\pi x) \sin(\pi y)\\
		t^3 \, x(x-1) \, y(y-1)
	\end{pmatrix},\\
	&p(x,y,t) =t^3 \cos(2\pi x) \cos(2\pi y).
\end{align*}
The boundary term $\mathbf{f}_1$ can be obtained from the exact solution. The values of physical parameters are listed in Table \ref{cn_tab_time}. Other physical quantities $\lambda$, $\mu$, $\kappa_1$, $\kappa_2$, $\kappa_3$ can be computed via \eqref{chap2-4-25-11} and \eqref{chapt4-4-25-2}.
\begin{table}[!h]
	\centering
	\caption{Values of Physical Parameters}\label{cn_tab_time}
	\begin{tabularx}{\textwidth}{c >{\centering\arraybackslash}X c}
		\toprule
		Parameter 		   & Description  & Value  \\ \midrule
		$\lambda^*$ & Secondary consolidation coefficient 	& 1e-6   \\
		$E$ 		& Young's modulus 		& 1e7\\
		$\nu$ 		& Poisson's ratio 			& 0.4\\
		$\alpha$	& Biot-Willis constant & 0.5\\
		$c_0$		& Constrained specific storage coefficient 	& 0.5\\
		$K$ 		& Permeability tensor 		& 1e-9\\
		$\mu_f$ 	& Fluid viscosity 			& 1\\
		\bottomrule
	\end{tabularx}
\end{table}
\begin{table}[!h]
	\centering
	\caption{Spatial error and corresponding convergence order for displacement $\mathbf{u}$ when $\tau=0.0001$}
	\label{cn_tab_space_convergence_u}
	\begin{tabularx}{\textwidth}{CCCCC}
		\toprule
		$h$ & $\|u-u_h\|_0$ & $L^2$-order & $\|u-u_h\|_1$ & $H^1$-order \\
		\midrule
		1/8 & 6.7991e-04 & - & 4.5462e-02 & - \\
		1/16 & 6.6777e-05 & 3.348 & 9.5779e-03 & 2.247 \\
		1/32 & 7.6058e-06 & 3.134 & 2.2252e-03 & 2.106 \\
		1/64 & 9.1896e-07 & 3.049 & 5.4040e-04 & 2.042 \\
		\bottomrule
	\end{tabularx}
\end{table}
\begin{table}[!h]
	\centering
	\caption{Spatial error and corresponding convergence order for pressure $p$ when $\tau=0.0001$}
	\label{cn_tab_space_convergence_p}
	\begin{tabularx}{\textwidth}{CCCCC}
		\toprule
		$h$ & $\|p-p_h\|_0$ & $L^2$-order & $\|p-p_h\|_1$ & $H^1$-order \\
		\midrule
		1/8 & 1.6934e-02 & - & 1.8340e+00 & - \\
		1/16 & 3.2550e-03 & 2.379 & 8.8519e-01 & 1.051 \\
		1/32 & 7.4333e-04 & 2.131 & 4.3785e-01 & 1.016 \\
		1/64 & 1.8118e-04 & 2.037 & 2.1830e-01 & 1.004 \\
		\bottomrule
	\end{tabularx}
\end{table}
\begin{table}[!h]
	\centering
	\caption{Temporal error and corresponding convergence order for displacement $\mathbf{u}$ when $h=\frac{1}{512}$}
	\label{cn_tab_time_convergence_u}
	\begin{tabularx}{\textwidth}{CCCCC}
	\toprule
	$\tau$ & $\|u-u_h\|_0$ & $L^2$-order & $\|u-u_h\|_1$ & $H^1$-order \\
	\midrule
	1/2 & 7.3784e-01 & - & 1.9577e+00 & - \\
	1/4 & 1.8446e-01 & 2.000 & 4.8944e-01 & 2.000 \\
	1/8 & 4.6115e-02 & 2.000 & 1.2236e-01 & 2.000 \\
	1/16& 1.1528e-02 & 2.000 & 3.0591e-02 & 2.000 \\
	\bottomrule
	\end{tabularx}
\end{table}
\begin{table}[!h]
	\centering
	\caption{Temporal error and corresponding convergence order for pressure $p$ when $h=\frac{1}{512}$}
	\label{cn_tab_time_convergence_p}
	\begin{tabularx}{\textwidth}{CCCCC}
	\toprule
	$\tau$ & $\|p-p_h\|_0$ & $L^2$-order & $\|p-p_h\|_1$ & $H^1$-order \\
	\midrule
	1/2 & 3.0935e-01 & - & 1.1389e+00 & - \\
	1/4 & 7.7337e-02 & 2.000 & 2.8599e-01 & 1.994 \\
	1/8 & 1.9334e-02 & 2.000 & 7.6258e-02 & 1.907 \\
	1/16& 4.8336e-03 & 2.000 & 3.2586e-02 & 1.227 \\
	\bottomrule
	\end{tabularx}
\end{table}
\begin{figure}[!h]
	\centering
	\begin{subfigure}[t]{0.4\linewidth}
		\centering
		\includegraphics[width=\linewidth]{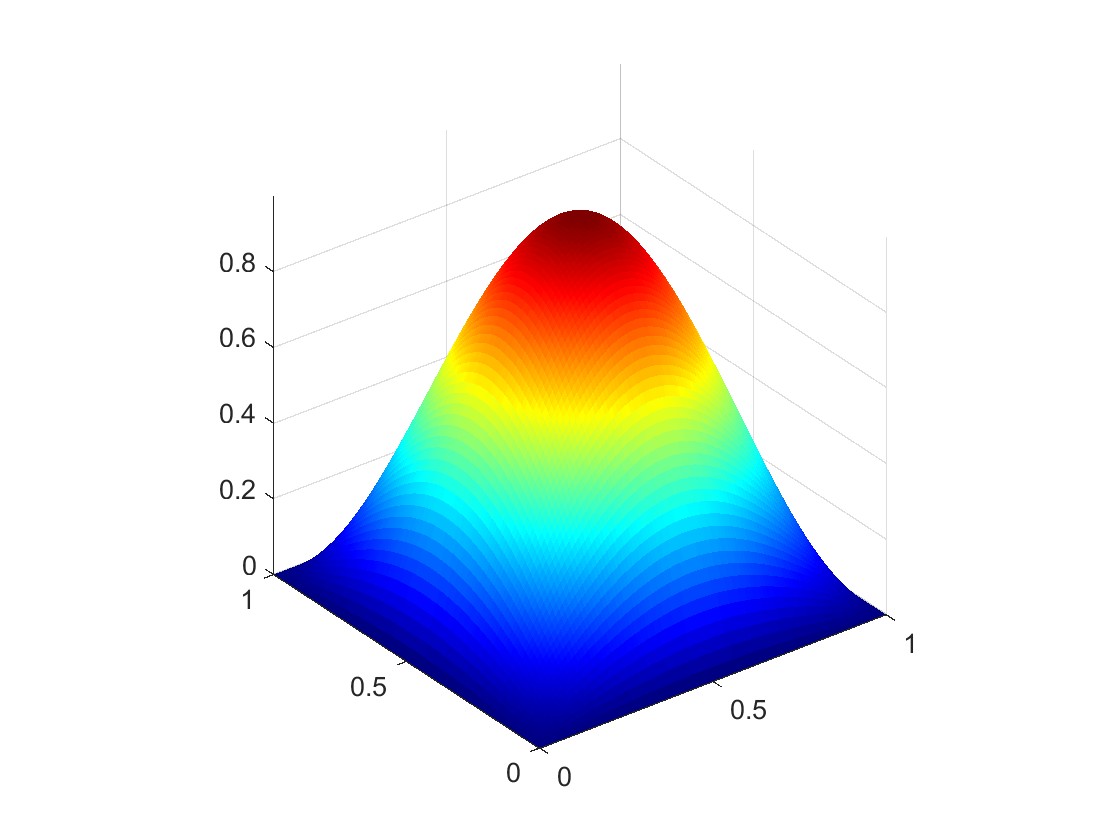}
		\caption{Exact solution}
		\label{chapt3_P2P1_test_u1}
	\end{subfigure}
	\hspace{0.05\linewidth}
	\begin{subfigure}[t]{0.4\linewidth}
		\centering
		\includegraphics[width=\linewidth]{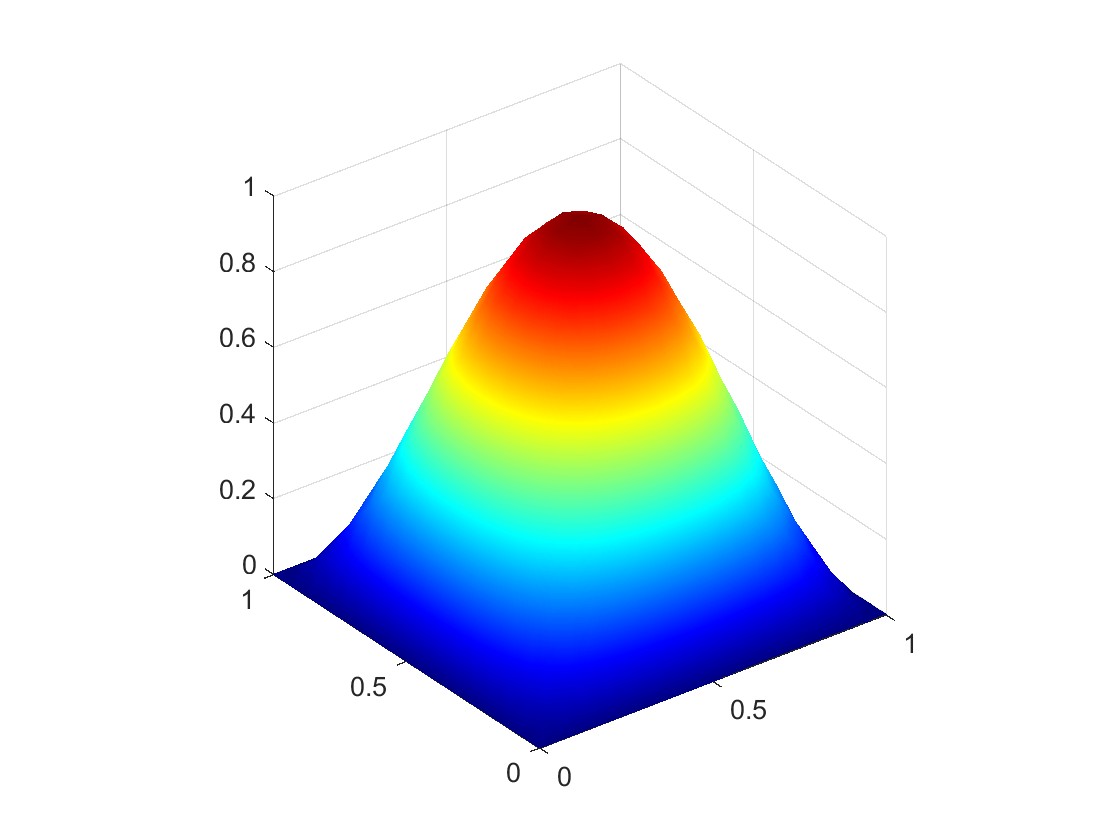}
		\caption{Numerical solution}
		\label{chapt3_P2P1_test_u1_h}
	\end{subfigure}
	\caption{Surface plots of displacement component $u_1$ at the final time $T$.}
	\label{chapt3_P2P1_test_u1_fig}
\end{figure}
\begin{figure}[!h]
	\centering
	\begin{subfigure}[t]{0.4\linewidth}
		\centering
		\includegraphics[width=\linewidth]{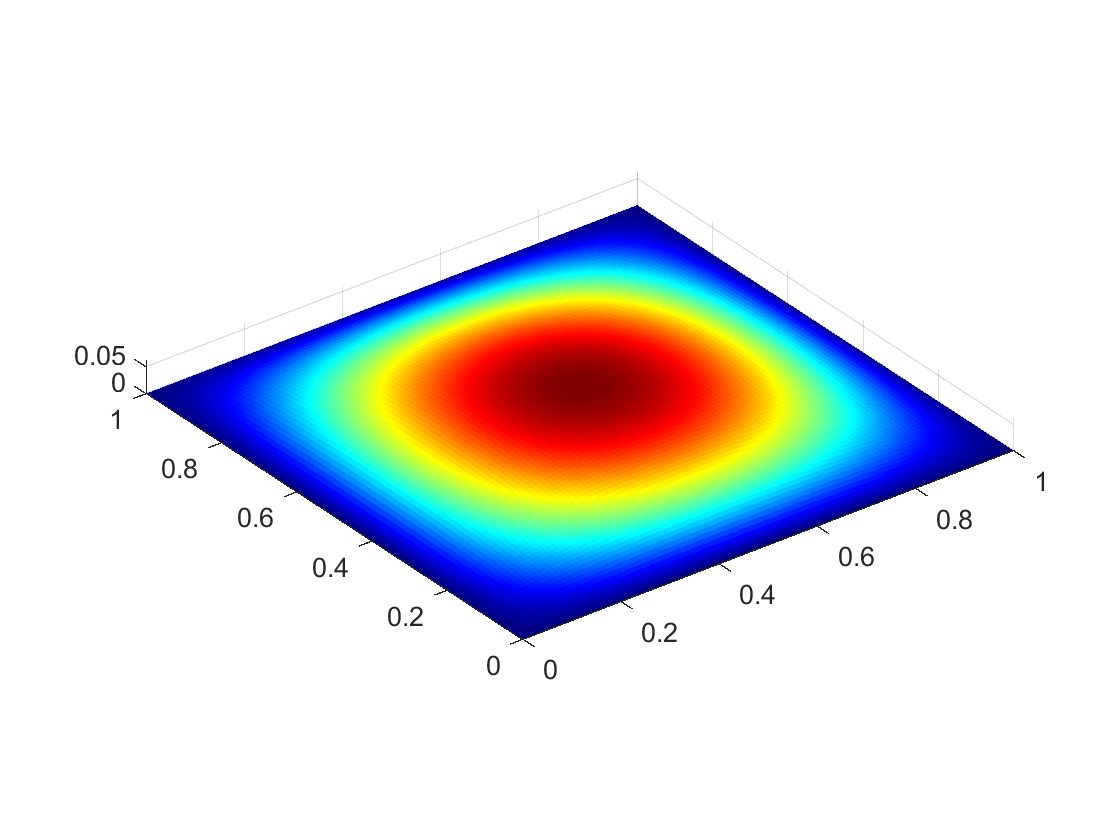}
		\caption{Exact solution}
		\label{chapt3_P2P1_test_u2}
	\end{subfigure}
	\hspace{0.05\linewidth}
	\begin{subfigure}[t]{0.4\linewidth}
		\centering
		\includegraphics[width=\linewidth]{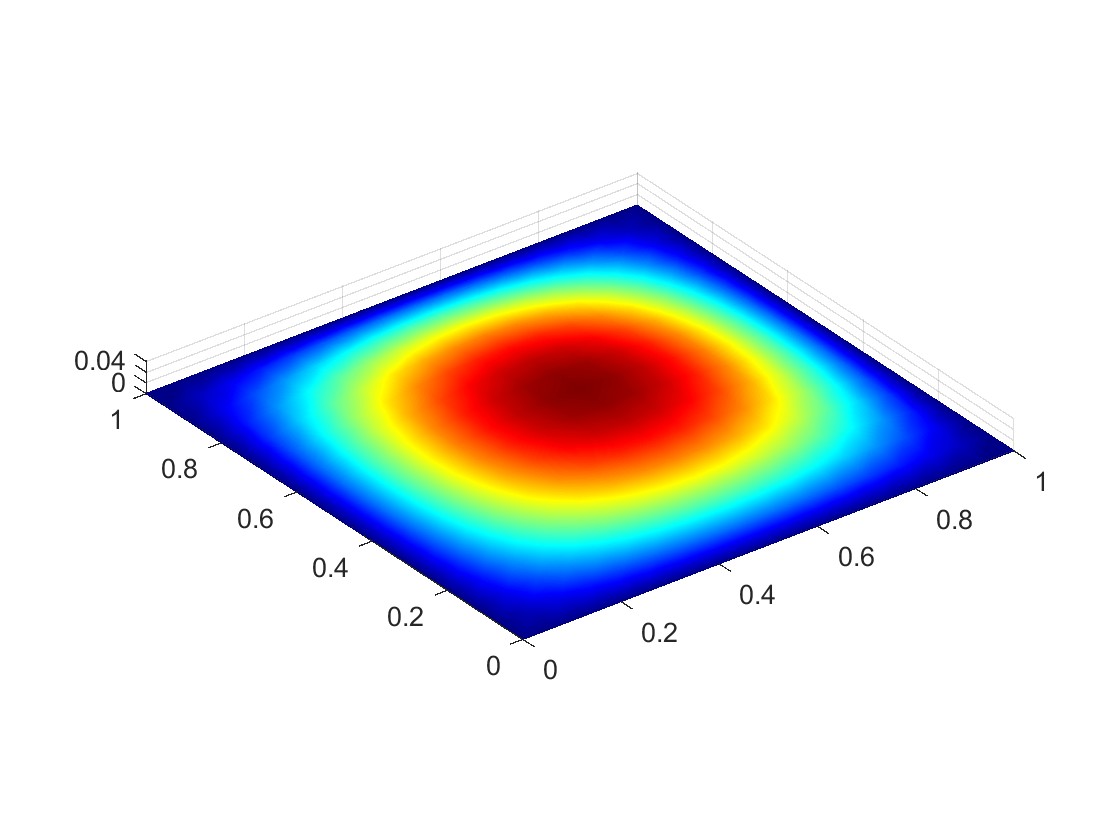}
		\caption{Numerical solution}
		\label{chapt3_P2P1_test_u2_h}
	\end{subfigure}
	\caption{Surface plots of displacement component $u_2$ at the final time $T$.}
	\label{chapt3_P2P1_test_u2_fig}
\end{figure}
\begin{figure}[!h]
	\centering
	\begin{subfigure}[t]{0.4\linewidth}
		\centering
		\includegraphics[width=\linewidth]{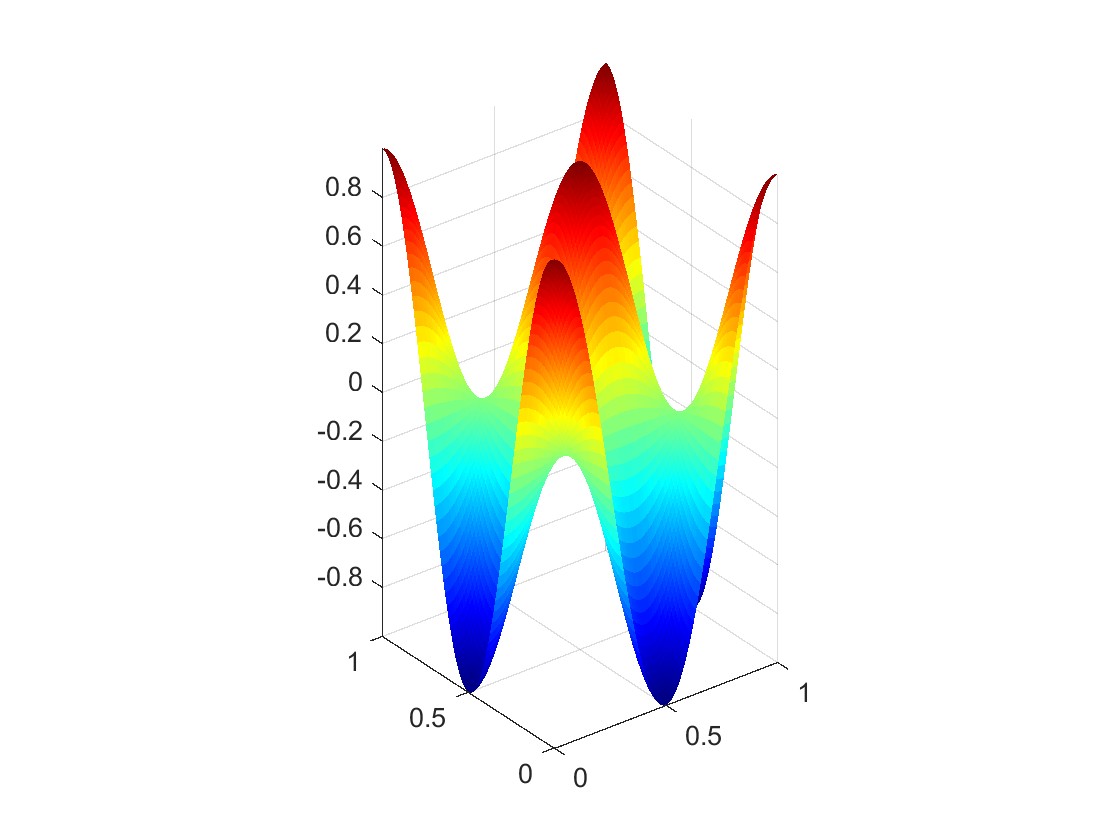}
		\caption{Exact solution}
		\label{chapt3_P2P1_test_p}
	\end{subfigure}
	\hspace{0.05\linewidth}
	\begin{subfigure}[t]{0.4\linewidth}
		\centering
		\includegraphics[width=\linewidth]{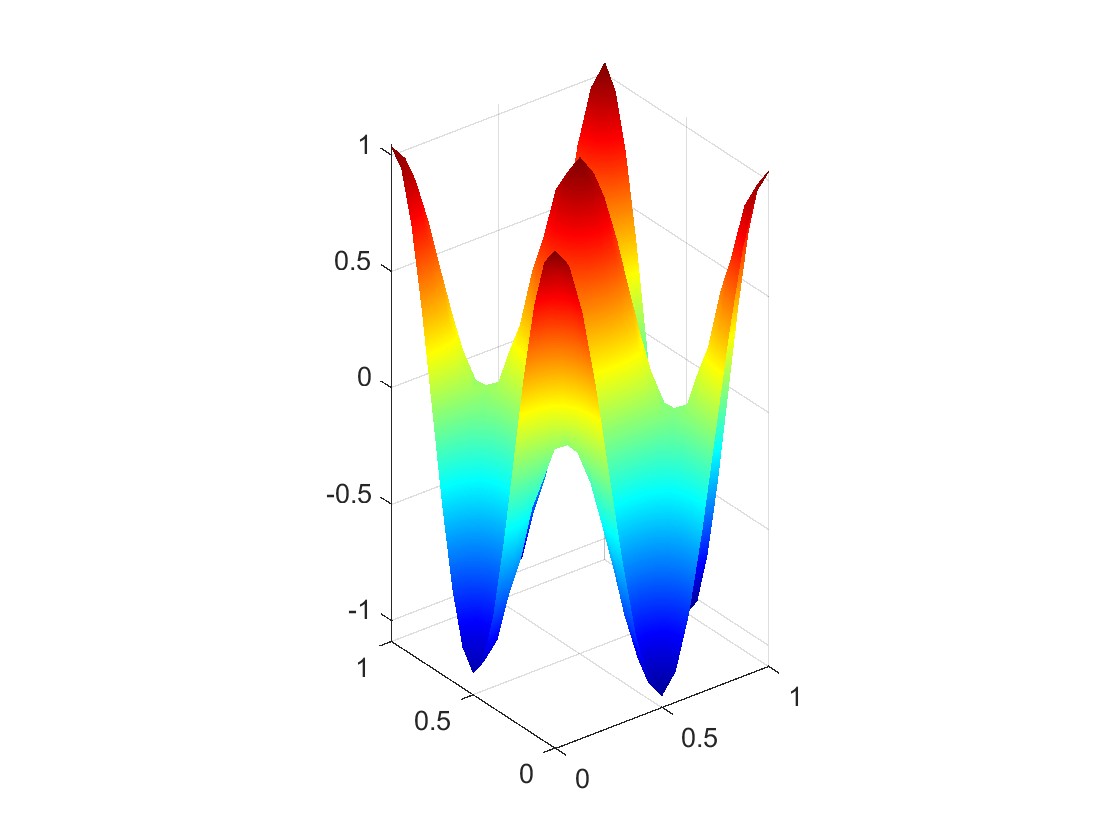}
		\caption{Numerical solution}
		\label{chapt3_P2P1_test_ph}
	\end{subfigure}
	\caption{Surface plots of pressure $p$ at the final time $T$.}
	\label{chapt3_P2P1_test_p_fig}
\end{figure}
Tables \ref{cn_tab_space_convergence_u} and \ref{cn_tab_space_convergence_p} present the numerical simulation results for displacement $\mathbf{u}$ and pressure $p$ in the spatial direction, respectively. The results show that the convergence orders of displacement $\mathbf{u}$ in the $L^2$ and $H^1$ norms are approximately $3$ and $2$, respectively; the convergence orders of pressure $p$ in the $L^2$ and $H^1$ norms are approximately $2$ and $1$, respectively. These results are consistent with the theoretical analysis.

Tables \ref{cn_tab_time_convergence_u} and \ref{cn_tab_time_convergence_p}
present the numerical simulation results for displacement $\mathbf{u}$ and pressure $p$ in the temporal direction, respectively. The results show that the convergence orders of both displacement $\mathbf{u}$ and pressure $p$ in the temporal direction are approximately $2$, which is consistent with the theoretical analysis.

Figures \ref{chapt3_P2P1_test_u1} through \ref{chapt3_P2P1_test_p_fig} display the surface plots of displacement components $u_1$, $u_2$, and pressure $p$ at the final time $T$, computed with mesh parameter $h=\frac{1}{16}$ and time step $\tau=h^2$. Panel (a) shows the exact solution (plotted on a $100 \times 100$ grid), and panel (b) shows the numerical solution obtained using the time-nonlocal multiphysics finite element method. The good agreement between them verifies the effectiveness of the numerical method.

{\bf Example 4.2~} 
This example mainly compares the error results between the Crank-Nicolson scheme and the backward Euler scheme.
Let $\Omega= (0,1)\times(0,1)$, $T=2$, $\Gamma_{1}= \{ (0,y);~0\leq y\leq1 \}$, $\Gamma_{2}= \{(1,y);~0\leq y\leq1 \}$, $\Gamma_{3}= \{(x,1);~0\leq x\leq1 \}$, $\Gamma_{4} = \{(x,0);~0\leq x\leq1\}$.\\
The body force $\mathbf{f}=[f_1,f_2]^{T}$ and the source term $\phi$ are given by
\begin{align*}
	f_1(x,y,t) &= e^t\Bigg\{
	-\mu\Bigl[2y(y-1) + x(x-1) + \tfrac12(2x-1)(2y-1)\Bigr] \\
	&- (\lambda+\lambda^*)\Bigl[2y(y-1) + (2x-1)(2y-1)\Bigr] 
	- 2\pi\alpha\,\sin(2\pi x)\cos(2\pi y)
	\Bigg\},
\end{align*}
\begin{align*}
	f_2(x,y,t) &= e^t\Bigg\{
	-\mu\Bigl[y(y-1) + 2x(x-1) + \tfrac12(2x-1)(2y-1)\Bigr] \\
	&- (\lambda+\lambda^*)\Bigl[(2x-1)(2y-1) + 2x(x-1)\Bigr] 
	- 2\pi\alpha\,\cos(2\pi x)\sin(2\pi y)
	\Bigg\},
\end{align*}
and
\begin{align*}
	\phi(x,y,t) &= e^t \Bigg\{
	\left(c_0 + \frac{8\pi^2 K}{\mu_f}\right) \cos(2\pi x)\cos(2\pi y) + \alpha\Big[(2x-1)y(y-1) + x(x-1)(2y-1)\Big]
	\Bigg\}.
\end{align*}
The initial and boundary conditions are prescribed as
\begin{align*}
	p(x,y,t) = e^{t} \cos(2\pi x) \cos(2\pi y) &\qquad\text{on~}\partial\Omega_T,\\
	u_1(x,y,t)= e^{t} \, x(x-1) \, y(y-1) &\qquad\text{on~}\Gamma_j\times (0,T),\, j=2,4,\\
	u_2(x,y,t) = e^{t} \, x(x-1) \, y(y-1) &\qquad\text{on~}\Gamma_j\times (0,T),\, j=2,4,\\
	\sigma(\mathbf{u})\mathbf{n}-\alpha p\mathbf{n} = \mathbf{f}_1 &\qquad \text{on~} \partial\Omega_T\backslash\Gamma_j,\\
	\mathbf{u}(x,y,0) = \mathbf{0}, \quad p(x,y,0) =0 &\qquad\text{in~} \Omega.
\end{align*}
It is easy to verify that the exact solution is
\begin{align*}
	&\mathbf{u}(x,y,t)=
	\begin{pmatrix}
		u_1(x,y,t)\\
		u_2(x,y,t)
	\end{pmatrix}
	=
	\begin{pmatrix}
		e^{t} \, x(x-1) \, y(y-1)\\
		e^{t} \, x(x-1) \, y(y-1)
	\end{pmatrix},\\
	&p(x,y,t) =e^{t} \cos(2\pi x) \cos(2\pi y).
\end{align*}
The boundary term $\mathbf{f}_1$ can be obtained from the exact solution. The values of physical parameters are listed in Table \ref{cn_tab_test2}. Other physical quantities $\lambda$, $\mu$, $\kappa_1$, $\kappa_2$, $\kappa_3$ can be computed via \eqref{chap2-4-25-11} and \eqref{chapt4-4-25-2}.
\begin{table}[!h]
	\centering
	\caption{Values of Physical Parameters}\label{cn_tab_test2}
	\begin{tabularx}{\textwidth}{c >{\centering\arraybackslash}X c}
		\toprule
		Parameter 		   & Description  & Value  \\ \midrule
		$\lambda^*$ & Secondary consolidation coefficient 	& 1e-6   \\
		$E$ 		& Young's modulus 		& 1e9\\
		$\nu$ 		& Poisson's ratio 			& 0.4\\
		$\alpha$	& Biot-Willis constant & 0.5\\
		$c_0$		& Constrained specific storage coefficient 	& 0.5\\
		$K$ 		& Permeability tensor 		& 1e-9\\
		$\mu_f$ 	& Fluid viscosity 			& 1\\
		\bottomrule
	\end{tabularx}
\end{table}
\begin{table}[!h]
	\centering
	\caption{Spatial error and corresponding convergence order for displacement $\mathbf{u}$ when $\tau=0.001$ and $T=2$}
	\label{cn_test2_space_convergence_u}
	\begin{tabularx}{\textwidth}{CCCCC}
		\toprule
		$h$ & $\|u-u_h\|_0$ & $L^2$-order & $\|u-u_h\|_1$ & $H^1$-order \\
		\midrule
	$1/8$   & 4.0961e-04 & -   & 2.8852e-02 & -   \\
	$1/16$  & 4.2300e-05 & 3.276 & 6.3881e-03 & 2.175 \\
	$1/32$  & 4.7723e-06 & 3.148 & 1.4921e-03 & 2.098 \\
	$1/64$  & 5.6994e-07 & 3.066 & 3.6000e-04 & 2.051\\
		\bottomrule
	\end{tabularx}
\end{table}
\begin{table}[!h]
	\centering
	\caption{Spatial error and corresponding convergence order for pressure $p$ when $\tau=0.001$ and $T=2$}
	\label{cn_test2_space_convergence_p}
	\begin{tabularx}{\textwidth}{CCCCC}
		\toprule
		$h$ & $\|p-p_h\|_0$ & $L^2$-order & $\|p-p_h\|_1$ & $H^1$-order \\
		\midrule
	$1/8$   & 1.3759e-01 & -   & 1.3366e+01 & -   \\
	$1/16$  & 2.8018e-02 & 2.296 & 6.5151e+00 & 1.037 \\
	$1/32$  & 6.5630e-03 & 2.094 & 3.2320e+00 & 1.011 \\
	$1/64$  & 1.6120e-03 & 2.026 & 1.6126e+00 & 1.003 \\
		\bottomrule
	\end{tabularx}
\end{table}

Tables \ref{cn_test2_space_convergence_u} and \ref{cn_test2_space_convergence_p} present the numerical simulation results for displacement $\mathbf{u}$ and pressure $p$ in the spatial direction, respectively. The results show that the convergence orders of displacement $\mathbf{u}$ in the $L^2$ and $H^1$ norms are approximately $3$ and $2$, respectively; the convergence orders of pressure $p$ in the $L^2$ and $H^1$ norms are approximately $2$ and $1$, respectively. These results are consistent with the theoretical analysis.

When the backward Euler scheme is employed for temporal discretization, we use the same physical parameter values as shown in Table \ref{cn_tab_test2} and the same time step size $\tau = 0.001$. The resulting errors and corresponding convergence orders are as follows:

\begin{table}[!h]
	\centering
	\caption{Spatial error and corresponding convergence order for displacement $\mathbf{u}$ when $\tau=0.001$ and $T=2$}
	\label{Euler_cn_tab_space_convergence_u}
	\begin{tabularx}{\textwidth}{CCCCC}
		\toprule
		$h$ & $\|u-u_h\|_0$ & $L^2$-order & $\|u-u_h\|_1$ & $H^1$-order \\
		\midrule
	$1/8$   & 4.5642e-04 & -   & 3.1374e-02 & -   \\
	$1/16$  & 4.5468e-05 & 3.327 & 6.7361e-03 & 2.220 \\
	$1/32$  & 4.9805e-06 & 3.190 & 1.5378e-03 & 2.131 \\
	$1/64$  & 5.7963e-07 & 3.103 & 3.6587e-04 & 2.071 \\
		\bottomrule
	\end{tabularx}
\end{table}
\begin{table}[!h]
	\centering
	\caption{Spatial error and corresponding convergence order for pressure $p$ when $\tau=0.001$ and $T=2$}
	\label{Euler_cn_tab_space_convergence_p}
	\begin{tabularx}{\textwidth}{CCCCC}
		\toprule
		$h$ & $\|p-p_h\|_0$ & $L^2$-order & $\|p-p_h\|_1$ & $H^1$-order \\
		\midrule
	$1/8$   & 1.4116e-01 & -   & 1.3373e+01 & -   \\
	$1/16$  & 2.7950e-02 & 2.336 & 6.5179e+00 & 1.037 \\
	$1/32$  & 6.3829e-03 & 2.131 & 3.2334e+00 & 1.011 \\
	$1/64$  & 4.0789e-03 & 0.646 & 1.6134e+00 & 1.003 \\
		\bottomrule
	\end{tabularx}
\end{table}

Tables \ref{Euler_cn_tab_space_convergence_u} and \ref{Euler_cn_tab_space_convergence_p} also present the numerical simulation results for displacement $\mathbf{u}$ and pressure $p$ in the spatial direction, respectively. The results show that Algorithm $\mathrm{\ref{algorithm-2025-10-15}}$ based on the backward Euler scheme, under exactly the same time step size and total simulation time, exhibits a sharp decline in the $L^2$ error convergence order of pressure to $0.646$ when $h=1/64$, showing obvious convergence order degradation. The absolute value of pressure error is also larger (the $L^2$ error of pressure at $h=1/64$ is $4.0789\times10^{-3}$, approximately $2.5$ times that of the Crank-Nicolson scheme). Theoretically, the backward Euler scheme is a first-order temporal discretization method with relatively large numerical dissipation, and the truncation error at each step accumulates continuously over time, leading to the decline in convergence order. To improve the computational results of the backward Euler scheme, one can reduce the time step size or shorten the total simulation time to mitigate the effects of dissipation accumulation.

Comparing the performance of the two schemes under the same time step size and total simulation time, the Crank-Nicolson scheme is significantly superior to the backward Euler scheme in both the convergence order and absolute error of the pressure field, and also shows slight advantages in the accuracy of the displacement field. As a second-order unconditionally stable scheme, the Crank-Nicolson scheme possesses time-reversal symmetry and lower numerical dissipation, enabling it to better maintain solution accuracy for long-time evolution problems. Therefore, for poroelasticity problems requiring long-time integration (such as $T=2$ and above), it is recommended to prioritize the Crank-Nicolson scheme for temporal discretization to obtain more reliable and accurate numerical results.

\section{Conclusion}\label{2026-4-1-1}

In this paper, we construct a time-nonlocal multiphysics finite element method based on the Crank-Nicolson scheme for the reformulated poroelasticity model. 

For the case where the physical parameters $\lambda,\lambda^*$ and $c_0$ are all finite positive constants, by introducing two auxiliary variables, the original strongly coupled model is reformulated into a generalized Stokes equation with time integral terms and a diffusion equation. 
Then a time-nonlocal multiphysics finite element method with the Crank-Nicolson scheme is proposed. In terms of discrete scheme design, the spatial discretization uses the high-order Taylor-Hood mixed finite element method: displacement $\mathbf{u}$ uses $r+1$-th order Lagrange elements, $\xi$ and $\eta$ use $r$-th order Lagrange elements, and the temporal discretization uses the Crank-Nicolson scheme. Systematic theoretical analysis is conducted in this paper, including the
stability analysis of the fully discrete scheme and the optimal-order error estimates. 
The theoretical results show that: the convergence order of displacement $\mathbf{u}$ in the $H^1$ norm is $O(\tau^2+h^{r+1})$, and the convergence orders of pressure $p$ in the $L^2$ and $H^1$ norms are $O(\tau^2+h^{r+1})$ and $O(\tau^2+h^r)$, respectively.

\bibliographystyle{unsrt}  
\bibliography{reference}  

\end{document}